\documentclass[a4paper,11pt]{amsart}

\numberwithin{equation}{section}

\usepackage{amsfonts}
\usepackage{amsthm}
\usepackage{amssymb}
\usepackage{amsmath}
\usepackage{stmaryrd} 
\usepackage{mathabx}
\usepackage{mathbbol}
\usepackage{xspace}
\usepackage{psfrag}

\usepackage{amssymb,amsthm, geometry}
\usepackage{enumerate}

\usepackage[matrix,arrow,tips,curve]{xy}

\UseRawInputEncoding

\newenvironment{sis}{\left\{\begin{aligned}}{\end{aligned}\right.}

\theoremstyle{plain}
\newtheorem{thm}{Theorem}[section]
\newtheorem{theoremalpha}{Theorem}

\newtheorem{lemma}[thm]{Lemma}
\newtheorem{prop}[thm]{Proposition}
\newtheorem{cor}[thm]{Corollary}
\newtheorem*{corspec}{Corollary F}

\newtheorem{fact}[thm]{Fact}

\newtheorem{lemdef}[thm]{Lemma-Definition}

\theoremstyle{definition}
\newtheorem{defi}[thm]{Definition}

\theoremstyle{remark}

\newtheorem{remark}[thm]{Remark}

\newtheorem{example}[thm]{Example}

\usepackage{layout}
\usepackage{fullpage}

\usepackage[usenames,dvipsnames]{xcolor}

\usepackage[backref=page]{hyperref}

\hypersetup{
 colorlinks,
 citecolor=Green,
 linkcolor=blue,
 urlcolor=Blue}

\newcommand{\wt}{\widetilde}

\newcommand{\wh}{\widehat}
\newcommand{\ov}{\overline}
\newcommand{\un}{\underline}

\newcommand{\val}{\operatorname{val}}

\newcommand{\m}{\mathfrak{m}}

\def\X{\mathcal X}

\def\L{\mathcal L}
\def\N{\mathcal N}
\def\M{\mathcal M}

\def\O{\mathcal O}
\def\bO{\mathbb O}
\def\bN{\mathbb N}
\def\D{\mathcal D}

\def\E{\mathcal E}
\def\P{\mathbb P}
\def\cA{\mathcal A}

\def\cP{\mathcal P}
\def\cC{\mathcal C}
\def\cE{\mathcal E}

\def \bX{\mathfrak X}

\def\NS{{\operatorname{N}} _{\XS}}

\def\int{M_X}

\newcommand{\Z}{\mathbb{Z}}
\newcommand{\Q}{\mathbb{Q}}
\newcommand{\R}{\mathbb{R}}
\newcommand{\C}{\mathbb{C}}

\newcommand{\Div}{\operatorname{Div}}
\newcommand{\Aut}{\operatorname{Aut}}

\newcommand{\Hom}{\operatorname{Hom}}
\newcommand{\Spec}{\operatorname{Spec}}

\newcommand{\End}{\operatorname{End}}

\newcommand{\con}{\operatorname{con}}
\newcommand{\sm}{\operatorname{sm}}
\renewcommand{\Im}{\operatorname{Im}}
\newcommand{\id}{\operatorname{id}}
\newcommand{\Var}{\operatorname{Var}}
\newcommand{\supp}{\operatorname{supp}}

\def\I{\mathcal I}

\DeclareMathOperator{\Def}{Def}

\newcommand{\Gr}{\operatorname{Gr}}

\newcommand{\Pic}{\operatorname{Pic}}

\newcommand{\Gm}{\mathbb{G}_m}

\newcommand{\XS}{X_S^{\nu}}

\newcommand{\Xsing}{X_{\rm{sing}}}

\newcommand{\NF}{\rm NF}

\DeclareMathOperator{\TF}{TF}
\DeclareMathOperator{\Simp}{Simp}
\DeclareMathOperator{\PIC}{Pic}

\renewcommand{\SS}{\operatorname{\mathcal{SS}}}
\newcommand{\SSc}{\operatorname{\mathcal{SS}}_{\rm con}}
\newcommand{\ST}{\operatorname{\mathcal{ST}}}
\newcommand{\BD}{\operatorname{BD}}
\newcommand{\bE}{\mathbb{E}}
\newcommand{\bC}{\mathbb{C}}
\newcommand{\be}{\mathbb{e}}

\begin{document}


\title[Classification of fine compactified Jacobians]{On the classification of fine compactified Jacobians of nodal curves}

\author{Filippo Viviani}
\address{Filippo Viviani, Dipartimento di Matematica, Universit\`a di Roma ``Tor Vergata'', Via della Ricerca Scientifica 1, 00133 Roma, Italy}
\email{filippo.viviani@gmail.it}




\begin{abstract}
We introduce a new class of fine compactified Jacobians for nodal curves, that we call  fine compactified Jacobians of vine type, or simply \emph{fine V-compactified Jacobians}.
This class is strictly larger than the class of fine classical compactified Jacobians, as constructed by Oda-Seshadri, Simpson, Caporaso and Esteves.
Inspired by a recent preprint of Pagani-Tommasi, we characterize fine V-compactified Jacobians as the fine compactified Jacobians that can arise as limits of Jacobians of smooth curves under a one-parameter smoothing of the nodal curve or over its semiuniversal deformation space.  Furthermore, fine V-compactified Jacobians are exactly the ones that give a compactification of the N\'eron model of the Jacobian of the generic fiber of a one-parameter regular smoothing of the nodal curve. Finally, we give a combinatorial characterization of fine V-compactified Jacobians in terms of their poset of orbits. 
\end{abstract}

\maketitle

\tableofcontents

\section*{Introduction}

The aim of this paper is to give a complete answer to the following natural

\vspace{0.1cm}

\textbf{Question:} Suppose that a smooth (projective, connected) curve  degenerates to a nodal (projective, connected) curve  in one parameter family, how can we degenerate "modularly" the Jacobian of the generic smooth curve?

\vspace{0.1cm}

This question has been raised by Pagani-Tommasi in the inspiring recent preprint \cite{PT2}, and our work can been seen as a completion of their work. 

\vspace{0.2cm}

Let us now clarify, following \cite{PT1} and \cite{PT2}, what do we mean by modular degeneration in the above Question. Given a nodal curve $X$ over an algebraically closed field $k=\ov k$, a \emph{fine compactified Jacobian}\footnote{The word fine refers to the fact that there is a universal sheaf and the word compactified Jacobian refers to the fact that the smooth locus is a disjoint union of several copies of the generalized Jacobian of $X$.} of $X$ of degree $d$ is a connected, open and proper substack  of the $\Gm$-rigidified stack $\TF_X$ of rank-$1$, torsion-free sheaves of degree $d$ on $X$ (see Section \ref{Sec:sheaves} for more details). Hence, the above question means: find all the degree-$d$ fine compactified Jacobians of the central nodal curve that can arise as limit of the degree-$d$ Jacobian of the generic smooth curve.

The problem of finding fine compactified Jacobians is very old and it goes back to the work of Igusa in \cite{igusa}  and Mayer-Mumford in \cite{MM}
in the 50's--60's. Since then, several authors have worked on the problem using different techniques and all of them came up with a class of fine compactified Jacobians, that we baptize \emph{fine classical compactified Jacobians}: Oda-Seshadri \cite{OS} using graph theory and toric geometry; Simpson \cite{simpson} using slope stability;  Caporaso \cite{caporaso} using GIT of embedded curves; Esteves \cite{est1} using vector bundle stability. We refer to \cite[\S 1,2]{alexeev}, \cite[\S 2]{MV}, \cite[\S 2.2]{CMKV}, \cite[\S 2.2]{MRV}  for an account on the way the different constructions relate to one another.

The main novelty in our work (and in the work of Pagani-Tommasi \cite{PT2}) is that we aim at classifying \emph{all} possible solution to the above Question. And, in doing so, we discover \emph{new} (i.e. non classical) fine compactified Jacobians, as we now explain.

\vspace{0.2cm}

\subsection*{The results}

We now describe our results. 

First of all, we study the properties that are common to every fine compactified Jacobian of a given nodal connected curve $X$.  
In order to state these properties, recall  that the stack $\TF_X$ of rank-$1$, torsion-free sheaves on $X$ admits an action of the generalized Jacobian $\PIC^{\un 0}(X)$ (which is the semiabelian variety that parametrizes line bundles on $X$ of degree $0$ on each irreducible component), whose orbits are given by $\TF_X=\coprod_{(G,D)} \TF(G,D)$, where $\TF(G,D)$ is the locally closed substack consisting of sheaves whose non-singular graph is the spanning subgraph $G$ of the dual graph $\Gamma_X$, and whose multidegree is the divisor $D$  on $\Gamma_X$ (see Section \ref{Sec:sheaves} for more details). Equivalently, $\TF(G,D)$ is the locus of sheaves that are push-forward of line bundles of multidegree $D$ on the partial normalization of $X$ whose dual graph is $G$. Recall also that given a connected graph $G$, its complexity $c(G)$ is  the number of spanning tres of $G$.


\begin{theoremalpha}(see Fact \ref{F:Simploc}, Lemma \ref{L:prop-fcJ}, Proposition \ref{P:fcJ-orb}, Theorem \ref{T:bound-orb}, Corollary \ref{C:irr-fcJ}) 
Let $\ov J_X$ be a fine compactified Jacobian of $X$. 
\begin{enumerate}
\item $\ov J_X$ is a connected proper $k$-variety (i.e. reduced scheme of finite type over $k$)  whose singularities are \'etale-locally a product of nodes and whose normalization is smooth.

In particular, $\ov J_X$ has locally complete intersection singularities. 
\item The smooth locus $\ov J_X$ is isomorphic to a disjoint union of at least $c(\Gamma_X)$  copies of the generalized Jacobian $\PIC_X^{\un 0}$  of $X$. 

In particular, $\ov J_X$ has pure dimension $g(X)$ and it has at least $c(\Gamma_X)$ irreducible components. 

\item $\ov J_X$ is the union of finitely many $\PIC^{\un 0}(X)$-orbits and, for any spanning subgraph $G$ of $\Gamma_X$, we have that 
$$|\{D\in \Div(\Gamma_X)\: : \TF(G,D)\subset \ov J_X\}|
\begin{cases} 
=0 & \text{ if } G \text{  is not connected, }\\
\geq c(G) & \text{ if } G \text{  is  connected. }
\end{cases}$$
\end{enumerate}
\end{theoremalpha}

Some of the above results were already known: the description of the singularities of $\ov J_X$ follows from the computation by Casalaina-Martin-Kass-Viviani in \cite[\S 3]{CMKV} of the semiuniversal deformation ring of a torsion-free, rank-$1$ sheaf on a nodal curve; the fact that the normalization of $\ov J_X$ is smooth follows from \cite[Cor. 13.3]{OS}; the fact that $\ov J_X$ is a union of $\PIC^{\un 0}(X)$-orbits was shown by Pagani-Tommasi in \cite[Lemma 7.2]{PT2}.

\vspace{0.1cm}

We now focus on the Question we started with. We first introduce the following new stability condition on nodal curves (see Definition \ref{D:Vstab} and Remark \ref{R:Vstab-cur}) 
A \emph{general stability condition of vine type} (or simply a \textbf{general V-stability condition)} of degree $d$ on a nodal curve $X$ (or rather on its dual graph\footnote{This should be called a V-stability condition on the dual graph $\Gamma_X$ of $X$. A slightly different definition of V-stability condition on a reduced curve will be introduced in \cite{FPV1} in terms of the Euler characteristic rather than the degree. However, for the sake of simplicity, we will abuse terminology and call such a collection $\frak n$ a V-stability condition on the curve $X$.}) is an assignment of integers 
$$
\frak n=\{\frak n_Y \: : \: Y\subset X\:  \text{ is a biconnected and non-trivial subcurve}\}
$$
satisfying the following properties:
\begin{enumerate}
\item for any $Y\subset X$ biconnected and non-trivial, we have 
\begin{equation*}
\frak n_Y+\frak n_{Y^c}=d+1-|Y\cap Y^c|;
\end{equation*}
\item  for any disjoint $Y_1,Y_2\subset X$ without common irreducible components such that $Y_1,Y_2$ and $Y_1\cup Y_2$ are biconnected and non-trivial, we have that 
\begin{equation*}
-1\leq \frak n_{Y_1\cup Y_2}-\frak n_{Y_1}-\frak n_{Y_2}-|Y_1\cap Y_2|\leq 0.
\end{equation*}
\end{enumerate}
In the above definition, a subcurve $Y\subset X$ is a union of some of the irreducible components of $X$, and we say that $Y$ is non-trivial if $Y\neq \emptyset, X$ and biconnected if both 
$Y$ and its complementary subcurve $Y^c:=\ov{X\setminus Y}$ are connected.  The terminology of vine type (or V-type for short) is explained in Remark \ref{R:Vtype}. 

\vspace{0.1cm}

To any general V-stability condition on $X$, we can associated a fine compactified Jacobian of $X$ as it follows.

\begin{theoremalpha}\label{ThmA}(see Theorem \ref{T:fcJ-vine}, Corollary \ref{C:unique-n})
For any general V-stability condition $\frak n$ on $X$ of degree $d$, then 
\begin{equation*}
\ov J_X(\frak n):=\{I \in \TF_X^d\: : \:  \deg_Y(I)\geq \frak n_Y \quad   \text{ for any biconnected and non-trivial subcurve } Y\subset X\}.
\end{equation*}
is a fine compactified Jacobian of $X$ of degree $d$, called the \textbf{fine V-compactified Jacobian} associated to $\frak n$. 

Moreover, $\frak n$ is uniquely determined by $\ov J_X(\frak n)$. 
\end{theoremalpha}

The fine classical compactified Jacobians (constructed in \cite{OS}, \cite{simpson},  \cite{caporaso} and  \cite{est1}) are special cases of fine V-compactified Jacobians, as we now explain. 

If $\phi$ is a  \emph{general numerical polarization} on $\Gamma$ of degree $d\in \Z$, i.e. a real divisor $\phi\in \Div(\Gamma)_{\R}:=\Div(\Gamma)\otimes \R=\R^{V(\Gamma)}$ whose degree $|\phi|$ is equal to $d$, such that 
\begin{equation*}
\phi_W-\frac{\val(W)}{2}=\sum_{v\in W} \phi_v-\frac{\val(W)}{2}\not\in \Z \: \text{ for every  biconnected non-trivial  } W\subset V(\Gamma),
\end{equation*}
 then the collection 
 \begin{equation*}
 \frak n(\phi):=\{\frak n(\phi)_W:=\lceil \phi_W-\frac{\val(W)}{2} \rceil  \: \text{ for any  biconnected and non-trivial } W\}
 \end{equation*}
 is a general V-stability condition of degree $d$, called the general V-stability condition associated to $\phi$. 
 Then the fine V-compactified Jacobian associated to  $\frak n(\phi)$ is the \textbf{fine classical compactified Jacobian} associated to $\phi$:  
\begin{equation}\label{E:fcJphi}
\ov J_X(\frak n(\phi)):=\ov J_X(\phi)=\left\{I \in \TF_X^d\: : \:  \deg_Y(I)\geq \phi_Y-\frac{|Y\cap Y^c|}{2}\right\}.
\end{equation}

In the next Theorem, we clarify the relation between fine compactified Jacobians and the two special subclasses of classical type and of V-type.  

\begin{theoremalpha}\label{ThmB}
Let $X$ be a nodal curve.
\begin{enumerate}
\item \label{ThmB1} The inclusions 
\begin{equation}\label{E:inc-fcJ}
\left\{\begin{aligned}
\text{Fine classical compactified} \\
\text{Jacobians of } X
\end{aligned}\right\} \hookrightarrow
\left\{\begin{aligned}
\text{Fine V-compactified} \\
\text{Jacobians of } X
\end{aligned}\right\} \hookrightarrow
\left\{\begin{aligned}
\text{Fine compactified} \\
\text{Jacobians of } X
\end{aligned}\right\} 
\end{equation}
are strict in general. 
Namely, there exist nodal curves $X$ for which the first inclusion is strict, and others for which the second inclusion is strict. 
\item \label{ThmB2} The fine V-compactified Jacobians of $X$ (and hence also the fine classical compactified Jacobian) are finite up to translation, while the fine compactified Jacobians are not finite up to translation in general. 
\end{enumerate}
\end{theoremalpha}
In the second part of the Theorem, we say that two fine compactified Jacobians of $X$ are \emph{equivalent by translation} if one can be obtained from the other by multiplying all its sheaves by some fixed line bundle on $X$, see Definition \ref{D:eq-fcJ}. This defines an equivalence relation on the set of fine compactified Jacobians, that preserves the class of fine compactified Jacobians of classical or V-type. 

 
\begin{proof}
Part \eqref{ThmB1} is proved as it follows:
\begin{itemize}
\item The fact that the first inclusion in \eqref{E:inc-fcJ} is strict amounts to find nodal curves $X$ having general V-stability conditions that do not come from general numerical polarizations: 
the smallest of  such a nodal curve  (i.e. the nodal curve with the smallest number of irreducible components) is the nodal curve of genus $3$ with $6$ irreducible components whose dual graph is described in  Example \ref{Ex:msa-not}.

\item The fact that the second inclusion in \eqref{E:inc-fcJ} is strict amount to find nodal curves $X$ having fine compactified Jacobians that are not of V-type: the easiest of such a nodal curve is the $n$-cycle curve of genus $1$ with $n\geq 3$, see Example \ref{Ex:cycle} (which is based on \cite[Sec. 3]{PT1}). 
\end{itemize}

Part \eqref{ThmB2}: the fact that fine V-compactified Jacobians are finite up to translation is proved in Corollary \ref{C:finite-eq}, while on the $n$-cycle curve of genus $1$ (with $n\geq 3$) the number of fine compactified Jacobian is not finite up to translation since the number of irreducible components can be arbitrarily large (as proved in \cite[Sec. 3]{PT1}, see  Example \ref{Ex:cycle}). 
\end{proof}

The above Theorem \ref{ThmB} says that the class of fine V-compactified Jacobians is strictly larger than the class of fine classical compactified Jacobians, and yet it is possible to hope for a classification of them up to translation (while the same problem appear hopeless for arbitrary fine compactified Jacobians). 

Let us mention that the fact that fine classical compactified Jacobians are finite up to translation was shown by Melo-Rapagnetta-Viviani in  \cite[Prop. 3.2]{MRV}.
The same authors have shown in \cite[Thm. B]{MRV} that the number of equivalence classes of fine classical compactified Jacobians up to translation (or more generally, the number
of isomorphism classes or the number of homeomorphism classes if $k=\C$) can become arbitrarily large as the curve vary. 

\vspace{0.1cm}

In the next Theorem, we will see that the fine V-compactified Jacobians provide a complete solution to the Question we have started with.

\begin{theoremalpha}\label{ThmC}(see Theorem \ref{T:V-fcJ})
Let $\ov J^d_X$ be a fine compactified Jacobian of $X$ of degree $d$. Then the following conditions are equivalent:
\begin{enumerate}
\item  $\ov J^d_X$ is fine V-compactified Jacobian, i.e. $\ov J_X=\ov J_X(\frak n)$ for some general V-stability condition $\frak n$ on $X$;
\item $\ov J^d_X$ is \emph{universally smoothable}, i.e. for the semiuniversal effective deformation $\Pi: \bX\to \Spec(R_X)$ of $X$, there exists an open substack $\ov J_{\bX/\Spec(R_X)}^d$ of $\TF^d_{\bX/\Spec(R_X)}$, which is proper and flat over $\Spec(R_X)$, and such that  $(\ov J^d_{\bX})_0\cong \ov J^d_X$. 
\item $\ov J^d_X$ is  \emph{smoothable}, i.e. if for any  one-parameter smoothing $\pi:\X\to \Delta$ of $X$, the open subscheme 
$\ov J_{\X/\Delta}^d:=\ov J_X^d\coprod \PIC_{\X_{\eta}}^d\subset \TF^d_{\X/\Delta}$ is proper over $\Delta$.
\item  $\ov J^d_X$ is \emph{weakly smoothable}, i.e. if there exists a one-parameter smoothing $\pi:\X\to \Delta$ of $X$ such that 
the open subscheme  $\ov J_{\X/\Delta}^d:=\ov J_X^d\coprod \PIC_{\X_{\eta}}^d\subset \TF^d_{\X/\Delta}$ is proper over $\Delta$. 
 \item $\ov J_X^d$ is of \emph{N\'eron-type}, i.e. for every one-parameter regular smoothing $\pi:\X\to \Delta$ of $X$, the open subscheme 
$\ov J_{\X/\Delta}^d:=\ov J_X^d\coprod \PIC_{\X_{\eta}}^d\subset \TF^d_{\X/\Delta}$ is proper over $\Delta$ and its relative smooth locus is canonically isomorphic to the N\'eron model of $\PIC_{\X_{\eta}}^d$. 
\end{enumerate}
\end{theoremalpha}
In the above Theorem, we denote by $\TF^d_{\bX/\Spec(R_X)}$ (resp. $\TF^d_{\X/\Delta}$) the $\Gm$-rigidified stack of relative rank-$1$, torsion-free, degree-$d$ sheaves on the family 
$\bX\to \Spec(R_X)$ (resp. $\X\to \Delta$). 

Let us comment on the origin of the above properties: 
\begin{itemize}
\item the property of being smoothable or weakly smoothable was introduced (with a slightly different terminology) by Pagani-Tommasi \cite{PT1} and \cite{PT2}, where it is also shown that these two properties are equivalent;
\item the property of being universally smoothable is inspired by the universal fine compactified Jacobians studied by Melo-Rapagnetta-Viviani in \cite[\S 5]{MRV}; 
\item the property of being of N\'eron-type was introduced by Caporaso in \cite{capNtype} (see also \cite{And}, \cite{capneron}, \cite{Chi}, \cite{Mel2}, \cite{Hol} for a study of this property for the compactified universal Jacobian over the moduli stack of stable curves). 

\end{itemize}

The above properties were known for fine classical compactified Jacobians: the fact that fine classical compactified Jacobians are universally smoothable (which then implies that they are smoothable and weakly smoothable, a fact originally due to Ishida \cite{Ish}) was proved by Melo-Rapagnetta-Viviani \cite[Thm. 5.4]{MRV}; the fact that fine classical compactified Jacobians are of N\'eron-type was proved  by Melo-Viviani \cite[Thm. 4.1]{MV}, generalizing previous partial results of Caporaso (see \cite[Thm. 6.1]{capneron}, \cite[Thm 2.9]{capNtype}). Moreover, these results hold true for fine classical compactified Jacobians of reduced curves with locally planar singularities: the universal smoothability was proved in \cite[Thm. 5.4]{MRV} and the property of being of N\'eron type was proved by Kass \cite{Kas2}.

\vspace{0.1cm}

Finally, fine V-compactified Jacobians can also be characterized in terms of their poset of orbits under the action of the generalized Jacobian $\PIC^{\un 0}(X)$. 

\begin{theoremalpha}\label{ThmD}(see Theorem \ref{T:fcJ-vine} and Corollary \ref{C:sm-VfcJ})
Let $\ov J_X^d$ be a fine compactified Jacobian of $X$ of degree $d$. 
Then the following conditions are equivalent:
\begin{enumerate}
\item  $\ov J^d_X$ is fine V-compactified Jacobian, i.e. $\ov J^d_X=\ov J_X(\frak n)$ for some general V-stability condition $\frak n$ on $X$;
\item $\ov J_X^d$ is of \emph{sN-type}, i.e. for any connected subgraph $G$ of $\Gamma_X$ we have that the composition 
$$\{D\in \Div^{d-|E(G)^c|}(G)\: : \TF(G,D)\subset \ov J_X^d\}\subset \Div^{d-|E(G)^c|}(G)\twoheadrightarrow  \Pic^{d-|E(G)^c|}(G)$$  
is an isomorphism. 
\item $\ov J_X^d$ is of \emph{numerical  sN-type}, i.e. for any connected subgraph $G$ of $\Gamma_X$ we have that 
$$|\{D\in \Div^{d-|E(G)^c|}(G)\: : \TF(G,D)\subset \ov J_X^d\}|=c(G).$$  
\item $\ov J_X^d$ is of \emph{N-type}, i.e. we have that the composition 
$$\{D\in \Div^{d}(\Gamma_X)\: : \TF(\Gamma_X,D)\subset \ov J_X^d\}\subset \Div^{d}(\Gamma_X)\twoheadrightarrow  \Pic^{d}(\Gamma_X)$$  
is an isomorphism. 
\item $\ov J_X^d$ is of \emph{numerical N-type}, i.e. we have that 
$$|\{D\in \Div^{d}(\Gamma_X)\: : \TF(\Gamma_X,D)\subset \ov J_X^d\}|=c(\Gamma_X).$$  
\item $\ov J_X^d$ has $c(\Gamma_X)$ irreducible components.
\end{enumerate}
\end{theoremalpha}
In the above Theorem, the group $\Pic^e(G)$ is the degree-$e$ Jacobian of the graph $G$, which is the quotient of the set $\Div^e(G)$ of degree-$e$ divisors of $G$ by the image of the Laplacian $\Delta^G$, or equivalently modulo the relation of chip-firing. The cardinality of $\Pic^e(G)$ is equal to the complexity $c(G)$ of the graph $G$, which is the number of spanning trees of $G$ (see Section \ref{Sec:not-gr} for more details).

The property of being of sN-type was introduced by Pagani-Tommasi \cite[Def. 4.3]{PT2} under the name of ''degree-$d$ stability condition". The same authors proved the crucial result that 
these degree-d stability conditions correspond bijectively to (resp. weakly) smoothable fine compactified Jacobians (see Theorem \ref{T:PT-thm} for a reformulation of their result in our language). 
The fact that fine classical compactified Jacobians are of sN-type (and hence of numerical sN-type, of N-type and of numerical N-type) was proved by Melo-Viviani \cite[Thm. 5.1]{MV}.

\vspace{0.2cm}

Using the above results, we can extend to fine V-compactified Jacobians of nodal curves some of the results that were proved in \cite{MRV}, \cite{MRV2}, \cite{MRV3}, \cite{MSV} and \cite{CPS} for fine classical compactified Jacobians. 

\begin{corspec}(see Corollaries \ref{C:sm-VfcJ}, \ref{C:K-var} and Propositions \ref{P:trivdual}, \ref{P:Lerper}, \ref{P:MumV}, \ref{P:autod}, \ref{P:der})
Let $\ov J_X(\frak n)$ be a fine V-compactified Jacobian on a connected nodal curve $X$ over $k=\ov k$.
\begin{enumerate}
\item The smooth locus of  $\ov J_X(\frak n)$ of $X$ is  isomorphic to $c(\Gamma_X)$ copies of the generalized Jacobian $\Pic^{\un 0}_X$ of $X$.

In particular, any two fine V-compactified Jacobians of $X$ are birational. 
\item   The class of $\ov J_X(\frak n)$ in the Grothendieck ring $K_0(\Var_k)$ of varieties over $k$ is equal to  
$$
[\ov J_X(\frak n)]=c(\Gamma_X)[J^0_{X^{\nu}}]\cdot \mathbb{L}^{g(\Gamma_X)},
$$
where $J^0_{X^{\nu}}$ is the Jacobian of the normalization $X^{\nu}$ of $X$ and $\mathbb{L}=[\mathbb{A}^1]$. 
\item $\ov J_X^d(\frak n)$ is a Calabi-Yau variety, i.e. the dualizing sheaf of  $\ov J_X^d(\frak n)$ is trivial.
\item If $k=\C$, the graded pieces of the Leray perverse filtration on the cohomology of $\ov J_X(\frak n)$ satisfy the following equality in the Grothendieck group of Hodge structures: 
\begin{equation*}
\sum_n q^n \Gr^n_P H^*(\ov J_X(\frak n), \Q)=\sum_m q^m H^m(J^0_{X^{\nu}},\Q) \cdot \sum_h n_h(\Gamma_X) (q\mathbb{L})^{g(\Gamma_X)-h} ((1-q)(1-q\mathbb{L}))^{h},
\end{equation*}
where 
$n_h(\Gamma_X)$ is the number of genus $h$ connected spanning subgraphs of the dual graph $\Gamma_X$ of $X$.  

In particular, $\Gr^n_P H^*(\ov J_X(\frak n), \Q)$ is independent of the chosen general V-stability condition $\frak n$.

\item There exists a Poincar\'e line bundle $\cP$ on $\PIC^{\un 0}(X) \times \ov J_X(\frak n)$ inducing a group scheme isomorphism 
$$
\begin{aligned}
\beta_{\frak n}: \PIC^{\un 0}(X) & \xrightarrow{\cong} \PIC^o(\ov J_X(\frak n))=\PIC^{\tau}(\ov J_X(\frak n))\\
M & \mapsto \cP_M:=\cP_{\ov J_X(\frak n)\times \{M\}}.
\end{aligned}
$$

\item Let $\ov J_X(\frak n')$ be another fine V-compactified Jacobian of $X$ and assume that $k$ has characteristic either zero or greater than $g(X)$.
Then there exists a maximal Cohen-Macauly sheaf (called the Poincar\'e sheaf) $\ov \cP$ on $\ov J_X(\frak n)\times \ov J_X(\frak n')$, flat over the two factors, such that the induced integral transform 
$$
\begin{aligned}
\Phi^{\ov \cP}: D^b_{\rm{(q)coh}}(\ov J_X(\frak n))& \longrightarrow D^b_{\rm{(q)coh}}(\ov J_X(\frak n'))\\
\cE^\bullet &\mapsto Rp_{2*}(p_1^*(\cE^{\bullet})\otimes^L \ov \cP)
\end{aligned}
$$ 
is an equivalence of triangulated categories, where $D^b_{\rm{(q)coh}}(-)$ denotes the derived category of quasi-coherent (resp. coherent) sheaves. 

In particular, any two fine V-compactified Jacobians of $X$ are derived equivalent.

\item Let $\pi:\X\to \Spec R$ be a flat and proper morphism over a  complete rank-$1$ valued ring $R$ with algebraically closed fraction field $K$ such that 
such that the generic fiber $\X_K$ is a smooth  and the special fiber is equal to $X$. 
Then the formal completion of the relative V-compactified Jacobian
$$
\ov J_{\X}(\frak n):=\ov J_X(\frak n)\coprod \Pic^d(\X_K)\to \Spec R
$$
along the special fiber is a Mumford model of $\PIC^d(\X_K)$. 
\end{enumerate}
\end{corspec}

\subsection*{Open questions}

This work leaves open some natural question which we plan to investigate in the near future.

\begin{enumerate}
\item Can one define non-general V-stability conditions on a nodal curve $X$ and their associated (non-fine) V-compactified Jacobians\footnote{This has now been solved with N. Pagani and M. Fava \cite{FPV1} and \cite{FPV2}}? 

In analogy with the fine case, we expect that V-compactified Jacobians form a class which is strictly smaller than the class of all compactified Jacobians, i.e. open and connected subsets of  the stack of torsion-free, rank-$1$ sheaf with a proper good moduli space, and at the same time strictly larger than the class of  (non-fine) classical compactified Jacobians associated to a (non-general) numerical polarization (see \cite{OS}).

\item Can one define (fine or non-fine) V-compactified Jacobians for worse-than-nodal singular reduced  curves\footnote{This has now been solved with N. Pagani and M. Fava \cite{FPV1}.}? 

Classical compactified Jacobians can be defined for arbitrary reduced curves (for example via slope stability \cite{simpson} or vector bundle stability \cite{est1}) and they behaves particularly well for reduced curves with locally planar singularities (see \cite{MRV}). We expect that the same should hold for V-compactified Jacobians.

\item For a fixed curve $X$, can we describe the poset structure of the set of V-stability conditions (resp. up to translation)?

For (classical) numerical polarizations, there is a hyperplane arrangement in a real vector space whose poset of regions is the poset of classical V-stability conditions (see Example \ref{Ex:phi}).
Moreover, this arrangement is invariant under the action of a lattice and it defines a toric arrangement in a real torus whose poset of regions is the poset of classical V-stability conditions up to translation (see Remark \ref{R:phi-tras}). Is there a similar structure for (non-classical) V-stability conditions?

\item Is any fine V-compactified Jacobian (or indeed, any fine compactified Jacobian) projective?

 Note that the classical fine compactified Jacobians are projective since they can be constructed via GIT, but such a construction is not known for fine V-compactified Jacobians.



\item Relative general V-stability conditions and their associated relative fine V-compactified Jacobians can also be defined for families of nodal curves. Can we classify the relative (fine or non-fine) V-stability conditions for the universal family $\cC_{g,n}/\ov \M_{g,n}$ over the moduli stack of stable $n$-pointed curves of genus $g$ \footnote{While this preprint was under referee's process, the classification has been solved for relative fine V-stability conditions by M. Fava \cite{Fav}. The case of relative (non fine) V-stability conditions is work in progress with N. Pagani and M. Fava \cite{FPV3}.}?

The relative classical stability conditions for the universal family $\cC_{g,n}/\ov \M_{g,n}$ have been classified in \cite{KP} and \cite{Mel1}. More generally, relative V-stability conditions for $\cC_{g,n}/\ov \M_{g,n}$ have been classified for $g=1$ in \cite[\S 6]{PT1} and for $n=0$ in \cite[\S 9]{PT2}.

\end{enumerate}

\subsection*{Outline of the paper}

The paper is divided in two sections. In the first section, we deal with general V-stability conditions on graphs and their relationship with generalized break divisors. In the second section, we deal with fine compactified Jacobians of nodal curves and we investigate the properties of fine V-compactified Jacobians. 

\vspace{0.5cm}


\paragraph { \bf Acknowledgements}

This paper owes its existence to the reading of the two very inspiring papers of Pagani-Tommasi \cite{PT1} and \cite{PT2}. I thank N. Pagani for several useful conversations on these two papers, and M. Fava and N. Pagani for several comments on this manuscript and for our joint works \cite{FPV1}, \cite{FPV2}, \cite{FPV3}, which can be seen as a completion of the work started here. I thank Sam Molcho for spotting a flaw in a previous (wrong) proof of the projectivity of compactified Jacobians (which is now an open question). 
I also take  the opportunity to thank the coauthors with whom I have worked on compactified Jacobians in the past years and that have help to shape my understanding of this topic: L. Caporaso, S. Casalaina-Martin, J.L. Kass, M. Melo, L. Migliorini, A. Rapagnetta, V. Shende. 

The author is supported by the MUR  ``Excellence Department Project'' MATH@TOV, awarded to the Department of Mathematics, University of Rome Tor Vergata, CUP E83C18000100006, by the  PRIN 2022 ``Moduli Spaces and Birational Geometry''  2022L34E7W funded by MUR,  and he is a member of INdAM and of the Centre for Mathematics of the University of Coimbra (funded by the Portuguese Government through FCT/MCTES, DOI 10.54499/UIDB/00324/2020).

\section{General V-stability and Generalized Break Divisors on graphs}

\subsection{Notation on graphs}\label{Sec:not-gr}

Let $\Gamma$ be a finite graph with vertex set $V(\Gamma)$ and edge set $E(\Gamma)$. The \emph{genus} of $\Gamma$ is 
$$g(\Gamma):=|E(\Gamma)|-|V(\Gamma)|+\gamma(\Gamma),$$
where $\gamma(\Gamma)$ is the number of connected components of $\Gamma$. 


We will be interested in two kinds of subgraphs of $\Gamma$:
\begin{itemize}
\item Given a subset $T\subset E(\Gamma)$, we denote by   $\Gamma\setminus T$ the
subgraph of $\Gamma$ obtained from $\Gamma$ by deleting the edges belonging to $T$.
Thus we have that $V(\Gamma\setminus T)=V(\Gamma)$ and $E(\Gamma\setminus T)=E(\Gamma)\setminus T$.
The subgraphs of the form $\Gamma\setminus T$ are called \emph{spanning subgraphs}.

The set of all spanning subgraphs of $\Gamma$ is denoted by $\SS(\Gamma)$ and it is a poset via the relation
$$
G=\Gamma\setminus T\leq G'=\Gamma\setminus T' \Longleftrightarrow T\supseteq T'.
$$
If $\Gamma$ is connected, we will denote by $\SSc(\Gamma)$ the subset of $\SS(\Gamma)$ consisting of all connecting spanning subgraphs.
The minimal elements of $\SSc(\Gamma)$ are exactly the \emph{spanning tree} of $\Gamma$, i.e. the connected spanning subgraphs that are trees.
The set of all spanning trees of $\Gamma$ will be denoted by $\ST(\Gamma)$.

\item Given a subset $W\subset V(\Gamma)$, we denote by   $\Gamma[W]$ the subgraph
whose vertex set is $W$ and whose edges are all the edges of $\Gamma$ that join two vertices in $W$.
The subgraphs of the form $\Gamma[W]$ are called \emph{induced subgraphs} and we say that
$\Gamma[W]$ is induced from $W$.
\end{itemize}

Given a subset $T\subset E(\Gamma)$, we will denote by $\Gamma/T$ the $T$-\emph{contraction} of $\Gamma$, i.e. the graph obtained by $\Gamma$ by contracting all the 
edges belonging to $T$.  A \emph{morphism of graphs} $f:\Gamma\to \Gamma'$ is a contraction of some edges of $\Gamma$ followed by an automorphism of $\Gamma'$. 
Such a morphism induces an injective pull-back map on edges $f^E:E(\Gamma')\hookrightarrow E(\Gamma)$ and a surjective push-forward map on vertices $f_V:V(\Gamma)\twoheadrightarrow V(\Gamma')$.
Note that the complement $(\Im f^E)^c$ of the image of $f^E$ consists of the edges that are contracted by $f$.

Fix a subset $S\subseteq E(\Gamma)$. For any pair $(W_1,W_2)$ of disjoint subsets of $V(\Gamma)$, we denote the \emph{$S$-valence} of $(W_1,W_2)$ to be 
$$\val_S(W_1,W_2):=|S\cap E(\Gamma[W_1],\Gamma[W_2])|,$$
where $E(\Gamma[W_1], \Gamma[W_2])$ is the subset of $E(\Gamma)$ consisting of all the edges
of $\Gamma$ that join some vertex of $W_1$ with some vertex of $W_2$. As a special case, the $S$-valence of a subset $W\subseteq V(\Gamma)$ is 
$$\val_S(W):=\val_S(W,W^c),$$
where $W^c:=V(\Gamma)\setminus W$ is the complementary subset of $W$.  For any subset $W\subseteq V(\Gamma)$, we denote the \emph{S-degree} of $W$ to be 
$$
e_S(W):=|S\cap E(\Gamma[W])|.
$$
In the special case $S=E(\Gamma)$, we set 
$$
e_{\Gamma}(W):=e_{E(\Gamma)}(W) \text{ and } \val_{\Gamma}(W_1,W_2):=\val_{E(\Gamma)}(W_1,W_2),
$$
and we call them the $\Gamma$-degree and the $\Gamma$-valence, respectively. When the graph $\Gamma$ is understood, we set 
$$
e(W):=e_{\Gamma}(W) \text{ and } \val(W_1,W_2):=\val_{\Gamma}(W_1,W_2).
$$
Note that, using the identification $V(\Gamma)=V(\Gamma\setminus S)$, we have that 
 \begin{equation}\label{E:val-valS}
e_{\Gamma}(W)=e_{S}(W)+e_{\Gamma\setminus S}(W) \text{ and } \val_{\Gamma}(W_1,W_2)=\val_{S}(W_1,W_2)+\val_{\Gamma\setminus S}(W_1,W_2).
\end{equation}
Given pairwise disjoint subsets $W_1, W_2, W_3\subseteq V(\Gamma)$, we have the following formulas
\begin{equation}\label{add-val}
\begin{sis}
& \val_S(W_1\cup W_2,W_3)=\val_S(W_1,W_3)+\val_S(W_2,W_3),\\
& \val_S(W_1\cup W_2)=\val_S(W_1)+\val_S(W_2)-2\val_S(W_1,W_2),\\
& e_S(W_1\cup W_2)=e_S(W_1)+e_S(W_2)+\val_S(W_1,W_2),\\
& |S|=e_S(W_1)+e_S(W_1^c)+\val_S(W_1).
\end{sis}
\end{equation}

We will denote by $\Div(\Gamma):=\Z^{V(\Gamma)}$ the group of \emph{divisors} on the graph $\Gamma$.
Given $D\in \Div(\Gamma)$, we will set $D_v:=D(v)$ and we will write $D$ as an integral linear combination 
$$
D=\sum_{v\in V(\Gamma)} D_v \cdot v.
$$ 
Given a subset $W\subseteq V(\Gamma)$, we will set 
$D_W:=\sum_{v\in V(\Gamma)} D_v.$
The degree of a divisor $D$ is given by 
$$
\deg(D)=|D|:=D_{V(\Gamma)}=\sum_{v\in V(\Gamma)} D_v.
$$
For any $d\in \Z$, we will denote by $\Div^d(\Gamma)$ the set of all divisors of degree $d$.

The space of divisors on $\Gamma$ is endowed with an endomorphism (called \emph{Laplacian}):  
\begin{equation}\label{Lapla}
\Delta^{\Gamma}(D)_v:=-D_v \val_{\Gamma}(v) +\sum_{v\neq w\in V(\Gamma)} D_w \val_{\Gamma}(v,w).
\end{equation}
We will set $\Delta=\Delta^{\Gamma}$ if the graph $\Gamma$ is understood. 

It turns out that $\Im(\Delta^{\Gamma})\subset \Div^0(\Gamma)$ and that, if $\Gamma$ is connected, then the kernel of $\Delta^{\Gamma}$ consists of the constant divisors, i.e. those divisors $D$ such that $D_v=d$ for some $d\in \Z$. Therefore, if $\Gamma$ is connected, the quotient
\[\Pic^0(\Gamma):=\frac{\Div^0(\Gamma)}{\Im(\Delta^{\Gamma})}\]
is a finite group, called the \emph{Jacobian group} of $\Gamma$ (also known as degree class group, or  sandpile group, or  critical
group of the graph $\Gamma$), see e.g. \cite{BdlHN}, \cite{BN}, \cite{BMS2}. The Kirchhoff's matrix-tree theorem implies that the cardinality of $\Pic(\Gamma)$ is 
equal to the \emph{complexity} $c(\Gamma)$ of $\Gamma$, which is the number of spanning trees of $\Gamma$.

For any $d\in \Z$, the set $\Div^d(\Gamma)$ is a torsor for the group $\Div^0(\Gamma)$.
Therefore, the subgroup $\Im(\Delta^{\Gamma})$ acts on $\Div^d(\Gamma)$ and the quotient 
\begin{equation}\label{lat-class0}
\Pic^d(\Gamma)=\frac{\Div^d(\Gamma)}{\Im(\Delta^{\Gamma})}
\end{equation}
is a torsor for $\Pic^0(\Gamma)$.


\subsection{The poset of divisors on spanning subgraphs}\label{S:poset-C0}

In this subsection we introduce a poset parametrizing divisors on spanning subgraphs of a given graph $\Gamma$.

In order to define this poset, we need to introduce indegree divisors. Given a partial orientation $\O$ of $\Gamma$, i.e. an orientation of a subset $S\subseteq E(G)$ denoted by $\supp(\O)$ and called the support of the orientation $\O$, we denote by  $\D(S,\O)$, or simply $\D(\O)$, the \emph{indegree divisor} of the orientation $\O$, i.e. 
the divisor on $\Gamma$ whose value at a vertex $v$ is the number of oriented edges of $\O$ whose end point is $v$. Note that $\D(\O)$ has total degree $|\supp(\O)|$. 
We denote the collection of all indegree divisors associated to a given $S\subseteq E(\Gamma)$ by  
\begin{equation}\label{E:DS}
\D(S):=\{\D(S, \O)\: : \: \O \text{ is an orientation of the edges of } S\}\subseteq \Div^{|S|}(\Gamma). 
\end{equation}
Given two spanning subgraphs $G\geq G'$ of $\Gamma$, we set 
$$
\begin{sis}
& \D(G-G',\O):=\D(E(G)-E(G'),\O) \text{ for any orientation } \O \text{ of } E(G)-E(G'), \\
 & \D(G-G'):=\D(E(G)-E(G')). 
\end{sis}
$$


\begin{defi}\label{D:posetO}
Let $\Gamma$ be a graph and let $d\in \Z$. The poset of \emph{divisors on spanning subgraphs} of degree $d$ is the set 
$$
\bO^d(\Gamma):=\{(G,D)\: :\: G\in \SS(\Gamma),  D\in \Div^{d-|E(G)^c|}(G)=\Div^{d-|E(G)^c|}(\Gamma)\}, 
$$
where $E(G)^c=E(\Gamma)-E(G)$, endowed with the following partial order
\begin{equation}\label{E:posetO}
(G,D)\geq (G',D') \Leftrightarrow 
 \begin{sis} 
 & G\geq G' \\
 & D\in D'+\D(G-G').
 \end{sis}
 \end{equation}
 
If $\Gamma$ is connected, then the sub-poset of \emph{divisors on connected spanning subgraphs} of degree $d$ is given by 
 $$
 \bO^d_{\con}(\Gamma):=\{(G,D)\in \bO^d(\Gamma) : G \text{ is connected}\}.
 $$
 
 We also set 
  $$
 \bO(\Gamma):=\coprod_{d\in \Z}\bO^d(\Gamma) \: \text{ and }  \bO_{\rm con}(\Gamma):=\coprod_{d\in \Z}\bO_{\rm con}^d(\Gamma),
 $$
 endowed with the disjoint poset structure.  
\end{defi}
Note that there are surjective and order-preserving forgetful maps 
$$
\bO^d(\Gamma) \to \SS(\Gamma) \quad \text{ and } \quad \bO_{\con}^d(\Gamma)\to \SSc(\Gamma).
$$
The poset $\bO^d(\Gamma)$ is graded with respect to the rank function
$$
\begin{aligned}
\rho: \bO^d(\Gamma)  \longrightarrow  \SS(\Gamma) & \longrightarrow \bN, \\
(G, D) \mapsto G & \mapsto |E(G)|. 
\end{aligned}
$$
In particular, $\bO^d(\Gamma)$ has rank equal to $|E(\Gamma)|$. 
Similarly,  $\bO_{\con}^d(\Gamma)$ is ranked with the respect to the restriction of the function $\rho$, and it has rank equal to $b_1(\Gamma)$.

Given $G\in \SS(\Gamma)$ (resp. $G\in \SSc(\Gamma)$), then there is a natural inclusion of posets 
\begin{equation}\label{E:incl-O}
\bO^{d-|E(G)^c|}(G)\hookrightarrow \bO^d(\Gamma) \quad (\text{resp. } \bO_{\con}^{d-|E(G)^c|}(G)\hookrightarrow \bO_{\con}^d(\Gamma)).
\end{equation}

We now study the behavior of the poset $\bO^d(\Gamma)$ under maps of graphs. Given a morphism $f:\Gamma \to \Gamma'$ of graphs, we define the push-forward along $f$ as
\begin{equation}\label{E:f*-O}
\begin{aligned}
f_*: \bO^d(\Gamma) & \longrightarrow \bO^d(\Gamma')\\
(\Gamma\setminus S, D) & \mapsto f_*(\Gamma\setminus S,D):= \left(f_*(\Gamma\setminus S), f_*(D)\right),
\end{aligned}
\end{equation}
where $f_*(\Gamma\setminus S):=\Gamma'\setminus (f^E)^{-1}(S)$ and $f_*(D)$ is defined by 
$$
 f_*(D):=\left\{f_*(D)_v:=D_{f_V^{-1}(v)}+e_{S\cap (\Im f^E)^c}(f_V^{-1}(v)) \right\}_{v\in V(\Gamma')}
$$
Note that this is well-defined since 
$$
|f_*(D)|=|D|+|S\cap (\Im f^E)^c|=d-|S|+|S\cap (\Im f^E)^c|=d-|(f^E)^{-1}(S)|.
$$
Moreover, the map $f$ induces a map of graphs 
$$
\Gamma\setminus (S\cap \Im(f^E)) \to \Gamma'\setminus (f^E)^{-1}(S)=f_*(\Gamma\setminus S)
$$
and $\Gamma\setminus S$ is a spanning subgraph of $\Gamma\setminus (S\cap \Im(f^E))$. 
Therefore, if $\Gamma\setminus S$ is connected, then also $f_*(\Gamma\setminus S)$ is connected, which implies that $f_*$ induces a map 
\begin{equation}\label{E:f*-Ocon}
f_*: \bO_{\con}^d(\Gamma)  \longrightarrow \bO_{\con}^d(\Gamma').
\end{equation}

\begin{prop}\label{P:mapsO}
Le $f:\Gamma \to \Gamma'$ be a map of graphs. The push-forward map $f_*:\bO^d(\Gamma)\to \bO^d(\Gamma')$ is such that:
\begin{enumerate}[(i)]
\item \label{P:mapsO1} it is surjective;
\item \label{P:mapsO2} it is order-preserving, i.e. if $(G_1,D_1)\geq (G_2,D_2)$ then $f_*(G_1,D_1)\geq f_*(G_2,D_2)$;
\item \label{P:mapsO3} it has the upper lifting property, i.e. for any $(G_1',D_1')\geq (G_2',D_2')$ in $\bO^d(\Gamma')$ and any $(G_2,D_2)\in f_*^{-1}(G_2',D_2')$, there exists 
$(G_1,D_1)\in f_*^{-1}(G_1',D_1')$ such that $(G_1,D_1)\geq (G_2,D_2)$. 
\end{enumerate}
The same properties hold true for the restriction $f_*:\bO^d_{\con}(\Gamma)\to \bO^d_{\con}(\Gamma')$.
\end{prop}
\begin{proof}
Part \eqref{P:mapsO1}: let $(\Gamma'\setminus S',D')\in \bO^d(\Gamma')$. Consider the subset $S:=f^E(S')\subseteq E(\Gamma)$. Pick a subset $W\subseteq V(\Gamma)$ such that 
$f_V$ induces a bijection between $W$ and $V(\Gamma')$ and define $D\in \Div(\Gamma)$ by 
$$
D_v:=
\begin{cases}
D'_{f_V(v)} & \text{ if } v\in W, \\
0 & \text{ otherwise.}
\end{cases}
$$
Notice that $(\Gamma\setminus S,D)\in \bO^d(\Gamma)$ since $|D|=|D'|=d-|S'|=d-|S|$, and we have that  $f_*(\Gamma\setminus S,D)=(\Gamma'\setminus S',D')$ by construction. 

Part \eqref{P:mapsO2}: let $(G_1,D_1)\geq (G_2,D_2)$ in $\bO^d(\Gamma)$. By definition of the poset, we have that 
$$D_1=D_2+\D(\O) \text{ for some orientation $\O$ of } G_1-G_2.$$
The orientation $\O$ induces an orientation $\O'$ of $f_*(G_1)-f_*(G_2)$ with respect to which we have that 
$$f_*(D_1)=f_*(D_2)+\D(\O').$$ 
Hence we get that 
$$f_*(G_1,D_1)=(f_*(G_1), f_*(D_1))\geq (f_*(G_2), f_*(D_2))=f_*(G_2,D_2).$$ 

Part \eqref{P:mapsO3}: let $(G_1',D_1')\geq (G_2',D_2')$ in $\bO^d(\Gamma')$. By definition of the poset, we have that 
$$D_1'=D_2'+\D(\O') \text{ for some orientation $\O'$ of } G_1'-G_2'.$$
Let $(G_2,D_2)\in f_*^{-1}(G_2',D_2')$. Define $S_1:=E(G_2)-f^E(E(G_2')-E(G_1'))$ and set $G_1:=\Gamma\setminus S_1$. 
Since $f^E$ induces a bijection between $E(G_2')-E(G_1')$ and $E(G_2)-E(G_1)$, the orientation $\O'$ on $G_1'-G_2'$ induces an orientation $\O$ on $G_1-G_2$. Define $D_1:=D_2+\D(\O)$. By construction, we have that $(G_1,D_1)\geq (G_2,D_2)$
and 
$$
f_*(G_1,D_1)=(\Gamma'\setminus (f^E)^{-1}(S_1),f_*(D_2+\D(\O)))=(G_1', D_2'+\D(\O')=D_1'),
$$
which concludes the proof.
\end{proof}

In what follows, we will be interested in the following types of subsets of $\bO_{\con}^d(\Gamma)$.

\begin{defi}\label{D:upper-sN}
Let $\Gamma$ be a connected graph and let $d\in \Z$. Let $\cP$ be an upper subset of  $\bO_{\con}^d(\Gamma)$, i.e. a subset such that 
$$(G_1,D_1)\geq (G_2,D_2) \text{ and } (G_2,D_2)\in \cP \Rightarrow (G_1,D_1)\in \cP.$$
We say that $\cP$ is 
\begin{enumerate}
\item an \emph{upper subset of sN-type} of degree $d$ on $\Gamma$ if, for any $G\in \SSc(\Gamma)$, the composition
$$
\cP(G):=\{D\in \Div^{d-|E(G)^c|}(G)\: : (G,D)\in \cP\}\subset \Div^{d-|E(G)^c|}(G)\twoheadrightarrow \Pic^{d-|E(G)^c|}(G)
$$
is a bijection.
\item an \emph{upper subset of numerical sN-type} of degree $d$ on $\Gamma$ if, for any $G\in \SSc(\Gamma)$, we have that 
$$
|\cP(G)|=c(G). 
$$
\item an \emph{upper subset of N-type} of degree $d$ on $\Gamma$ if the composition
$$
\cP(\Gamma):=\{D\in \Div^{d}(\Gamma)\: : (\Gamma,D)\in \cP\}\subset \Div^{d}(\Gamma)\twoheadrightarrow \Pic^{d}(\Gamma)
$$
is a bijection.
\item an \emph{upper subset of numerical N-type} of degree $d$ on $\Gamma$ if we have that 
$$
|\cP(\Gamma)|=c(\Gamma). 
$$
\end{enumerate}
\end{defi}
An upper subset $\cP$ of $\bO_{\con}^d(\Gamma)$ of sN-type is the same a degree-$d$ stability condition on $\Gamma$ in the sense of \cite[Def. 4.3]{PT2}.

Note that, since the cardinality of $\Pic^e(G)$ is equal to the complexity $c(G)$ of $G$ for any connected graph and any $e\in \Z$, we have the following implications
$$
\xymatrix{
\left\{\text{Upper subsets of sN-type}\right\} \ar@{=>}[r] \ar@{=>}[d]  & \left\{\text{Upper subsets of N-type}\right\} \ar@{=>}[d] \\
\left\{\text{Upper subsets of numerical sN-type}\right\} \ar@{=>}[r] & \left\{\text{Upper subsets of numerical N-type}\right\}.
}
$$
A few words about the terminology: the term N-type comes from the relation between fine compactified Jacobians of N-type and of N\'eron-type  (see Definition \ref{D:Neron} and Proposition \ref{P:Neron-N}), the term sN-type means \emph{strongly} N-type, the term \emph{numerical} N-type (resp. sN-type) refers to the fact that it is a numerical version of the property of being  N-type (resp. sN-type).


\subsection{General V-stability conditions and V-upper subsets}

In this subsection, we introduce general V-stability conditions on graphs and their associated V-subset.

Let $\Gamma$ be a connected graph. A subset $W\subseteq V(\Gamma)$ is said to be \emph{connected} if the induced subgraph $\Gamma[W]$ is connected, and it is said to be 
\emph{biconnected} if both $W$ and $W^c$ are connected. A subset $W\subseteq V(\Gamma)$ is said to be non-trivial if $W\neq \emptyset, V(\Gamma)$. Note that a non-trivial subset 
$W\subseteq V(\Gamma)$ is biconnected if and only if $E(\Gamma[W],\Gamma[W^c])$ is a bond, i.e. a minimal cut, of $\Gamma$. 

\begin{defi}\label{D:Vstab}
A \textbf{general\footnote{The $V$-stability conditions that we define in this paper are called general since they are the "general" case of a broader notion of $V$-stability conditions, that we will define in \cite{FPV1} and \cite{FPV2}.} $V$-stability condition} (or simply a general V-stability) of degree $d\in \Z$ on $\Gamma$ is an assignment of integers 
$$
\frak n=\{\frak n_W \: : \: W\:  \text{ is biconnected and non-trivial}\}
$$
satisfying the following properties:
\begin{enumerate}
\item for any $W\subseteq V(\Gamma)$ biconnected and non-trivial, we have 
\begin{equation}\label{E:Vstab1}
\frak n_W+\frak n_{W^c}=d+1-\val(W);
\end{equation}
\item  for any disjoint $W_1,W_2\subseteq V(\Gamma)$ such that $W_1,W_2$ and $W_1\cup W_2$ are biconnected and non-trivial, we have that 
\begin{equation}\label{E:Vstab2}
-1\leq \frak n_{W_1\cup W_2}-\frak n_{W_1}-\frak n_{W_2}-\val(W_1,W_2)\leq 0.
\end{equation}
\end{enumerate}
We set $|\frak n|:=d$ and we call it the degree of the general V-stability $\frak n$. 
\end{defi}

A source of general V-stability comes from general numerical polarizations on the graph.

\begin{example}\label{Ex:phi}
Let $\phi$ be a  \emph{general numerical polarization} on $\Gamma$ of degree $d\in \Z$, i.e. a real divisor $\phi\in \Div(\Gamma)_{\R}:=\Div(\Gamma)\otimes \R=\R^{V(\Gamma)}$ whose degree $|\phi|$ is equal to $d$, such that 
\begin{equation}\label{E:gen-q}
\phi_W-\frac{\val(W)}{2}=\sum_{v\in W} \phi_v-\frac{\val(W)}{2}\not\in \Z \: \text{ for every  biconnected non-trivial  } W\subset V(\Gamma).
\end{equation}
 Then the collection 
 \begin{equation}\label{E:V-q}
 \frak n(\phi):=\{\frak n(\phi)_W:=\lceil \phi_W-\frac{\val(W)}{2} \rceil  \: \text{ for any  biconnected and non-trivial } W\}
 \end{equation}
 is a general V-stability condition of degree $d$, called the general V-stability condition associated to $\phi$. A general V-stability condition $\frak n$ is called \emph{classical} if it is of the form $\frak n=\frak n(\phi)$ for some 
 general numerical polarization $\phi$.
 
This follows from the two identities 
$$
\begin{sis}
& \left(\phi_W-\frac{\val(W)}{2} \right)+\left(\phi_{W^c}-\frac{\val(W^c)}{2} \right)=d-\val(W),\\
& \left(\phi_{W_1\cup W_2}-\frac{\val(W_1\cup W_2)}{2} \right)-\left(\phi_{W_1}-\frac{\val(W_1)}{2} \right)-\left(\phi_{W_2}-\frac{\val(W_2)}{2} \right)-\val(W_1,W_2)=0,\\
\end{sis}
$$ 
by passing to the upper integral parts and using condition \eqref{E:gen-q}.

Consider the arrangement of hyperplanes in the affine space $\Div^d(\Gamma)_{\R}$ of numerical polarizations of degree $d$ (see \cite[\S 7]{OS} and \cite[\S 3]{MRV}):
\begin{equation}\label{E:arr-hyper}
\cA_{\Gamma}^d:=\left\{\phi_W-\frac{\val(W)}{2}=n\right\}_{W\subset V(\Gamma), n\in \Z}
\end{equation}
where $W$ varies among the biconnected non-trivial subsets of $V(\Gamma)$. We get an induced wall and chamber decomposition of $\Div^d(\Gamma)_{\R}$ such that 
\begin{itemize}
\item $\phi\in \Div^d(\Gamma)_{\R}$ is a general numerical polarization if and only if $\phi$ does not lie on any wall;
\item two general numerical polarizations $\phi, \phi'$ belong to the same chamber if and only if  $\frak n(\phi)=\frak n(\phi')$.
\end{itemize}
In other words, the set of chambers induced by $\cA_{\Gamma}^d$ is the set of classical general V-stability conditions of degree $d$.
\end{example}

However, not every general V-stability condition comes from a general numerical polarization, as the following Example (which is a re-adaptation of \cite[Ex. 6.15]{PT1} to our setting) shows.

\begin{example}\label{Ex:msa-not}
Let $\Gamma$ be the graph with vertices $V(\Gamma)=\{1,\ldots, 6\}$ and edge set $E(\Gamma)=\{\ov{13},\ov{35}, \ov{24}, \ov{45}, \ov{16}, \ov{26}, \ov{14}, \ov{23}\}$, 
where $\ov{ij}$ denotes an edge joining $i$ and $j$. Note that $\Gamma$ is a connected graph of genus $g(\Gamma)=3$. 
Consider the V-stability condition $\frak n$ of degree $4$ on $\Gamma$ given by
$$
\frak n_W-g(\Gamma[W]):=
\begin{cases}
0 & \text{ if } |W|\leq 2 \text{ or }  W=\{1,4,5\}, \{2,3,5\}, \{3,4,5\}, \{2,4,6\}, \{1,3,6\},\\
1 & \text{ if } |W|\geq 4 \text{ or }  W=\{1,3,5\}, \{2,4,5\}, \{2,3,6\}, \{1,4,6\}, \{1,2,6\}.\\contradiction
\end{cases}
$$
We claim that there exists no general numerical polarization $\phi$ of degree $4$ on $\Gamma$ such that $\frak n(\phi)=\frak n$.

Indeed, by contradiction, assume that there exists such a numerical polarization. Then, on the one hand, we have that 
$$
\begin{sis}
0=\frak n_{145}=\lceil \phi_{145}-\frac{\val(145)}{2}\rceil=\lceil \phi_{145}-\frac{4}{2}\rceil\\
1=\frak n_{135}=\lceil\phi_{135}-\frac{\val(135)}{2}\rceil= \lceil\phi_{135}-\frac{4}{2}\rceil
\end{sis}
\Rightarrow \phi_3>\phi_4,
$$
and, on the other hand, we have that 
$$
\begin{sis}
0=\frak n_{235}=\lceil \phi_{235}-\frac{\val(235)}{2}\rceil=\lceil \phi_{235}-\frac{4}{2}\rceil\\
1=\frak n_{245}=\lceil\phi_{245}-\frac{\val(245)}{2}\rceil= \lceil\phi_{245}-\frac{4}{2}\rceil
\end{sis}
\Rightarrow \phi_4>\phi_3,
$$
which is the desired contradiction. 
\end{example}

A general V-stability on a connected graph  induces a general V-stability on each connected spanning subgraph, as we show in the following

\begin{lemdef}\label{LD:V-subgr}
Let $\frak n$ be a general V-stability of degree $d$ on a connected graph $\Gamma$ and let $\Gamma\setminus S$ be a connected spanning subgraph of $\Gamma$. 
The general V-stability induced by $\frak n$ on $\Gamma\setminus S$ is given by 
\begin{equation}\label{E:V-subgr}
 \frak n(\Gamma\setminus S):=\{\frak n(\Gamma\setminus S)_W:=\frak n_W-e_S(W)  \: : W\:  \text{ is biconnected and non-trivial in } \Gamma\setminus S\}.
 \end{equation}
We have that $| \frak n(\Gamma\setminus S)|=|\frak n|-|S|$. 
\end{lemdef}
\begin{proof}
First of all, observe that if $W\subset V(\Gamma\setminus S)=V(\Gamma)$ is biconnected  in $\Gamma\setminus S$ then it is also biconnected in $\Gamma$; hence the definition of  
$ \frak n(\Gamma\setminus S)$ makes sense. We now check the two conditions for $ \frak n(\Gamma\setminus S)$ to be a V-stability. 

First of all, using that $\frak n$ satisfies \eqref{E:Vstab1} with $d=|\frak n|$ and the relations \eqref{E:val-valS} and \eqref{add-val}, 
we have that 
$$\frak n(\Gamma\setminus S)_W+\frak n(\Gamma\setminus S)_{W^c}=\frak n_W-e_S(W)+\frak n_{W^c}-e_S(W^c)=d+1-\val_{\Gamma}(W)-e_S(W)-e_S(W^c)=$$
$$=d+1-\val_{\Gamma}(W)-|S|+\val_S(W)= d+1-|S|-\val_{\Gamma\setminus S}(W).$$ 

Second, using again  the relations \eqref{E:val-valS} and \eqref{add-val}, we have that 
$$
\frak n(\Gamma\setminus S)_{W_1\cup W_2}-\frak n(\Gamma\setminus S)_{W_1}- \frak n(\Gamma\setminus S)_{W_2}-\val_{\Gamma\setminus S}(W_1,W_2)=
$$
$$
=\frak n_{W_1\cup W_2}-e_S(W_1\cup W_2)-\frak n_{W_1}+e_S(W_1)- \frak n_{W_2}+e_S(W_2)-\val_{\Gamma\setminus S}(W_1,W_2)=
$$
$$
=\frak n_{W_1\cup W_2}-\val_S(W_1,W_2)-\frak n_{W_1}- \frak n_{W_2}-\val_{\Gamma\setminus S}(W_1,W_2)=
$$
$$
=\frak n_{W_1\cup W_2}-\frak n_{W_1}- \frak n_{W_2}-\val_{\Gamma}(W_1,W_2)
$$
and we conclude by the fact that $\frak n$ satisfies \eqref{E:Vstab2}. 
\end{proof}

We now define the push-forward of a general V-stability along a morphism of graphs.

\begin{lemdef}\label{LD:V-func}
Let $\frak n$ be a general V-stability of degree $d$ on a connected graph $\Gamma$. Let $f:\Gamma\to \Gamma'$ be a morphism of graphs and 
denote by $f_V:V(\Gamma)\twoheadrightarrow V(\Gamma')$ the induced surjection on vertices. 
The \emph{push-forward} of $\frak n$ along $f$ is the general V-stability on $\Gamma'$ given by 
\begin{equation}\label{E:V-func}
f_*( \frak n):=\{f_*(\frak n)_W:=\frak n_{f_V^{-1}(W) }  \: : W\subset V(\Gamma') \:  \text{  biconnected and non-trivial}\}.
 \end{equation}
We have that $| f_*(\frak n)|=|\frak n|$. 
\end{lemdef}
\begin{proof}
First of all, if $W\subset V(\Gamma')$ is biconnected and non-trivial, then also $f_V^{-1}(W)\subset V(\Gamma)$ is biconnected and non-trivial; hence the definition of $f_*(\frak n)$ make sense.

The fact that $f_*\frak n$ is a general V-stability on $\Gamma'$ of degree equal to $|\frak n|$ follows from the fact that $f_V^{-1}(W^c)=f_V^{-1}(W)^c$ and $\val_{\Gamma'}(W_1,W_2)=\val_{\Gamma}(f_V^{-1}(W_1), f_V^{-1}(W_2))$, for any $W\subset V(\Gamma')$ and any pairwise disjoint $W_1,W_2\subset V(\Gamma')$. 
\end{proof}

Each general V-stability condition of degree $d$ on $\Gamma$ determines a subset of $\bO^d(\Gamma)$. 


\begin{defi}\label{D:stabV}
Let $\frak n$ be a general V-stability condition of degree $d$ on a graph $\Gamma$. 
The \textbf{V-subset} associated to $\frak n$ is:
$$\cP_{\frak n}:=\{(\Gamma\setminus S,D)\in \bO^d(\Gamma)\: : \: D_W+e_S(W) \geq \frak n_W \: \text{ for any non-trivial biconnected  } W\subset V(\Gamma)\}.$$
 The elements $(G, D)$ are called \emph{stable} with respect to $\frak n$. For any $G\in \SS(\Gamma)$, we set 
 $$
 \cP_{\frak n}(G)=\{D\in \Div^{d-|E(G)^c|}(\Gamma)\: : (G,D)\in \cP_{\frak n}\}. 
 $$
\end{defi}

We make some comments on the above Definition.

\begin{remark}\label{sym-inequ}
Let $D\in  \cP_{\frak n}(G)$ and $W$ a non-trivial biconnected subset of $V(\Gamma)$.
By applying Definition \ref{D:stabV} to $W^c\subset V(\Gamma)$ and using
\eqref{E:Vstab1}, we get that
\begin{equation*}
d-|S|-D_W+e_S(W^c)=D_{W^c}+e_S(W^c) \geq \frak n_{W^c}=d+1-\frak n_W-\val(W),
\end{equation*}
which, using \eqref{add-val}, is equivalent to 
\begin{equation}\label{E:stabV-upp}
D_W+e_S(W)\leq \frak n_W-1+\val_{\Gamma\setminus S}(W).
\end{equation} 
\end{remark}

\begin{remark}\label{R:connW}
Let $W$ a non-trivial connected subset of $V(\Gamma)$ and denote by 
$$\Gamma[W^c]=\Gamma[Z_1]\coprod \ldots \coprod \Gamma[Z_k]$$
 be the decomposition of $\Gamma[W^c]$ into connected components 
(so that $Z_i$ is non-trivial and biconnected for every $1\leq i \leq k$). 

If $D\in \cP_{\frak n}(\Gamma\setminus S)$, then we have that 
\begin{equation}\label{E:connW}
\begin{sis}
& \sum_{i=1}^k \frak n_{Z_i}\leq D_{W^c}+e_S(W^c)\leq \sum_{i=1}^k \frak n_{Z_i} -k +\val_{\Gamma\setminus S}(W^c),\\
& d-  \sum_{i=1}^k \frak n_{Z_i}+k -\val_{\Gamma}(W)\leq D_W+e_S(W)\leq d-\val_S(W)-\sum_{i=1}^k \frak n_{Z_i}.
\end{sis}
\end{equation}

Indeed, the first inequality in \eqref{E:connW} follows by summing the $k$-inequalities
$$
\frak n_{Z_i}\leq D_{Z_i}+e_S(Z_i)\leq \frak n_{Z_i}-1+\val_{\Gamma\setminus S}(Z_i).
$$

The second inequality is deduced from the first one using that (by \eqref{add-val})
$$
D_W+e_S(W)=d-|S|-D_{W^c}+e_S(W)=d-(D_{W^c}+e_S(W^c))-\val_S(W).
$$

\end{remark}

The V-subsets behaves well with respect to the push-forward on the poset of divisors on spanning subgraphs in \eqref{E:f*-O} and the push-forward of V-stabilities defined in Lemma-Definition \ref{LD:V-func}. 


\begin{prop}\label{P:stabV-fun}
Let $\frak n$ be a general V-stability on a connected graph $\Gamma$ and let $f:\Gamma\to \Gamma'$ be a morphism of graphs. 
Then the push-forward $f_*$ of \eqref{E:f*-O} defines a map 
$$
f_*:\cP_{\frak n}\to \cP_{f_*(\frak n)}.
$$
Equivalently, for any spanning subgraph $\Gamma\setminus S$ of $\Gamma$, we get a map of sets
$$
\begin{aligned}
f_*:\cP_{\frak n}(\Gamma\setminus S)& \longrightarrow \cP_{f_*(\frak n)}(f_*(\Gamma\setminus S)),\\
D & \mapsto f_*(D),\\
\end{aligned}
$$
where $f_*(\Gamma\setminus S):=\Gamma'\setminus (f^E)^{-1}(S)$. 
\end{prop}
\begin{proof}
Pick an element $D\in \cP_{\frak n}(\Gamma\setminus S)$ and let us show that $f_*(D)\in \cP_{f_*(\frak n)}(f_*(\Gamma\setminus S))$. 

First of all, the total degree of $f_*(D)$ is equal to 
$$
|f_*(D)|=f_*(D)_{V(\Gamma')}=D_{V(\Gamma)}+e_{S\cap (\Im f^E)^c}(V(\Gamma))=|\frak n|-|S|+|S\cap (\Im f^E)^c|
=|f_*(\frak n)|-|(f^E)^{-1}(S)|.
$$
 
 Second, if $W$ is a non-trivial and biconnected subset of $V(\Gamma')$, then using the definition of $f_*(\frak n)$ and the general V-stability of $D$, we get
$$
f_*({D})_W+e_{(f^E)^{-1}(S)}(W)=D_{f_V^{-1}(W)}+e_{S\cap (\Im f^E)^c}(f_V^{-1}(W))+e_{(f^E)^{-1}(S)}(W)= 
$$ 
$$
=D_{f_V^{-1}(W)}+e_{S}(f_V^{-1}(W))\geq \frak n_{f_V^{-1}(W)}=f_*(\frak n)_W,
$$
which concludes the proof. 
\end{proof}

We now study the V-subset $\cP_{\frak n}\subset \bO^d(\Gamma)$ associated to a general V-stability condition $\frak n$ of degree $d$.

\begin{prop}\label{P:Bn-upper}
The V-subset $\cP_{\frak n}\subset \bO^d(\Gamma)$ is an upper set, i.e. if $(G_1,D_1)\geq (G_2,D_2)$ in $\bO^d(\Gamma)$ and $(G_2,D_2)\in \cP_{\frak n}$, then 
$(G_1,D_1)\in \cP_{\frak n}$. 

Equivalently, for any $T\subset S\subseteq E(\Gamma)$ we have that 
$$
\cP_{\frak n}(\Gamma\setminus S)+\D(S-T)\subseteq \cP_{\frak n}(\Gamma\setminus T).
$$
\end{prop} 
\begin{proof}
Let $D\in \cP_{\frak n}(\Gamma\setminus S)$ and $E\in \D(S\setminus T)$. We want to show that $D+E\in \cP_{\frak n}(\Gamma\setminus T)$.

First, note that the total degree of $D+E$ is equal to 
$$
|D+E|=|D|+|E|=|\frak n|-|S|+|S|-|T|=|\frak n|-|T|.
$$

Second, for any non-trivial biconnected $W\subset V(\Gamma)$, we have that 
$$
(D+E)_W+e_T(W)= D_W+E_W+e_T(W)\geq D_W+e_{S\setminus T}(W)+e_T(W)=D_W+e_S(W)\geq \frak n_W,
$$
where we have used that $E_W\geq e_{S\setminus T}(W)$ since if an edge of $S\setminus T$ belongs to $ E(\Gamma[W])$ then both its end vertices belong $W$. 
\end{proof}

\begin{prop}\label{P:Bn-card}
Let $\frak n$ be a general V-stability condition on $\Gamma$ of degree $d$. 
\begin{enumerate}
\item \label{P:Bn-card1} If $\cP_{\frak n}(\Gamma\setminus S)\neq \emptyset$ then $\Gamma\setminus S$ is connected. In particular, $\cP_{\frak n}\subset \bO_{\con}^d(\Gamma)$. 
\item \label{P:Bn-card2} If $\Gamma \setminus S$ is connected, then the composed map
\begin{equation}\label{E:map-pi}
\pi(\Gamma\setminus S):\cP_{\frak n}(\Gamma\setminus S)\subseteq \Div^{d-|S|}(\Gamma\setminus S) \twoheadrightarrow \Pic^{d-|S|}(\Gamma\setminus S)
\end{equation}
is injective. 
\end{enumerate}
\end{prop}
We will prove later that $\cP_{\frak n}$ is an upper subset of sN-type (see Theorem \ref{T:main-comb}\eqref{T:main-comb2}), which will then imply that the above map \eqref{E:map-pi} is a bijection.

\begin{proof}
Part \eqref{P:Bn-card1}: by contradiction, assume that $\Gamma\setminus S$ is not connected and $\cP_{\frak n}(\Gamma\setminus S)\neq \emptyset$.
This means that there exist $D\in \cP_{\frak n}(\Gamma\setminus S)$ and a non-trivial subset $W\subset V(\Gamma)$ such that $\val_{\Gamma\setminus S}(W)=0$.
Up to passing to a connected component of $\Gamma[W]$ or $\Gamma[W^c]$, we can also assume that  $W$ is biconnected.  
By the Definition \ref{D:stabV} and Remark \ref{sym-inequ}, we get that
$$\frak n_W \leq D_W+e_S(W) \leq \frak n_W-1+\val_{\Gamma\setminus S}(W)= \frak n_W-1,$$
which is absurd.

\item Part \eqref{P:Bn-card2}: by contradiction, assume that there exist $D\neq E\in \cP_{\frak n}(\Gamma\setminus S)$
such that $\pi(D)=\pi(E)$. This is equivalent to the existence of an element $F\in \Div^0(\Gamma\setminus S)$ such that $\Delta^{\Gamma\setminus S}(F)=D-E$.

Since $D, E\in \cP_{\frak n}(\Gamma\setminus S)$, Remark \ref{R:connW} implies that,
 for any non-trivial connected subset $W\subset V(\Gamma)$ such that $\Gamma[W^c]=\Gamma[Z_1]\coprod \ldots \coprod \Gamma[Z_k]$ with $Z_i$ connected (and hence biconnected), we have that 
$$ D_W-E_W \leq \left(-e_S(W)+d-\val_{S}(W)-\sum_{i=1}^k \frak n_{Z_i}\right)-
\left(-e_S(W)+d-\sum_{i=1}^k \frak n_{Z_i}+k-\val_{\Gamma}(W)\right)=$$
\begin{equation}\label{eq1}
=\val_{\Gamma\setminus S}(W)-k<\val_{\Gamma\setminus S}(W).
\end{equation}

Consider now the (non-empty) subset
\begin{equation*}
V_0:=\{v\in V(\Gamma)=V(\Gamma\setminus S)\: :\: F_v=\min_{w\in V(\Gamma)}
F_w:=l\}\subseteq V(\Gamma)=V(\Gamma\setminus S).
\end{equation*}
If $V_0=V(\Gamma\setminus S)$ then $F$ is the identically zero divisor in $\Gamma\setminus S$, and therefore $0=\Delta^{\Gamma\setminus S}(F)=D-E$,
which contradicts the hypothesis that $D\neq E$. Therefore $V_0$ is a proper (and non-empty) subset of $V(\Gamma\setminus S)$.
Pick a non-trivial subset  $W\subseteq V_0$ such that $\Gamma[W]$ is connected component of $\Gamma[V_0]$, and denote by $k$ the number of connected components of $\Gamma[W^c]$. 

From the definition (\ref{Lapla}), using that $F_v\geq l$ for any $v\in V(\Gamma\setminus S)$ with equality
if $v\in V_0$ and that $\val_{\Gamma\setminus S}(W,W^c\setminus V_0^c)=0$, we compute 
\[\Delta^{\Gamma\setminus S}(F)_{V_0}=\sum_{v\in W}\left[-F_v \cdot \val_{\Gamma\setminus S}(v)+\sum_{v\neq w\in W} F_w \val_{\Gamma\setminus S}(v,w)+\sum_{v\neq w\in W^c} F_w \val_{\Gamma\setminus S}(v,w)\right]=\]
\[=\sum_{v\in W}\left[-l \cdot \val_{\Gamma\setminus S}(v)+\sum_{v\neq w\in W} l \cdot \val_{\Gamma\setminus S}(v,w)+\sum_{w\in V_0^c} F_w \val_{\Gamma\setminus S}(v,w)\right]=\]
\[\geq \sum_{v\in W}\left[-l \cdot \val_{\Gamma\setminus S}(v)+ l\cdot \val_{\Gamma\setminus S}(v,W\setminus \{v\})+ \sum_{w\in V_0^c} (l+1) \val_{\Gamma\setminus S}(v,w)\right]=\]
\[=\sum_{v\in W}\left[-l \cdot \val_{\Gamma\setminus S}(v, W^c)+(l+1)\val_{\Gamma\setminus S}(v,V_0^c)\right]=\sum_{v\in W}\left[-l \cdot \val_{\Gamma\setminus S}(v, V_0^c)+(l+1)\val_{\Gamma\setminus S}(v,V_0^c)\right]=\]
\begin{equation}\label{eq2}
=\sum_{v\in W}\val_{\Gamma\setminus S}(v,V_0^c)=\val_{\Gamma\setminus S}(W,V_0^c)=\val_{\Gamma\setminus S}(W).
\end{equation}
Using the fact that $\Delta^{\Gamma\setminus S}(F)=D-E$, the above inequality (\ref{eq2}) contradicts the inequality (\ref{eq1}), and the proof is complete. 
\end{proof}

We now introduce an equivalence relation on  general V-stability conditions on a graph, and we will show that there are only finitely many equivalence classes. 


\begin{defi}\label{D:eq-Vstab}
Let $\Gamma$ be a connected graph. 
\begin{enumerate}
\item Let $\frak n$ be a general V-stability condition of degree $d$ on $\Gamma$.
For any divisor $D$ of degree $e$ on $\Gamma$, then 
$$
\frak n+D=\{(\frak n+D)_W:=\frak n_W+D_W\} .
$$ 
is a general V-stability condition of degree $d+e$ on $\Gamma$. 
\item Given two general V-stability conditions $\frak n_1, \frak n_2$  on a graph $\Gamma$, we say that  $\frak n_1$ and $\frak n_2$ are \emph{equivalent by translation} if there exists a divisor $D$ on $\Gamma$ such that  $\frak n_2=\frak n_1+D$. 
\end{enumerate}
\end{defi}

\begin{remark}\label{R:phi-tras}
Consider the hyperplane arrangement $\cA_{\Gamma}^d$ of \eqref{E:arr-hyper} on the affine space $\Div^d(\Gamma)_{\R}$ of numerical polarizations of degree $d$.
This hyperplane arrangement is invariant under the action of the lattice $\Div^0(\Gamma)$ on $\Div^d(\Gamma)_{\R}$  by translation and it defines a toric arrangement $\ov{\cA}_{\Gamma}^d$
of the affine real torus $\Div^d(\Gamma)_{\R}/\Div^0(\Gamma)$. 
The set of chambers induced by $\ov{\cA}_{\Gamma}^d$ on $\Div^d(\Gamma)_{\R}/\Div^0(\Gamma)$ is the set of classical general V-stability conditions of degree $d$ up to translation. 
\end{remark}

\begin{example}\label{Ex:msa}
\noindent 
\begin{enumerate}
\item \label{Ex:msa1} Let $\Gamma$ be a tree. Then there is a unique equivalence class of general V-stability conditions on $\Gamma$.

Indeed, the pairs of complementary biconnected sets of $\Gamma$ are all constructed in the following way: given an edge $e$ of $\Gamma$, then $\Gamma\setminus e$ has two connected components which are equal to $\Gamma[W_e]$ and  $\Gamma[W_e^c]$ for a pair $W_e,W_e^c$ of complementary biconnected sets of $\Gamma$ and any such pair arises in this way. 
Hence, arguing by induction on the number of edges, one can show that there exists a (unique) divisor $D$ on $\Gamma$ such that $D_{W_e}=\frak n_{W_e}$ and $D_{W_e^c}=\frak n_{W_e^c}$ for any edge $e$. This shows that $\frak n$ is equivalent to the identically zero general V-stability condition. 

\item \label{Ex:msa2} Let $\Gamma$ be a graph such that every non-trivial subset $W\subset V(\Gamma)$ is biconnected (which happens if and only if for any two vertices there exists at least one edge joining them). 
Then every general V-stability condition is equivalent by translation to a (unique) general V-stability condition $\frak n$ such that $\frak n_v=g(\Gamma[v])$ for every $v\in V(\Gamma)$. 
Such a general V-stability condition $\frak n$  will satisfy 
$$
0\leq \frak n_W-g(\Gamma[W])\leq |W|-1.
$$
\end{enumerate}
\end{example}

We now show that, choosing a spanning tree, we can produce a canonical representative  of each equivalence class by translation of general V-stability condition on a given graph.

\begin{prop}\label{P:cano-msa}
Let $\Gamma$ be a connected graph and choose a spanning tree $T$ of $\Gamma$. Then any general V-stability condition $\frak n$ on $\Gamma$ is equivalent by translation to a unique general V-stability condition $\wt{\frak n}$ of degree $g(\Gamma)$ such that
\begin{equation}\label{E:norm-msa}
-\val_T(W)+1\leq \wt{\frak n}_W-g(\Gamma[W])\leq \val_T(W)-1.
\end{equation}
\end{prop}
\begin{proof}
Note that there is an inclusion 
$$
\{\text{biconnected sets of } T\} \hookrightarrow \{\text{biconnected sets of } \Gamma\} 
$$
whose image consists of the biconnected sets $W$ of $\Gamma$ such that $\val_T(W)=1$. 

Observe that for any  biconnected subset $W$ of $\Gamma$ we have that 
$$
\frak n_W-g(\Gamma[W])+\frak n_{W^c}-g(\Gamma[W^c])=\frak n_W+\frak n_{W^c}-g(\Gamma)+\val(W)-1=d-g(\Gamma),
$$
where we have used \eqref{E:Vstab1} in the last equality.  
Therefore,  arguing as in Example \ref{Ex:msa}\eqref{Ex:msa1}, there exists a unique divisor $D\in \Div(T)=\Div(\Gamma)$ such that $\wt{\frak n}:=\frak n+D$ is such that 
$$
 \wt{\frak n}_W=g(\Gamma[W]) \text{ for any biconnected set } W \text{ of } \Gamma \text{ such that } \val_T(W)=1.
$$
In particular, $|\wt{\frak n}|=g(\Gamma)$, which implies that $\wt{\frak n}_W-g(\Gamma[W])=-(\wt{\frak n}_{W^c}-g(\Gamma[W^c]))$. Hence, \eqref{E:norm-msa} for $W$ is equivalent to \eqref{E:norm-msa} for $W^c$. 

We now prove that $\wt{\frak n}$ satisfies \eqref{E:norm-msa} by induction on $\val_T(W)$. The case $\val_T(W)=1$ follows by the construction of $\wt{\frak n}$.  
Assume that $\val_T(W)>1$, or equivalently that $W$ is not a biconnected set of $T$. Hence, up to exchanging $W$ with $W^c$, 
we can assume that $T[W]$ is not connected and we write $T[W]=T[W_1]\coprod \ldots \coprod T[W_k]$ for its decomposition into connected components.  Then we have that 
$$\val_T(W)=\sum_{i=1}^k \val_T(W_i)\quad  \text{ and } \: \val_T(W_i)\geq 1 \text{ for each } 1\leq i \leq k.$$
Since $\Gamma[W]$ is connected, there exists a non-trivial decomposition $W=Z_1\coprod Z_2$ such that $Z_1$ and $Z_2$ are unions of some of the subsets $W_i$ and 
$\Gamma[Z_1]$ and $\Gamma[Z_2]$ are connected. We have that 
\begin{equation}\label{E:sum-val}
\val_T(W)=\val_T(Z_1)+\val_T(Z_2) \quad  \text{ and } \: \val_T(Z_i)\geq 1 \text{ for each } i=1,2.
\end{equation}
Moreover, $\Gamma[Z_1^c]=\Gamma[Z_2\cup W^c]$ is connected since $\Gamma[W^c]$ is connected, $\Gamma[Z_2]$ is connected and $e_T(Z_2,W^c)=\val_T(Z_2)\geq 1$. Similarly, also $\Gamma[Z_2^c]$ is connected. By putting everything together, we get that $Z_1$ and $Z_2$ are biconnected subsets of $\Gamma$ such that $\val_T(Z_1),\val_T(Z_2)<\val_T(W)$. Hence, we can apply our inductive hypothesis and deduce that 
$$
-\val_T(Z_i)+1\leq \wt{\frak n}_{Z_i}-g(\Gamma[Z_i])\leq \val_T(Z_i)-1 \quad \text{ for } i=1,2.
$$
By applying \eqref{E:Vstab2} to $Z_1$, $Z_2$ and $W=Z_1\cup Z_2$ and using \eqref{E:sum-val} and the formula $g(\Gamma[W])=g(\Gamma[Z_1])+g(\Gamma[Z_2])+\val(Z_1,Z_2)-1$,  we get 
$$
\begin{aligned}
& \wt{\frak n}_W-g(\Gamma[W])\leq \wt{\frak n}_{Z_1}+\wt{\frak n}_{Z_2}+\val(Z_1,Z_2)-g(\Gamma[W])= \\
& =\wt{\frak n}_{Z_1}-g(\Gamma[Z_1])+\wt{\frak n}_{Z_2}-g(\Gamma[Z_2])+1
\leq \val_T(Z_1)+\val_T(Z_2)-1=\val_T(W)-1, \\
& \wt{\frak n}_W-g(\Gamma[W])\geq \wt{\frak n}_{Z_1}+\wt{\frak n}_{Z_2}+\val(Z_1,Z_2)-1-g(\Gamma[W])= \\
& =\wt{\frak n}_{Z_1}-g(\Gamma[Z_1])+\wt{\frak n}_{Z_2}-g(\Gamma[Z_2])
\geq -\val_T(Z_1)+1-\val_T(Z_2)+1>\val_T(W)+1. \\
\end{aligned}
$$
\end{proof}

\begin{cor}\label{C:fin-msa}
Let $\Gamma$ be a connected graph. Then there is a finite number of equivalence classes  of general V-stability conditions on $\Gamma$. 
\end{cor}
\begin{proof}
It follows from Proposition \ref{P:cano-msa} since there is a finite number of general V-stability conditions on $\Gamma$ satisfying \eqref{E:norm-msa}.
\end{proof}

\subsection{Generalized break divisors}\label{S:BD}

In this subsection, we introduce generalized break divisors and study their properties. The definition is inspired by \cite[\S 3.5]{Yuen}.

\begin{defi}\label{D:BD}
Let $\Gamma$ be a connected graph and consider a function (called \emph{a tree function of degree $d$}) 
$$
I:\ST(\Gamma)\to \Div^{d-b_1(\Gamma)}(\Gamma),
$$
where $\ST(\Gamma)$ is the set of all spanning trees of $\Gamma$.
Then the  \textbf{BD-set} (=generalized break divisors set)  with respect to the tree function $I$ is the subset $\BD_I\subset \bO^d_{\con}(\Gamma)$ defined by
$$\BD_I(G)=\bigcup_{G\geq T\in \ST(\Gamma)} I(T)+\D(G-T)\subset \Div^{d-|E(G)^c|}(\Gamma),$$
for any $G\in \SS_{\con}(\Gamma)$.
\end{defi}
Note that $\BD_I$ is an upper subset of $\bO^d_{\con}(\Gamma)$ by definition. The terminology comes from the fact that if $I$ is the identically zero function $\un 0$, then $\BD_{\un 0}(G)$ 
is the set of break divisors on $G$, see \cite{ABKS} and  \cite[Chap. 3]{Yuen}.

The BD-sets  behave well with respect to restriction to connected spanning graphs. 
Indeed, let $I:\ST(\Gamma)\to \Div^{d-b_1(\Gamma)}(\Gamma)$ and let $\Gamma\setminus S$ be  a connected spanning subgraph of $\Gamma$.
Then there is a natural inclusion 
\begin{equation}\label{E:ST-res}
\ST(\Gamma\setminus S)\hookrightarrow \ST(\Gamma),
\end{equation}
whose image is the subset of spanning trees  that do not contain edges of $S$. Hence, we can restrict the function $I$ to $\ST(\Gamma\setminus S)$ and we obtain a function
\begin{equation}\label{E:I-res}
I_{\Gamma\setminus S}:=I_{|\ST(\Gamma\setminus S)}: \ST(\Gamma\setminus S)\to \Div^{d-|S|-b_1(\Gamma\setminus S)}(\Gamma\setminus S).
\end{equation}

\begin{lemma}\label{L:inc-BD} 
With respect to the inclusion \eqref{E:incl-O}, we have that 
\begin{equation}\label{E:inc-BD}
\BD_{I_{\Gamma\setminus S}}= \BD_{I}\cap \bO_{\con}^{d-|S|}(\Gamma\setminus S).
\end{equation}
\end{lemma}
\begin{proof}
An element $I(T)+\D(G-T, \O)$ of $\BD_I(G)$ belongs to the subposet $\bO_{\con}^{d-|S|}(\Gamma\setminus S)\subset \bO_{\con}^{d}(\Gamma)$ if and only if 
$G$ does not contain any edge of $S$. This implies that $T\leq  G\leq \Gamma\setminus S$. In particular, $T$ is a spanning tree of $\Gamma\setminus S$ and then 
we conclude using that $I_{\Gamma\setminus S}(T)=I(T)$.  
\end{proof}

The BD-sets are also compatible in a weak sense with respect to morphisms of graphs that preserve the genus. 
Indeed, let $f:\Gamma\to \Gamma'$ be a morphism of graphs such that $b_1(\Gamma)=b_1(\Gamma')$. 
Then there is a natural inclusion 
\begin{equation}\label{E:ST-fun}
\begin{aligned}
f^*:\ST(\Gamma')& \hookrightarrow \ST(\Gamma),\\
\Gamma'\setminus S & \mapsto \Gamma\setminus f^E(S),
\end{aligned}
\end{equation}
whose image is the subset of spanning trees of $\Gamma$ that contain all the edges contracted by $f$. 
Hence, we can restrict the function $I$ to $\ST(\Gamma')$ and we obtain a function
\begin{equation}\label{E:I-fun}
\begin{aligned}
f_*(I): \ST(\Gamma')& \longrightarrow \Div^{d-b_1(\Gamma')}(\Gamma')\\
T &\mapsto f_*(I(f^*(T))).
\end{aligned}
\end{equation}
In the particular case where $f:\Gamma\to \Gamma/S$ is the contraction of some edge subset $S\subset E(\Gamma)$ that do not decrease the genus of $\Gamma$, then we set $f_*(I)=I_{\Gamma/S}$.  

\begin{lemma}\label{L:mor-BD} 
With respect to the map \eqref{E:f*-O}, we have that 
\begin{equation}\label{E:map-BD}
\BD_{f_*(I)}\subseteq f_*(\BD_{I}).
\end{equation}
\end{lemma}
\begin{proof}
An element of $\BD_{f_*(I)}$ can be written as $f_*(I)(\Gamma'\setminus S)+\D(S-T,\O)$ for some spanning tree $\Gamma'\setminus S$, some $T\subseteq S$ and some orientation of the edges of $S-T$. By definition of $f_*(I)$, we have that $f_*(I)(\Gamma'\setminus S)=f_*(I(\Gamma\setminus f^E(S)))$. Moreover, the orientation $\O$ induces an orientation $\wt \O$ of $f^E(S)-f^E(T)$
such that $f_*(\D(f^E(S)-f^E(T),\wt \O))=\D(S-T,\O)$. Therefore, we have that 
$$f_*(I)(\Gamma'\setminus S)+\D(S-T,\O)=f_*(I(\Gamma\setminus f^E(S))+\D(f^E(S)-f^E(T),\wt \O)),$$
which concludes the proof. 
\end{proof}

We now investigate which BD-sets are upper subsets of numerical N-type, in the sense of Definition \ref{D:upper-sN}. 

\begin{thm}\label{T:BD-Ntype}
Let $\Gamma$ be a connected graph and let $I:\ST(\Gamma)\to \Div^{d-b_1(\Gamma)}(\Gamma)$.  
\begin{enumerate}
\item \label{T:BD-Ntype1} (\cite[Thm. 3.5.1]{Yuen}) We have that $|\BD_I(\Gamma)|\geq c(\Gamma)$. 
\item \label{T:BD-Ntype2} If $|\BD_I(\Gamma)|=c(\Gamma)$ then 
\begin{enumerate}
\item \label{T:BD-Ntypea}  for any connected spanning subgraph $\Gamma\setminus S$, we have that $|\BD_I(\Gamma\setminus S)|=c(\Gamma\setminus S)$;
\item  \label{T:BD-Ntypeb} for any  morphism $f:\Gamma\to \Gamma'$ which preserves the genus, we have that $|\BD_{f_*(I)}(\Gamma')|=c(\Gamma')$. 
\end{enumerate}
\end{enumerate}
\end{thm}
\begin{proof}
Part \eqref{T:BD-Ntype1} is proved by Yuen \cite[Thm. 3.5.1]{Yuen}. We will give a slight modification of the proof of loc. cit., which also proves part \eqref{T:BD-Ntype2}. 

We will do induction on the number of edges that are neither loops nor disconnecting edges. If there are no such edges, then there is a unique spanning subgraph and the Theorem is obvious.

Suppose this is not the case, and take an edge $e$ which is neither a loop nor a separating edge. Consider the tree functions $I_{\Gamma\setminus e}$ and $I_{\Gamma/e}$ induced by $I$ on the connected spanning subgraph $\Gamma\setminus e$ (as in \eqref{E:I-res}) and on the contraction $\Gamma/e$ (as in \eqref{E:I-fun}). Denote by $u$ and $v$  the two distinct vertices that are incident to the edge $e$. There is a injective map
\begin{equation}\label{E:+u}
\begin{aligned}
+u: \BD_{I_{\Gamma\setminus e}}(\Gamma\setminus e)& \hookrightarrow \BD_I(\Gamma)\\
D&\mapsto D+u \\
\end{aligned}
\end{equation}
which is well-defined since if $D=I_{\Gamma\setminus e}(T)+\D(\Gamma\setminus e-T,\O)$ for some spanning tree $T$ of $\Gamma\setminus e$ and some orientation $\O$ of $\Gamma\setminus e-T$, then 
$$
D+u=I_{\Gamma\setminus e}(T)+\D(\Gamma\setminus e-T,\O)+u=I(T)+\D(\Gamma-T,\O\cup \overrightarrow{vu}) 
$$
where $\O\cup \overrightarrow{vu}$ is the orientation of $\Gamma-T$ obtained by adding the orientation $\O$ on $\Gamma\setminus e-T$ with the orientation of $e$ that goes from $v$ to $u$. 
Similarly, there is a well-defined injective map 
\begin{equation}\label{E:+v}
\begin{aligned}
+v: \BD_{I_{\Gamma\setminus e}}(\Gamma\setminus e)& \hookrightarrow \BD_I(\Gamma)\\
D&\mapsto D+v \\
\end{aligned}
\end{equation}

Moreover there is a map 
\begin{equation}\label{E:map-i}
\begin{aligned}
\iota: \BD_{I_{\Gamma/e}}(\Gamma/e)& \hookrightarrow \BD_I(\Gamma)\\
I_{\Gamma/e}(T)+\D(\Gamma/e-T,\O)& \mapsto I(T\cup e)+\D(\Gamma-(T\cup e),\O),
\end{aligned}
\end{equation}
where we used that $T\cup e$ is a spanning subgraph of $\Gamma$ and we lifted the orientation $\O$ from $\Gamma/e-T$ to $\Gamma-(T\cup e)$, since the contraction morphism $c_e:\Gamma\to \Gamma/e$ induces a bijection between $E(\Gamma)-E(T\cup e)$ and  $E(\Gamma/e)-E(T)$.  The map $\iota$ is injective since $(c_e)_*\circ \iota=\id$. 

We now modify slightly the map $\iota$ as follows. For any $D\in \BD_{I_{\Gamma/e}}(\Gamma/e)$, we define by $K(D)$ the biggest natural number such that 
$$\iota(D)-K(D)u+K(D)v\in \BD_I(\Gamma) \quad \text{ and } \:  \iota(D)-(K(D)+1)u+(K(D)+1)v\not \in \BD_I(\Gamma).$$ 
Such a natural number $K(D)$ exists since $\iota(D)\in \BD_I(\Gamma)$ and $\BD_I(\Gamma)$ is a finite set. Then we define the new map 
\begin{equation}\label{E:map-tilde}
\begin{aligned}
\wt{\iota}: \BD_{I_{\Gamma/e}}(\Gamma/e)& \hookrightarrow \BD_I(\Gamma)\\
D& \mapsto \iota(D)-K(D)u+K(D)v.
\end{aligned}
\end{equation}

We now claim that 
\begin{equation}\label{E:+u-tilde}
+u\coprod \wt{\iota}: \BD_{I_{\Gamma\setminus e}}(\Gamma\setminus e)\coprod \BD_{I_{\Gamma/e}}(\Gamma/e) \to \BD_I(\Gamma)\quad  \text{ is injective.}
\end{equation}
Indeed, since $+u$ and $\wt{\iota}$ are injective, then it is enough to prove that $\Im(+u)\cap \Im(\wt{\iota})=\emptyset$. By contradiction, assume there exists $D\in \BD_{I_{\Gamma/e}}(\Gamma/e) $ such that $\wt{\iota}(D)\in \Im(+u)$, or in other words that $\wt{\iota}(D)-u\in  \BD_{I_{\Gamma\setminus e}}(\Gamma\setminus e)$. Then, using the map $+v$ of \eqref{E:+v},  we would have 
that 
$$ \BD_I(\Gamma)\ni \wt{\iota}(D)-u+v=\iota(D)-(K(D)+1)u+(K(D)+1)v, $$
which contradicts the choice of $K(D)$. 

From \eqref{E:+u-tilde} and the induction hypothesis, we get 
\begin{equation}\label{E:ineq-c}
|\BD_I(\Gamma)|\geq |\BD_{I_{\Gamma\setminus e}}(\Gamma\setminus e)|+|\BD_{I_{\Gamma/e}}(\Gamma/e)|\geq c(\Gamma\setminus e)+c(\Gamma/e)=c(\Gamma),
\end{equation}
where the last equality follows from the fact that (see \eqref{E:ST-res} and \eqref{E:ST-fun})
$$
\ST(\Gamma)=\ST(\Gamma\setminus e)\coprod c_e^*(\ST(\Gamma/e)).
$$
The inequality \eqref{E:ineq-c} proves part \eqref{T:BD-Ntype1}. Moreover, it shows that 
\begin{equation}\label{E:eq-ind}
|\BD_I(\Gamma)|=c(\Gamma)\Rightarrow |\BD_{I\setminus e}(\Gamma\setminus e)|=c(\Gamma\setminus e) \text{ and } |\BD_{I/e}(\Gamma/e)|=c(\Gamma/e).
\end{equation}
This proves  part \eqref{T:BD-Ntype2} since any connected spanning subgraphs $\Gamma\setminus S$ can be obtained by iteratively deleting non-separating edges 
and any morphism $f:\Gamma\to \Gamma'$ preserving the genus can be obtained, up to post-componing with an automorphism (which clearly will preserve the equality in question), 
 by iteratively contracting edges that are not loops.
\end{proof}

\begin{cor}\label{C:N=sN}
If $\BD_I$ is of numerically N-type, then:
\begin{enumerate}[(i)]
\item $\BD_{I_{\Gamma\setminus S}}$ is  of numerical SN-type for every connected spanning subgraph $\Gamma\setminus S$ of $\Gamma$.

In particular, $\BD_I$ is of numerical sN-type.
\item $\BD_{f_*(I)}$ if of numerical SN-type for every morphism $f:\Gamma\to \Gamma'$ which preserves the genus. 

In particular, $\BD_{I_{\Gamma/S}}$ is of numerical sN-type for every contraction $\Gamma/S$ such that $b_1(\Gamma/S)=b_1(\Gamma)$. 
\end{enumerate}
\end{cor}
\begin{proof}
This follows from Theorem \ref{T:BD-Ntype}\eqref{T:BD-Ntype2}.
\end{proof}

\begin{cor}\label{C:upper-sN}
If $\cP$ is an upper subset of $\bO_{\con}^d(\Gamma)$ of numerical sN-type, then $\cP=\BD_I$ for a unique $I$.
\end{cor}
\begin{proof}
Since $\cP$ is of numerical sN-type, we have that $\cP(T)$ has cardinality equal to $c(T)=1$ for every spanning subgraph of $\Gamma$. 
Hence we can define the following tree function of degree $d$
$$
\begin{aligned}
I: \ST(\Gamma) & \longrightarrow \Div^{d-b_1(\Gamma)}(\Gamma) \\
T & \mapsto I(T)\quad \text{ such that } \: \cP(T)=\{I(T)\}.
\end{aligned}
$$
Since $\cP$ is an upper subset, we have that $\BD_I(G)\subseteq \cP(G)$ for every spanning subgraph $G$ of $\Gamma$. 
Since $|\BD_I(G)|\geq c(G)$ by Theorem \ref{T:BD-Ntype}\eqref{T:BD-Ntype1} and $|\cP(G)|=c(G)$ since $\cP$ is of numerical sN-type, then we have that 
 $\BD_I(G)=\cP(G)$ for every $G\in \SS_{\con}(\Gamma)$, which shows that $\cP=\BD_I$.
\end{proof}

\subsection{The Main Theorem}

The aim of this subsection is to prove the following Theorem, which can be seen as the Main Theorem of this section.

\begin{thm}\label{T:main-comb}
Let $\Gamma$ be a connected graph. 
\begin{enumerate}
\item \label{T:main-comb1} There is a bijection 
$$
\begin{aligned}
\left\{\text{V-stabilities of degree $d$}\right\} & \longrightarrow \left\{\begin{aligned}\text{degree-$d$ tree functions $I$ such that}\\ \text{$\BD_I$ is of numerical N-type}\end{aligned}\right\} \\
\frak n &\mapsto I_{\frak n},
\end{aligned}
$$
such that $\cP_{\frak n}=\BD_{I_{\frak n}}$.

The tree function  $I_{\frak n}$ is defined as it follows: for any $T\in \ST(\Gamma)$ and any $e\in E(T)$, write $T\setminus e=T[W_e] \coprod T[W_e^c]$, with $T[W_e]$ and $T[W_e^c]$ connected; then $I_{\frak n}(T)$ is the unique element of $\Div^{d-b_1(\Gamma)}(\Gamma)$ such that 
\begin{equation}\label{E:form-In}
I_{\frak n}(T)_{W_e}=\frak n_{W_e}-b_1(\Gamma[W_e]) \text{ and } I_{\frak n}(T)_{W_e^c}=\frak n_{W_e^c}-b_1(\Gamma[W_e^c]).
\end{equation}
\item \label{T:main-comb2} For a subset $\cP\subset \bO^d(\Gamma)$, the following conditions are equivalent:
\begin{enumerate}
\item  \label{T:main-combA} $\cP$ is a V-subset;
\item  \label{T:main-combB} $\cP$ is an upper subset of sN-type;
\item  \label{T:main-combC}$\cP$ is an upper subset of numerical sN-type;
\item  \label{T:main-combD} $\cP$ is a BD-set of N-type;
\item  \label{T:main-combE} $\cP$ is a BD-set of numerical N-type.
\end{enumerate}
\end{enumerate}
\end{thm}
The fact that, for a general numerical polarization $\phi$, the set $\cP_{\frak n(\phi)}$ is an upper subset of sN-type was proved in \cite[Prop. 3.5]{MV}. 
\begin{proof}
Let us first prove part \eqref{T:main-comb1}.
First of all, note that $I_{\frak n}$ is well-defined, since $I_{\frak n}(T)$ is uniquely determined by the formula \eqref{E:form-In} for all the edges $e$ of $T$ and it has total degree equal to $d-b_1(\Gamma)$ since 
$$
I_{\frak n}(T)_{W_e}+I_{\frak n}(T)_{W_e^c}=\frak n_{W_e}-b_1(\Gamma[W_e])+\frak n_{W_e^c}-b_1(\Gamma[W_e^c])=
$$
$$=d+1-\val(W)-b_1(\Gamma[W_e])-b_1(\Gamma[W_e^c])=d-b_1(\Gamma).
$$

We divide the proof into the following two Claims. 

$\un{\text{Claim 1}}:$ $\BD_{I_{\frak n}}$ is of numerical N-type and $\cP_{\frak n}=\BD_{I_{\frak n}}$.

Indeed, Lemma \ref{L:STnonem} implies that for any $T:=\Gamma\setminus S\in \ST(\Gamma)$, we have that 
$$\cP_{\frak n}(T)=\{D(T)\},$$ 
for some $D(T)\in \Div^{d-b_1(\Gamma)}(T)$. Consider now an edge $e\in T$,  and write $T\setminus e=T[W_e] \coprod T[W_e^c]$, with $T[W_e]$ and $T[W_e^c]$ connected.
Then, we have that $T[W_e]$ is a spanning tree of $\Gamma[W_e]$, $T[W_e^c]$ is a spanning tree of $\Gamma[W_e^c]$ and $\val_T(W)=1$. Therefore, using Definition \ref{D:stabV} and \eqref{E:stabV-upp}, we get that 
\begin{equation}\label{E:forDT}
\frak n_{W_e}\leq D(T)_{W_e}+e_S(W_e)=D(T)_{W_e}+b_1(\Gamma[W_e])\leq \frak n_{W_e}+1-\val_T(W_e)=\frak n_{W_e}.
\end{equation}
Comparing \eqref{E:form-In} and \eqref{E:forDT}, we infer that $D(T)_{W_e}=I_{\frak n}(T)_{W_e}$. Since this is true for any edge $e$ of $T$, we deduce that $D(T)=I_{\frak n}(T)$, or in other words that 
 $$\cP_{\frak n}(T)=\{I_{\frak n}(T)\}.$$ 
This implies, using the Definition \ref{D:BD} of BD-sets and the fact that a V-set is an upper subset by Proposition \ref{P:Bn-upper}, that 
\begin{equation}\label{E:1equ}
\BD_{I_{\frak n}}\subseteq \cP_{\frak n}, \text{ or, equivalently, that } \BD_{I_{\frak n}}(\Gamma\setminus S) \subseteq \cP_{\frak n}(\Gamma\setminus S),
\end{equation}
for any connected spanning subgraph $\Gamma\setminus S$. Now, Theorem \ref{T:BD-Ntype} and Lemma \ref{L:inc-BD} imply that 
\begin{equation}\label{E:2equ}
c(\Gamma\setminus S)\leq |\BD_{I_{\frak n}}(\Gamma\setminus S)|,
\end{equation}
while Proposition \ref{P:Bn-card}\eqref{P:Bn-card2} implies that 
\begin{equation}\label{E:3equ}
|\cP_{\frak n}(\Gamma\setminus S)|\leq c(\Gamma\setminus S). 
\end{equation}
Combining \eqref{E:1equ}, \eqref{E:2equ}, \eqref{E:3equ}, we deduce that $\cP_{\frak n}=\BD_{I_{\frak n}}$ and that $\BD_{I_{\frak n}}$ is of numerical N-type, which proves Claim 1.

$\un{\text{Claim 2}}:$ Any BD-set of numerical N-type is of the form $\BD_{I_{\frak n}}$, for a unique general V-stability $\frak n$. 

Indeed, consider a BD-set $\BD_I$ of numerical N-type. Consider the general V-stability condition $\frak n^I$ defined in Lemma \ref{L:I-n} below. If $T$ is a spanning tree of $\Gamma$, $e$ is an edge of $T$ and we write $T\setminus e=T[W_e]\coprod T[W_e^c]$, with $T[W_e]$ and $T[W_e^c]$ connected, then we have that $T[W_e]$ is a spanning tree of $\Gamma[W_e]$, $T[W_e^c]$ is a spanning tree of $\Gamma[W_e^c]$ and $\val_T(W)=1$. Hence, by applying  \eqref{E:nI} for $\frak n^I$ and \eqref{E:form-In} for $I_{\frak n^I}$, we deduce that $I(T)_{W_e}=I_{\frak n^I}(T)_{W_e}$. Since this holds for any edge $e$  of $T$, we infer that $I(T)=I_{\frak n^I}(T)$, which then implies that 
$$
I=I_{\frak n^I} \quad \text{ or, equivalently, } \BD_I=\BD_{I_{\frak n^I}}.
$$

It remains to show that $\frak n^I$ is the unique general V-stability of degree $d$ such that $I_{\frak n^I}=I$. With this aim, suppose that  $\frak n$ is a general V-stability of degree $d$ such that $I_{\frak n}=I$.  
Let $W\subset V(\Gamma)$ be a biconnected subset of $\Gamma$. Choose a spanning tree $T_1$ of $\Gamma[W]$, a spanning tree $T_2$ of $\Gamma[W^c]$ and an edge $e\in E(W,W^c)$. Then the union $T:=T_1\cup T_2\cup \{e\}$ is a spanning subgraph of $\Gamma$, and we have that $W=W_e$ and $W_e^c=W^c$ (up to switching $W$ and $W^c$).
Then, applying \eqref{E:form-In} to $\frak n$ and $\frak n^I$ and using that $I_{\frak n}=I=I_{\frak n^I}$, we get that 
$$
\frak n_W-b_1(\Gamma[W])=I(T)_W=\frak n^I_W-b_1(\Gamma[W]),
$$
which implies that $\frak n=\frak n^I$.

Let us now prove part \eqref{T:main-comb2} by proving the following cycle of implications.

$\bullet$ $\eqref{T:main-combA} \Rightarrow \eqref{T:main-combB}$:  first of all, if $\cP$ is a V-subset then $\cP$ is an upper subset of $\bO_{\con}^d(\Gamma)$ by Proposition \ref{P:Bn-upper} and \eqref{P:Bn-card}\eqref{P:Bn-card1}.  Consider now a connected spanning subgraph $\Gamma\setminus S$ of $\Gamma$. Proposition \ref{P:Bn-card}\eqref{P:Bn-card2} implies that 
the composition 
\[\pi(\Gamma\setminus S):\cP_{\frak n}(\Gamma\setminus S)\subseteq \Div^{d-|S|}(\Gamma\setminus S) \twoheadrightarrow \Pic^{d-|S|}(\Gamma\setminus S)\]
is injective. Since $\cP_{\frak n}$ is a BD-set of numerical N-type by part \eqref{T:main-comb1}, Theorem \ref{T:BD-Ntype}\eqref{T:BD-Ntypea} implies that the cardinality of $\cP_{\frak n}(\Gamma\setminus S)$ is equal to the complexity of $\Gamma\setminus S$, which is also the cardinality of $\Pic^{d-|S|}(\Gamma\setminus S)$. Hence, the above map $\pi$ is bijective, and we conclude that $\cP$ is of sN-type.

$\bullet$ $\eqref{T:main-combB} \Rightarrow \eqref{T:main-combC}$: obvious. 

$\bullet$ $\eqref{T:main-combC} \Rightarrow \eqref{T:main-combD}$: it follows from Corollary \ref{C:upper-sN}.

$\bullet$ $\eqref{T:main-combD} \Rightarrow \eqref{T:main-combE}$: obvious. 

$\bullet$ $\eqref{T:main-combE} \Rightarrow \eqref{T:main-combA}$: it follows from part \eqref{T:main-comb1}. 
\end{proof}

\begin{lemma}\label{L:STnonem}
Let $\frak n$ be a general V-stability condition on $\Gamma$.  If $T$ is a spanning tree of $\Gamma$, then $\cP_{\frak n}(T)$ consists of a single divisor. 
\end{lemma}
\begin{proof}
Using Proposition \ref{P:Bn-card}\eqref{P:Bn-card2} and the fact  that the complexity (and hence the cardinality of the Picard group) of $T$ is equal to one, it is enough to prove that  $\cP_{\frak n}(T)\neq \emptyset$. 
 
We will prove this by induction on the number of vertices of $\Gamma\setminus S:=T$, the case of a tree with one vertex being trivial. 
Pick an end vertex $v$ of $\Gamma\setminus S$, let $e$ be the unique edge of $\Gamma\setminus S$ incident to $v$ and let $w$ be the other vertex incident to $e$. 
Consider the contraction of $e$ 
$$
f:\Gamma\to \Gamma/e:=\wt \Gamma,
$$
which sends the two vertices $v,w$ into the vertex $\wt w$. We will identify the set $V(\Gamma)\setminus \{v,w\}$ with the set $V(\wt \Gamma)\setminus \{\wt w\}$. 

The image $f_*(\Gamma\setminus S)=\wt \Gamma\setminus \wt S$, where $\wt S=(f^E)^{-1}(S)$,  is a spanning tree of $\wt \Gamma$ and the restriction 
$$
f_{\Gamma\setminus S}:\Gamma\setminus S\to \wt \Gamma \setminus \wt S
$$
is again the contraction of the edge $e$. Proposition \ref{P:stabV-fun} gives a well-defined map 
$$
f_*:\cP_{\frak n}(\Gamma\setminus S)\to \cP_{f_*\frak n}(\wt \Gamma \setminus \wt S).
$$
By our induction assumption, the set $\cP_{f_*\frak n}(\wt \Gamma \setminus \wt S)$ consists of a unique divisor $\wt D$. 
We want to show $\wt D$ is the image via $f_*$ of an element of $\cP_{\frak n}(\Gamma\setminus S)$. With this aim, we define the divisor $D$ on $\Gamma$ by 
$$
D_t=
\begin{cases}
\frak n_v-e_S(v) & \text{ if } t=v, \\
\wt D_w-\frak n_v+e_S(v) & \text{ if } t=w, \\
\wt D_t & \text{ if } t\neq v, w.
\end{cases}
$$
We conclude by showing the following

\un{Claim:} $D\in \cP_{\frak n}(\Gamma\setminus S)$.

Indeed, first of all we have that 
$|D|+S=|\wt D|+|\wt S|=|f_*\frak n|=|\frak n|$. 
Consider next a non-trivial biconnected subset $W\subset V(\Gamma)$ and let us check that 
\begin{equation}\label{E:stabV}
D_W+e_S(W) \geq \frak n_W.
\end{equation}
We will distinguish  several cases according to whether or not $v$ and $w$ belong to $W$. Using Remark \ref{sym-inequ}, we can assume that either $v,w\in W$ or $v\in W\not\ni w$, provided that we prove both the inequality \eqref{E:stabV} and \eqref{E:stabV-upp} for such $W$'s.

$\bullet$ Case I: $v,w\in W$

We have that $W=f_V^{-1}(\wt W)$ with $\wt w\in \wt W\subset V(\wt \Gamma)$ biconnected non-trivial  and we compute 
$$
\frak n_W= (f_*\frak n)_{\wt W}\leq D_W+e_S(W)=\wt D_{\wt W}+e_{\wt S}(\wt W)\leq (f_*\frak n)_{\wt W}-1+\val_{\Gamma\setminus S}(\wt W)=\frak n_W-1-\val_{\Gamma\setminus S}(W).
$$

$\bullet$ Case II: $W=\{v\}$.

We have that 
$$
\frak  n_v=D_v+e_S(v)=\frak n_v-1+\val_{\Gamma\setminus S}(v)=\frak n_v.
$$

$\bullet$ Case III: $\{v\}\subsetneq W\not\ni w$.

Let $Z:=W\setminus \{v\}$ and consider the decomposition $\Gamma[Z]=\Gamma[Z_1]\coprod \ldots \coprod \Gamma[Z_k]$ into connected components. 
Note that $Z_i^c=\cup_{j\neq i} Z_j \cup \{v\} \cup W^c$ is connected since $W^c$ is connected, there is at least one edge between $v$ and $W^c$ (namely $e$) and there is at least one edge between $v$ and $Z_j$ (since $W$ is connected). Hence each $Z_i$ is biconnected (and non-trivial). Moreover, we can write $Z_i=f_V^{-1}(\wt Z_i)$ for some $\wt Z_i\subset V(\Gamma)$ non-trivial and biconnected. We set $\wt Z=\cup_i \wt Z_i$ and note that $Z=f_V^{-1}(\wt Z)$. 

Now we compute 
\begin{equation}\label{E:equa1}
D_W+e_S(W)=D_v+D_Z+e_S(v)+e_S(Z)+\val_S(v,Z)=\frak n_v+\wt D_{\wt Z}+e_{\wt S}(\wt Z)+\val_{\Gamma}(v,Z),
\end{equation}
where we have used that $D_v=\frak n_v-e_S(v)$ and that all the edges joining $v$ with $Z$ belong to $S$.

Since $\wt D$ belongs to $\cP_{f_*\frak n}(\wt \Gamma \setminus \wt S)$, we have that (for any $1\leq i \leq k$):
 \begin{equation*}\label{E:equa2}
\frak n_{Z_i}=f_*(\frak n)_{\wt Z_i}\leq \wt D_{\wt Z_i}+e_{\wt S}(\wt Z_i)\leq f_*(\frak n)_{\wt Z_i}-1+\val_{\wt \Gamma\setminus \wt S}(\wt Z_i)=\frak n_{Z_i}-1+\val_{\Gamma\setminus S}(Z_i).
\end{equation*}
Summing the above inequalities over all indices $i=1,\ldots, k$, and using that $\val_{\Gamma}(Z_i,Z_j)=0$ for $i\neq j$, we find 
 \begin{equation}\label{E:equa3}
\sum_i \frak n_{Z_i}\leq \wt D_{\wt Z}+e_{\wt S}(\wt Z)\leq \sum_i \frak n_{Z_i}-k+\val_{\Gamma\setminus S}(Z).
\end{equation}
Finally, applying \eqref{E:Vstab2} to $\frak n$ and the biconnected subcurves $\{v\}\cup_{1\leq i\leq r} Z_i$ for any $1\leq r\leq k$, and using again that $\val_{\Gamma}(Z_i,Z_j)=0$ for $i\neq j$, 
we get that 
\begin{equation}\label{E:equa4}
-k\leq \frak n_{W}-\frak n_v-\sum_i \frak n_{Z_i}-\val_{\Gamma}(v,Z)\leq 0. 
\end{equation}
Putting together \eqref{E:equa1}, \eqref{E:equa3} and \eqref{E:equa4}, and using that $\val_{\Gamma\setminus S}(v,Z)=0$ and $\val_{\Gamma\setminus S}(v,W^c)=1$, 
we get the desired inequalities
\begin{equation*}\label{E:equa5}
\frak n_W\leq D_W+e_S(W)\leq \frak n_w+\val_{\Gamma\setminus S}(Z)=\frak n_w-1+\val_{\Gamma\setminus S}(W).
\end{equation*}
\end{proof}

\begin{lemma}\label{L:I-n}
Let $\BD_I$ be a BD-set of numerical N-type of degree $d$ on $\Gamma$. For any $W\subset V(\Gamma)$, pick a spanning tree $T\in \ST(\Gamma)$ that is \emph{adapted to $W$}, i.e.
such that: 
\begin{equation}\label{E:adaptW}
\begin{sis}
& T[W] \text{ is a spanning tree of } \Gamma[W], \\
& T[W^c] \text{ is a spanning tree of } \Gamma[W^c],\\
& \val_T(W)=1.
\end{sis}
\end{equation}
We define  
\begin{equation}\label{E:nI}
\frak n^I_W:=I(T)_W+b_1(\Gamma[W]). 
\end{equation}
Then $\frak n^I:=\{\frak n^I_W\}_W$ is a well-defined (i.e. independent of the choice of $T$) general V-stability of degree $d$ on $\Gamma$. 
\end{lemma}
\begin{proof}
First of all, we will show that $\frak n^I$ is well-defined. This will follow from the 

\un{Claim 1:} Let $T_i\in \ST(\Gamma)$ (with $i=0,1$) be two spanning trees of $\Gamma$ that are adapted to $W$. Then 
$$I(T_1)_W=I(T_2)_W.$$ 

We will prove this by induction on the complexity of $\Gamma$. If the complexity of $\Gamma$ is equal to one, then $T_1=T_2$ and the Claim is obvious. 

We will now treat first a special case and then the general case.

$\bullet$ Special case: $c(\Gamma[W])=c(\Gamma[W^c])=1$. 

Set $E(W,W^c)=\{e_1,\ldots, e_t\}$ for some $t\geq 1$. Then there exists $1\leq k(1), k(2)\leq t$ such that 
$$
\begin{aligned}
& T_1=T_W\cup T_{W^c}\cup \{e_{k(1)}\}, \\
& T_2=T_W\cup T_{W^c}\cup \{e_{k(2)}\}, \\
\end{aligned}
$$
where $T_W$ (resp. $T_{W^c}$) is the unique spanning tree of $\Gamma[W]$ (resp. $\Gamma[W^c]$).

Consider the contraction $f:\Gamma\to \ov \Gamma$ of the edges of $T_W$ and the edges of $T_{W^c}$.
The graph $\ov \Gamma$ has the same genus of $\Gamma$ and it consists of two vertices $v_1$ and $v_2$ (which are the images of, respectively, $\Gamma[W]$ and 
$\Gamma[W^c]$, under $f$), each of which has some loops (of number equal to, respectively, $b_1(\Gamma[W])$ and $b_1(\Gamma[W^c])$)  and $E(v_1,v_2)=\{\ov e_1,\ldots, \ov e_t\}$, where $\ov e_i=f(e_i)$. 
For $i=1,2$, denote by $\ov T_i$ the spanning trees of $\ov \Gamma$ which is the images of $T_i$ under $f$, or in other words such that 
$f^*(\ov T_i)=T_i$ under the map \eqref{E:ST-fun}. Then we have that $E(\ov T_i)=\{\ov e_{k(i)}\}$ for $i=1,2$. 
Consider the tree function $\ov I:=f_*(I)$ on $\ov \Gamma$, as in \eqref{E:I-fun}. Note that, since $\BD_I$ is of numerical N-type, then $\BD_{\ov I}$ is of numerical sN-type by Corollary \ref{C:N=sN}. By the definition \eqref{E:I-fun} of $f_*(I)$, we have that 
$$
\ov I(T_i)=I(T_i)_W v_1+I(T_i)_{W^c} v_2= I(T_i)_W v_1+(d-I(T_i)_{W}-b_1(\Gamma)) v_2 \text{ for } i=1,2. 
$$ 
Consider now the spanning subgraph $\wt \Gamma$ of $\ov \Gamma$ such that $E(\wt \Gamma)=\{\ov e_1,\ldots, \ov e_t\}$, or in other words, the spanning subgraph obtained from $\ov \Gamma$ by deleting all the loops. Since $T_i$ is obtained from $\wt \Gamma$ by deleting the edges of $E(W,W^c)-\{\ov e_{k(i)}\}$, by Definition \ref{D:BD} we have that (for $i=1,2$)
\begin{equation}\label{E:inc-vine}
\BD_{\ov I}(\wt \Gamma)\supseteq \bigcup_{r=0}^{t-1} \{(I(T_i)_W+r) v_1+(d-I(T_i)_{W}-b_1(\Gamma)+t-1-r) v_2 \}
\end{equation}
Since $\wt \Gamma$ has complexity equal to $t$ and $\BD_{\ov I}$ is of numerical sN-type, then $\BD_{\ov I}(\wt \Gamma)$ has cardinality equal to $t$. Therefore, the inclusion in \eqref{E:inc-vine} must be an equality both for $i=1$ and for $i=2$. This can happen only if $I(T_1)_W=I(T_2)_W$, and we are done. 

$\bullet$ General case: either $c(\Gamma[W])>1$ or $c(\Gamma[W^c])>1$. 

Since $I(T_i)_W+I(T_i)_{W^c}=d-b_1(\Gamma)$, up to switching $W$ and $W^c$, we can assume that $c(\Gamma[W])>1$. Hence, there exists an edge $e$ of $\Gamma[W]$ that is neither a separating edge nor a loop. Consider the deletion graph $\Gamma\setminus e$ (resp. the contraction graph $\Gamma/e$) and the tree function $I_{\Gamma\setminus e}$ (resp. the tree function $I_{\Gamma/e}$) induced by $I$ as in \eqref{E:I-res} (resp. \eqref{E:I-fun}). Since $\BD_I$ is of numerical N-type, then also $\BD_{I_{\Gamma\setminus e}}$ and $\BD_{I_{\Gamma/e}}$ are of numerical N-type by Corollary \ref{C:N=sN}. Hence, since $c(\Gamma\setminus e),c(\Gamma/e)<c(\Gamma\setminus e)+c(\Gamma/e)=c(\Gamma)$, 
Claim 1 holds true for $I_{\Gamma\setminus e}$ and $I_{\Gamma/e}$ by our induction assumption. We now distinguish three cases:
\begin{enumerate}[(a)]
\item If $T_1$ and $T_2$ do not contain $e$, then they are spanning trees of $\Gamma\setminus e$ and hence
$$
I(T_1)_W=I_{\Gamma\setminus e}(T_1)_W=I_{\Gamma\setminus e}(T_2)_W=I(T_2)_W,
$$ 
where the middle equality follows from the fact that Claim 1 holds for $I_{\Gamma\setminus e}$.
\item If $T_1$ and $T_2$ both contains $e$, then the image of $T_i$ in $\Gamma/e$ is a spanning tree $\ov T_i$ of $\Gamma/e$ and we have that 
$$
I(T_1)_W=I_{\Gamma/e}(\ov T_1)_{\ov W}=I_{\Gamma/e}(\ov T_2)_{\ov W}=I(T_2)_W,
$$ 
where $\ov W\subset V(\Gamma/e)$ is the image of $W\subset V(\Gamma)$, and the middle equality  follows from the fact that Claim 1 holds for $I_{\Gamma/ e}$.

\item Suppose that $T_1$ contains $e$ and $T_2$ does not contain $e$ (the opposite case is treated similarly). Consider the spanning subgraph $T_2\cup e$, which will have genus one $1$. Hence, $T_2\cup e$ will contain a cycle $\{e,f_1,\ldots,f_r\}$ of some length $r+1\geq 2$, entirely contained in $\Gamma[W]$, and the other edges will be separating edges. 
The spanning trees of $T_2\cup e$ are $T_2$ and $\wt T_i:=T_2\cup e -\{f_i\}$ for $1\leq i \leq r$. By Case (b) above, we have that 
\begin{equation}\label{E:IwtTi}
I(T_1)_W=I(\wt T_i)_W \text{ for any } 1\leq i \leq r. 
\end{equation}
By definition of $\BD_I$, we have that 
\begin{equation}\label{E:Cyc-con}
\BD_I(T_2\cup e)\supseteq \bigcup_{i=1}^r I(\wt T_i)+\D(f_i).
\end{equation}
Since $\BD_I$ is of numerical N-type and hence of numerical sN-type by Corollary \ref{C:N=sN}, we have that $\BD_I(T_2\cup e)$ has cardinality equal to 
\begin{equation}\label{E:Cyc-car}
|\BD_I(T_2\cup e)|=c(T_2\cup e)=r+1.
\end{equation}
Combining \eqref{E:Cyc-con} and \eqref{E:Cyc-car} and arguing as in \cite[Lemma 3.7]{PT1},  we get that 
\begin{equation}\label{E:Cyc-eq}
\BD_I(T_2\cup e)= \bigcup_{i=1}^r I(\wt T_i)+\D(f_i).
\end{equation}
Since the edges $\{f_1,\ldots, f_r\}$ are contained in $\Gamma[W]$,  \eqref{E:IwtTi} and \eqref{E:Cyc-eq} imply that 
\begin{equation}\label{E:D-cyc}
D\in \BD_I(T_2\cup e)\Rightarrow D_W=I(\wt T_i)_W+1=I(T_1)_W+1. 
\end{equation}
Since $\BD_I(T_2\cup e)\supseteq I(T_2)+\D(e)$ and $e$ is contained in $\Gamma[W]$, \eqref{E:D-cyc} implies that 
$$
I(T_2)_W=I(T_1)_W,
$$
as required.
\end{enumerate}

\vspace{0.1cm}

We next shows that $\frak n^I$ is a general V-stability condition of degree $d$ on $\Gamma$, by checking that it satisfies the two conditions of Definition \ref{D:Vstab}.

\un{Claim 2:} for any biconnected $W\subset V(\Gamma)$, we have that 
$$
\frak n^I_W+\frak n^I_{W^c}=d+1-\val(W).
$$

Indeed, pick $T\in \ST(\Gamma)$ such that $T[W]$ is a spanning tree of $\Gamma[W]$, $T[W^c]$ is a spanning tree of $\Gamma[W^c]$ and $\val_{T}(W)=1$.
Then 
$$
\frak n^I_W+\frak n^I_{W^c}=I(T)_W+b_1(\Gamma[W])+I(T)_{W^c}+b_1(\Gamma[W^c])=\deg I(T)+b_1(\Gamma)-\val(W)+1=d-\val(W)+1,
$$
where we have used the formula $b_1(\Gamma)=b_1(\Gamma[W])+b_1(\Gamma[W^c])+\val(W)-1$ and the fact that $I(T)\in \Div^{d-b_1(\Gamma)}(\Gamma)$
since $\BD_I$ has degree $d$. 

\un{Claim 3:}  for any disjoint $W_1,W_2\subseteq V(\Gamma)$ such that $W_1,W_2$ and $W_1\cup W_2$ are biconnected and non-trivial, we have that 
\begin{equation}\label{E:condII}
\frak n^I_{W_1\cup W_2}-\frak n^I_{W_1}-\frak n^I_{W_2}-\val(W_1,W_2)\in \{-1, 0\}.
\end{equation}

Indeed, set $W_3:=(W_1\cup W_2)^c$ and pick spanning subgraphs $T_i$ of $\Gamma[W_i]$ for $i=1,2,3$, and an edge $e_{ij}\in E(W_i,W_j)$ for $1\leq i \neq j\leq 3$.
Consider the connected subgraph of genus $1$:
$$
C:=T_1\cup T_2\cup T_3\cup \{e_{12}, e_{13}, e_{23}\}. 
$$
The right hand side of \eqref{E:condII} is equal to
$$
\frak n^I_{W_1\cup W_2}-\frak n^I_{W_1}-\frak n^I_{W_2}-\val(W_1,W_2)=
$$
$$
=I(C-e_{23})_{W_1\cup W_2}+b_1(\Gamma[W_1\cup W_2])-I(C-e_{13})_{W_1}-b_1(\Gamma[W_1])-I(C-e_{12})_{W_2}-b_1(\Gamma[W_2])-\val(W_1,W_2)=
$$
\begin{equation}\label{E:RHS-1}
=I(C-e_{23})_{W_1\cup W_2}-I(C-e_{13})_{W_1}-I(C-e_{12})_{W_2}-1,
\end{equation}
where we have used the formula $b_1(\Gamma[W_1\cup W_2])=b_1(\Gamma[W_1])+b_1(\Gamma[W_2])+\val(W_1,W_2)-1$.

Consider now the graph $\ov \Gamma$ which is obtained from $\Gamma$ by first contracting the edges of the trees $T_i$ for $1\leq i \leq 3$, and then deleting all the edges 
of $E(W_i,W_j)-e_{ij}$ for any $1\leq i\neq j\leq 3$. The graph $\ov \Gamma$ has three vertices $\{v_1,v_2,v_3\}$ which are the images of the subgraphs $\{\Gamma[W_1], \Gamma[W_2], \Gamma[W_3]\}$, an edge $\ov e_{ij}$ connecting $v_i$ with $v_j$ for $1\leq i\neq j\leq 3$ (which are the images of the three edges $\{e_{ij}\}$ of $\Gamma$), and several loops centered at the vertices (as many as the genus of each $\Gamma[W_i]$). The function $I$ on $\Gamma$ determines a function $\ov I$ on $\ov \Gamma$ using the formulas \eqref{E:I-fun} (not that the contraction we have performed did not decrease the genus of the graph) and \eqref{E:I-res}. 
Moreover, under the natural maps of spanning trees \eqref{E:ST-fun} and \eqref{E:ST-res}, the spanning tree $\ov e_{ij}+\ov e_{ik}$ of $\ov \Gamma$ is sent to the spanning tree $C-e_{jk}$ of $\Gamma$, for any $\{i,j,k\}=\{1,2,3\}$. Hence, by definition of $\ov I$, we have that 
\begin{equation}\label{E:RHS-2}
I(C-e_{23})_{W_1\cup W_2}-I(C-e_{13})_{W_1}-I(C-e_{12})_{W_2}=\ov I(\ov e_{12}+\ov e_{13})_{v_1\cup v_2}-\ov I(\ov e_{12}+\ov e_{23})_{v_1}-\ov I(\ov e_{13}+\ov e_{23})_{v_2},
\end{equation}

By assumption, $\BD_I$ is of numerical N-type, and this implies by Corollary \ref{C:N=sN} that $\BD_{\ov I}$ is of numerical sN-type. In particular, on the spanning subgraph $\ov e_{12}+\ov e_{13}+\ov e_{23}$ of $\ov \Gamma$, we should have that 
$$|\BD_{\ov I}(\ov e_{12}+\ov e_{13}+\ov e_{23})|=c(\ov e_{12}+\ov e_{13}+\ov e_{23})=3,$$ 
since $\ov e_{12}+\ov e_{13}+\ov e_{23}$  is a cycle of length three.  It is easy to see that this forces the existence of a unique $(d_1,d_2,d_3)\in \Z^3$ such that exactly one of the following two possibilities occurs for $\ov I$:

\un{Case A:} $\ov I(\ov e_{ij}+ \ov e_{ik})=d_1v_1+d_2v_2+d_3v_3$ for any $\{i,j,k\}=\{1,2,3\}$.

In this case,  $\BD_{\ov I}(\ov e_{12}+\ov e_{13}+\ov e_{23})=\bigcup_{i=1}^3\{d_1v_1+d_2v_2+d_3v_3+v_i\}$. 

\un{Case B:} $\ov I(\ov e_{ij}+ \ov e_{ik})=d_1v_1+d_2v_2+d_3v_3+v_i$ for any $\{i,j,k\}=\{1,2,3\}$.

In this case,  $\BD_{\ov I}(\ov e_{12}+\ov e_{13}+\ov e_{23})=\bigcup_{1\leq i\neq j\leq 3}\{d_1v_1+d_2v_2+d_3v_3+v_i+v_j\}$. 

\vspace{0.1cm}

We can now compute the right hand side of \eqref{E:RHS-2}:
\begin{equation}\label{E:RHS-3}
\ov I(\ov e_{12}+\ov e_{13})_{v_1\cup v_2}-\ov I(\ov e_{12}+\ov e_{23})_{v_1}-\ov I(\ov e_{13}+\ov e_{23})_{v_2}=
\begin{cases}
d_1+d_2-d_1-d_2=0 & \text{ in case A,}\\
d_1+d_2+1-d_1-d_2=1 & \text{ in case B.}\\
\end{cases}
\end{equation}
By putting together \eqref{E:RHS-1}, \eqref{E:RHS-2} and \eqref{E:RHS-3}, we deduce that \eqref{E:condII} holds true. 
\end{proof}

A corollary of Theorem \ref{T:main-comb} is the following

\begin{cor}\label{C:unique-n}
A general V-stability condition $\frak n$ on $\Gamma$ is uniquely determined by its V-set $\cP_{\frak n}$, i.e. 
$$
\cP_{\frak n_1}=\cP_{\frak n_2}\Rightarrow \frak n_1=\frak n_2.
$$
\end{cor}
\begin{proof}
It follows from Theorem \ref{T:main-comb}\eqref{T:main-comb1} together with the obvious fact that if $I_1,I_2$ are two degree-$d$ tree functions such that $\BD_{I_1}=\BD_{I_2}$ then $I_1=I_2$. 
\end{proof}

\section{Fine V-compactified Jacobians}\label{Sec:fcJ}

\subsection{Notation on nodal curves}\label{sub:not-nodal}

Let $X$ be a nodal curve, i.e. a projective and reduced curve (over an algebraically closed field)  having only nodes as singularities.

The \emph{dual graph} of $X$, denoted by $\Gamma_X$, is the graph having one vertex for each irreducible component of $X$, one edge for each node of $X$ and such that an edge is adjacent to a vertex if the corresponding node belongs to the corresponding irreducible component. 
We will denote the irreducible components of $X$ by 
$$\{X_v\::\: v\in V(\Gamma_X)\},$$ 
and the nodes of $X$ by 
$$\Xsing:=\{n_e\: :\: e\in E(\Gamma_X)\}.$$ 
Note that $X$ is connected if and only if $\Gamma_X$ is connected. In general, we will denote by $\gamma(X)$ the number of connected components of $X$ (or of its dual graph $\Gamma_X$). 

A \emph{subcurve} $Y\subset X$ is a closed subscheme of $X$ that is a curve, or in other words $Y$ is the union of some irreducible components of $X$. 
Hence the subcurves of $X$ are in bijection with the subsets of $V(\Gamma_X)$ and we will often denote them by 
$$
X[W]:=\bigcup_{v\in W} X_v \:  \: \text{ for } W\subseteq V(\Gamma_X).
$$
We say that a subcurve $X[W]$ is non-trivial if $X[W]\neq \emptyset, X$, which happens if and only if $W$ is non-trivial.
The dual graph of $X[W]$ is equal to the induced subgraph $\Gamma_X[W]$. Hence, a subcurve $X[W]$ is connected if and only if $W\subseteq V(\Gamma_X)$ is connected. 

The complementary subcurve of $X[W]$ is defined to be 
$$
X[W]^c:=X[W^c]. 
$$
In other words, $X[W]^c$ is the closure of the complementary subset $X\setminus X[W]$. 
We say that a subcurve $X[W]$ is biconnected it if it connected and its complementary subcurve $X[W]^c$ is connected, which happens if and only if $W$ is biconnected. 

We define the join and the meet of two subcurves by 
$$
\begin{sis}
& X[W_1]\vee X[W_2]:=X[W_1\cup W_2], \\
& X[W_1]\wedge X[W_2]:=X[W_1\cap W_2].
\end{sis}
$$
In other words, the join of two subcurves is simply their union, while the meet of two subcurves is the union of their common irreducible components. 

Given two subcurves $Y_1,Y_2$ of $X$ without common irreducible components (i.e. such that $Y_1\wedge Y_2=\emptyset$), we define 
by $|Y_1\cap Y_2|$ the cardinality of their intersection (which is a subset of $\Xsing$).  Note that 
$$
|X[W_1]\cap X[W_2]|=\val_{\Gamma_X}(W_1,W_2) \:\: \text{ for any  } \: W_1\cap W_2=\emptyset.
$$



Given a subset $S\subset E(\Gamma_X)$, we denote by  $X_S$ the \emph{partial normalization} of $X$ at the nodes corresponding to $S$ and by $\nu_S:X_S\to X$ the partial normalization morphism. The dual graph of $X_S$ is equal to 
$$\Gamma_{X_S}=\Gamma_X\setminus S.$$

\subsection{Torsion-free, rank-$1$ sheaves on nodal curves}\label{Sec:sheaves}

Let $X$ be a connected nodal curve over an algebraically closed field $k$.

Let $I$ be a coherent sheaf on $X$. We say that $I$ is:
\begin{itemize}
\item  \emph{torsion-free}  if its associated points are generic points of $X$. Equivalently, $I$ is pure of dimension one (i.e. it does not have torsion subsheaves) and it has support equal to $X$.
\item \emph{rank-$1$} if $I$ is invertible on a dense open subset of $X$. 
\item  \emph{simple} if $\End(I) = k $.
\end{itemize}
Note that each line bundle on $X$ is torsion-free, rank-$1$ and simple. The \emph{degree} of a rank-$1$ sheaf $I$ is defined to be
$$
\deg(I)=\chi(I)-\chi(\O_X). 
$$

A torsion-free sheaf $I$ is locally free away from the nodes of $X$. We will denote by $\NF(I)$, and call it the \emph{non-free locus} of $I$, the set of nodes of $X$ at which $I$ is not free and by 
$G(I)$ the spanning subgraph $\Gamma_X\setminus \NF(I)$, and we call it the \emph{free subgraph} of $I$. 
If $I$ is a rank-$1$ torsion-free sheaf then the stalk at a point $p\in X$ is equal to 
$$
I_p=
\begin{cases}
\O_{X,p} & \text{ if } p\not \in \NF(I), \\
\m_p & \text{ if } p\in \NF(I),
\end{cases}
$$ 
where $\m_p$ is the maximal ideal of the local ring $O_{X,p}$. A rank-$1$, torsion-free sheaf $I$ is equal to 
$$I=\nu_{\NF(I),*}(L_I),$$
for a uniquely determined line bundle $L_I$ on the partial normalization $\nu_{\NF(I)}:X_{\NF(I)}\to X$ of $X$ at $\NF(I)$. Indeed, the line bundle $L_I$ is given by the pull-back $\nu_{\NF(I)}^*(I)$ quotient out by its torsion subsheaf and its degree is equal to 
$$
\deg L_I=\deg I-|\NF(I)|. 
$$

The ($\Gm$-rigidified) \emph{stack of torsion-free rank-$1$ sheaves} on $X$ is denoted by $\TF_X$. 
More precisely, $\TF_X$ represents the functor that associates to each $k$-scheme $T$ the set of $T$-flat coherent sheaves $\I$ on $X\times_k T$ such that
$\I|_{X\times t}$ is torsion-free, rank-$1$ sheaf on $X$ for each $t\in T$, modulo the following equivalence relation: we say that two such sheaves $\I_1$ and $\I_2$
are equivalent if there is an invertible sheaf $\N$ on $T$ such that $\I_1\cong\I_2\otimes p_2^*\N$, where $p_2: X\times T\to T$ is the projection map.
Since the $\Gm$-rigidification is a trivial $\Gm$-gerbe, the stack $\TF_X$ comes equipped with a universal sheaf $\I$ on $X\times_k \TF_X$, uniquely determined up to pull-back by a line bundle on $\TF_X$. The automorphism group of $I\in \TF_X$ is equal to   
\begin{equation}\label{E:Aut-I}
\ov \Aut(I):=\frac{\Aut(I)}{\Gm}=\Gm^{\gamma(X_{\NF(I)})-1}=\Gm^{\gamma(G(I))-1},
\end{equation}
where the quotient by $\Gm$ comes from the fact that we have rigidified the stack by scalar automorphisms. 

We denote by 
$$\Simp_X\subseteq \TF_X$$ 
the open substack parametrizing simple rank-$1$, torsion-free sheaves. Note that $\Simp_X$ is the biggest open algebraic subspace of $\TF_X$, and indeed it is also 
a reduced scheme. It follows from \eqref{E:Aut-I} that 
\begin{equation}\label{E:simp-I}
I\in \Simp_X \Leftrightarrow G(I) \text{ is connected } \Leftrightarrow \Gamma_{G(I)}=\Gamma_X\setminus \NF(I) \text{ is connected. }
\end{equation}

The scheme $\Simp_X$ contains the Picard scheme 
$$
\PIC_X\subseteq \Simp_X,
$$
as the open subscheme parametrizing line bundles on $X$.

The local structure of $\Simp_X$ is described in the next result. 

\begin{fact}\label{F:Simploc}
$\Simp_X$ is a scheme such that:
\begin{enumerate}
\item 
 The completion of the local ring of $\Simp_X$ at $I$ is isomorphic to
$$
\wh \O_{\Simp_X,I}\cong \wh \bigotimes_{n\in\NF(I)} \frac{k[[x_n,y_n]]}{(x_ny_n)}\wh \bigotimes k[[t_1,\ldots, t_{g(X_{\NF(I)})}]].
$$
In particular, $\Simp_X$ is a reduced scheme of pure dimension $g(X)$ with locally complete intersection singularities and its smooth locus is $\PIC_X$. 
\item The normalization of $\Simp_X$ is smooth. 
\end{enumerate} 
\end{fact}
\begin{proof}
The fact that $\Simp_X$ is a scheme follows from \cite[Thm. B]{est1}. The completed local ring $\wh \O_{\Simp_X,I}$ is isomorphic to the semiuniversal deformation ring of $I$, whose structure is determined in \cite[\S 3]{CMKV}. The fact that the normalization of $\Simp_X$ is smooth follows from the proof of \cite[Cor. 13.3]{OS}.
\end{proof}

We have a decomposition into connected components  
\begin{equation}\label{E:Torsd}
\TF_X=\coprod_{d\in \Z} \TF^d_X
\end{equation}
where $\TF^d_X$ parametrizes sheaves of degree $d$. The decomposition \eqref{E:Torsd} induces the following decompositions 
$$
\Simp_X=\coprod_{d\in \Z} \Simp^d_X \: \text{ and }\:  \PIC_X=\coprod_{d\in \Z} \PIC^d_X,
$$
although $\Simp^d_X$ and $\PIC^d_X$ are not necessarily connected.  

For each subcurve $Y$ of $X$, let $I_Y$ be the restriction $I_{|Y}$ of $I$ to $Y$ modulo torsion.
If $I$ is a torsion-free (resp. rank-$1$) sheaf on $X$, so is $I_Y$ on $Y$.
We let $\deg_Y (I)$ denote the degree of $I_Y$, that is, $\deg_Y(I) := \chi(I_Y )-\chi(\O_Y)$. The \emph{multidegree} of a rank-$1$ torsion-free sheaf $I$ on $X$ is the divisor on $\Gamma_X$ defined as the multidegree of the line bundle $L_I$ on $X_{\NF(I)}$, i.e.
$$D(I):=D(L_I)=\{D(I)_v:=\deg_{|(X_{\NF(I)})_v}(L_I)\: : \: v\in V(\Gamma)\}.$$
where we identify as usual $V(\Gamma_{X_{\NF(I)}})=V(\Gamma_X\setminus \NF(I))=V(\Gamma_X)$.
It turns out that for the subcurve $X[W]$ of $X$ associated to $W\subset V(\Gamma_X)$, we have
\begin{equation}\label{E:deg-sub}
\deg_{X[W]}(I)=D(I)_{W}+e_{\NF(I)}(W). 
\end{equation}
In particular, the degree of $I$ on an irreducible component $X_v$ is given by 
$$
\deg_{X_v}(I)=D(I)_v+e_{\NS(I)}(v),
$$
and the total degree of $I$ is given by
\begin{equation}\label{E:deg-mdeg}
\deg(I)=|D(I)|+|\NF(I)|. 
\end{equation}

The Picard variety $\PIC_X$ of $X$ acts on $\TF_X$ via tensor product. We can restrict this action to the subgroups 
$$
\PIC_X^{\un 0}\subseteq \PIC_X^0\subseteq \PIC_X
$$
where $\PIC^{\un 0}_X$, called the \emph{generalized Jacobian} of $X$, is the semiabelian variety parametrizing  line bundles of multidegree zero.
The orbits of the  action of $\PIC_X^{\un 0}$ on $\TF_X$ are described in the following well-known 

\begin{fact}\label{F:orbits} (see e.g. \cite[Sec. 5]{MV})
Let $X$ be a connected nodal curve. 
\begin{enumerate}
\item \label{F:orbits1}
The orbits of $\PIC_X^{\un 0}$ on $\TF_X$ are given by 
\begin{equation}\label{E:orbit1}
\TF_X(G,D):=\{I\in \TF_X\: :\: G(I)=G \: \text{ and } \: D(I)=D\}.
\end{equation}
In particular, we get a decomposition into disjoint $\PIC_X^{\un 0}$-orbits
$$
\TF_X^d=\coprod_{(G,D)\in \bO^d(\Gamma_X)} \TF_X(G,D) \subset \TF_X=\coprod_{(G,D)\in \bO(\Gamma_X)} \TF_X(G,D) 
$$
\item \label{F:orbits2} We have that 
$$\begin{sis}
& \TF_X(G,D)\subset \Simp_X \Leftrightarrow G \:  \text{ is connected, }\\
& \TF_X(G,D)\subset \PIC_X \Leftrightarrow G=\Gamma_X.
\end{sis}
$$ 
In particular, we get a decomposition into disjoint $\PIC_X^{\un 0}$-orbits
$$
\Simp_X^d=\coprod_{(G,D)\in \bO_{\con}^d(\Gamma_X)} \TF_X(G,D) \subset \Simp_X=\coprod_{(G,D)\in \bO_{\con}(\Gamma_X)} \TF_X(G,D) 
$$
\item \label{F:orbits3}
Each $\TF_X(G,D)$ is a locally closed substack of $\TF_X$ and, if we endow it with the reduced stack structure, there is an isomorphism (if we write $G=\Gamma_X\setminus S$)
$$
(\nu_S)_*: \PIC^{D}_{X_S}:=\{L\in \PIC(X_{S}): \:  D(L)=D\} \xrightarrow{\cong} \TF_X(\Gamma_X\setminus S, D). 
$$
Under this isomorphism, the action of $\PIC_X^{\un 0}$ factors through the quotient $\PIC_X^{\un 0}\twoheadrightarrow \PIC_{X_S}^{\un 0}$ followed by the tensor product action of $\PIC_{X_S}^{\un 0}$ on $ \PIC^{D}_{X_S}$.
\item \label{F:orbits4}
The closure of $\TF_X(G,D)$ is equal to 
$$
\ov{\TF_X(G,D)}=\coprod_{(G,D)\geq (G',D')} \TF_X(G',D'),
$$
where $\geq $ is defined in \eqref{E:posetO}. In particular, the poset of $\PIC_X^{\un 0}$-orbits of $\TF_X$ (resp. $\TF_X^d$) is isomorphic to $\bO(\Gamma_X)$ (resp. $\bO^d(\Gamma_X)$)
and the poset of $\PIC_X^{\un 0}$-orbits of $\Simp_X$ (resp. $\Simp_X^d$) is isomorphic to $\bO_{\con}(\Gamma_X)$ (resp. $\bO_{\con}^d(\Gamma_X)$).
 \end{enumerate}







\end{fact}

\begin{remark}
The orbits of $\PIC_X^0$ on $\TF_X$ are given by 
\begin{equation}\label{E:orbit2}
\TF_X(G,d):=\{I\in \TF_X\: :\: G(I)=G \: \text{ and } \: \deg(I)=d\}=\coprod_{D: |D|=d} \TF_X(G,D).
\end{equation}

The orbits of $\PIC_X$ on $\TF_X$ are given by 
\begin{equation}\label{E:orbit3}
\TF_X(G):=\{I\in \TF_X\: :\: G(I)=G\}=\coprod_{D} \TF_X(G,D).
\end{equation}
\end{remark}

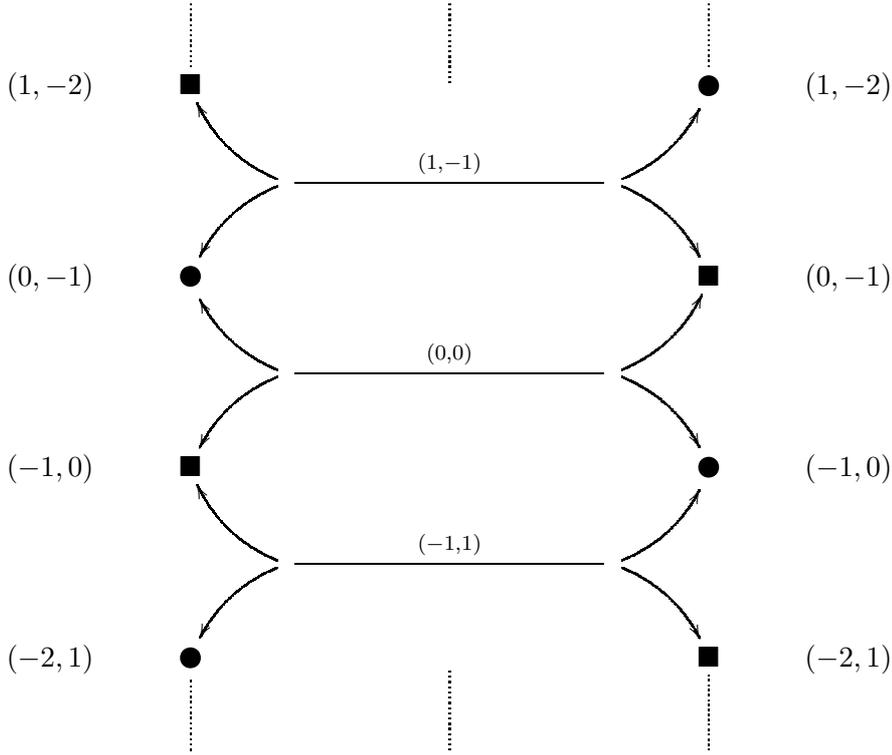
\begin{figure}[hbt!]
\[\xymatrix{
&&&&&&&&\\
(1,-2)&{ \blacksquare}\ar@{.}[u] &&& \ar@{.}[u] &&& { \mathbin{\vcenter{\hbox{\scalebox{2}{$\bullet$}}}}}\ar@{.}[u]& (1,-2)\\
&&\ar@{-}[rrrr]^{(1,-1)} \ar@/^/[lu]  \ar@/_/[ld] && & &  \ar@/_/[ru]   \ar@/^/[rd]  &  &&\\
(0,-1)&{ \mathbin{\vcenter{\hbox{\scalebox{2}{$\bullet$}}}}} &&&&&& { \blacksquare}& (0,-1)\\
&&\ar@{-}[rrrr]^{(0,0)} \ar@/^/[lu]  \ar@/_/[ld] && & &  \ar@/_/[ru]   \ar@/^/[rd]  &  &&\\
(-1,0)&{ \blacksquare} &&&&&& { \mathbin{\vcenter{\hbox{\scalebox{2}{$\bullet$}}}}}& (-1,0)\\
&&\ar@{-}[rrrr]^{(-1,1)} \ar@/^/[lu]  \ar@/_/[ld] && & &  \ar@/_/[ru]   \ar@/^/[rd]  &  && \\
(-2,1)&{ \mathbin{\vcenter{\hbox{\scalebox{2}{$\bullet$}}}}} \ar@{.}[d] &&& \ar@{.}[d]  &&&  { \blacksquare}\ar@{.}[d] &(-2,1)\\
&&&&&&&& \\
}\]

\caption{If $X$ a $2$-cycle curve with $2$ smooth rational components meeting in two nodes $n_1$ and $n_2$, then $\Simp_X^0$ has infinitely many irreducible components
$\{\ov{\TF_X(\Gamma_X,(d,-d))}\}_{d\in \Z}$, each of which is a $\P^1$ with a double origin $\{\TF_X(\Gamma_X\setminus \{n_1\},(d,d-1)),\TF_X(\Gamma_X\setminus \{n_2\},(d-1,d))\}$ and a double infinity $\{\TF_X(\Gamma_X\setminus \{n_2\},(d,d-1)),\TF_X(\Gamma_X\setminus \{n_1\},(d-1,d))\}$ 
}
\end{figure}

The previous Fact implies that any upper subset of $\bO(\Gamma_X)$ (resp. $\bO_{\con}(\Gamma_X)$) determines an open subset of  of $\TF_X$ (resp. $\Simp_X$), as we now formalize in the following

\begin{defi}\label{D:opTF}
\noindent 
\begin{enumerate}
\item \label{D:opTF1} Given an upper subset $\cP$ of $\bO^d(\Gamma_X)$ (resp. $\bO_{\con}^d(\Gamma_X)$), we set 
$$
U(\cP):=\coprod_{(G,D)\in \cP} \TF(G,D) \subset \TF_X \: (\text{resp. } \Simp_X),  
$$
Note that $U(\cP)$ is an open subset by Fact \ref{F:orbits}, and it will be called the open subset associated to the upper subset $\cP$.
\item \label{D:opTF2} We say that an open subset $U\subset  \Simp_X$ is of  \emph{sN-type} (resp. of \emph{numerical sN-type}, resp. of \emph{N-type}, resp. of \emph{numerical N-type})  if $U=U(\cP)$ for an upper subset $\cP$ of  $\bO_{\con}^d(\Gamma_X)$ having the same property (as in Definition \ref{D:upper-sN}).
\end{enumerate}
\end{defi}
Note that, we have the following implications
$$
\xymatrix{
\left\{\text{Open subsets of sN-type}\right\} \ar@{=>}[r] \ar@{=>}[d]  & \left\{\text{Open subsets of N-type}\right\} \ar@{=>}[d] \\
\left\{\text{Open subsets of numerical sN-type}\right\} \ar@{=>}[r] & \left\{\text{Open subsets of numerical N-type}\right\}.
}
$$

\subsubsection{Specializations}

We now examine the specializations of torsion-free rank-$1$ sheaves on nodal curves. 

First of all, we describe the behavior of the decomposition of $\TF_X$ into $\PIC_X^{\un 0}$-orbits under specializations of the nodal curve $X$.

\begin{prop}\label{P:spec-orb}
Let $\pi:\X\to \Delta=\Spec R$ be a flat and proper morphisms whose geometric fibers are connected nodal curves, $R$ is a complete DVR with algebraically closed residue field $k$.
Denote by $\X_0$ the closed fiber of $\pi$ and by $\X_{\ov \eta}$ the geometric generic fiber of $\pi$. Consider the morphism of dual graphs $f:\Gamma_{\X_0}\to \Gamma_{\X_{\ov \eta}}$ 
induced by the \'etale specialization $\ov\eta \rightsquigarrow 0$. Consider the ($\Gm$-rigidified) stack $\TF_{\X/\Delta}$ of relative torsion-free rank-$1$ sheaves on $\X/\Delta$.
\begin{enumerate}
\item \label{P:spec-orb1} If $\sigma:\Delta \to \TF_{\X/\Delta}$ is a section of $\TF_{\X/\Delta}\to \Delta$ with $I_0:=\sigma(0)$ and $I_{\ov \eta}=\sigma(\ov \eta)$, then 
\begin{equation}\label{E:spec-type}
f_*(\NF(I_0), \un \deg(I_0))\leq (\NF(I_{\ov \eta}), \un \deg(I_{\ov \eta})).
\end{equation}
\item \label{P:spec-orb2} Conversely, if $I_0\in \TF_{\X_0}$ and $I_{\ov \eta}\in \TF_{\X_{\ov \eta}}^d$ are such that \eqref{E:spec-type} holds, then there exists a section $\sigma:\Delta \to \TF_{\X/\Delta}$ such that $\sigma(0)=I_0$ and $\sigma(\ov \eta)=I_{\ov \eta}$.
\end{enumerate}
\end{prop}
\begin{proof}
Part \eqref{P:spec-orb1}: see \cite[Prop. 3.41]{CC} or \cite[Lemma 5.2]{PT2}.
Part \eqref{P:spec-orb2}: see \cite[Prop. 3.42]{CC}.
\end{proof}

We now give a more detailed analysis of the specializations of torsion-free rank-$1$ sheaves on a fixed nodal curve $X$. 

With this aim, we fix a discrete valuation $k$-ring $R$ with residue field $k$ and quotient field $K$, and we denote by $\val:K\to \Z\cup\{\infty\}$ the associated valuation. Set $B:=\Spec R$ with generic point $\eta:=\Spec K$ and special point $o:=\Spec k$.  Consider a relative torsion-free rank-$1$ sheaf $\I\in \TF_X(B)$ on $X_B:=X\times_k B$ and denote by $\I_{\eta}\in \TF_X(\eta)$ its generic fiber and by $\I_o\in \TF_X(k)$ its special fiber.  
We now introduce a $1$-cochain on a spanning subgraph of  $\Gamma_X$ associated to $\I$ that will allow us to describe the combinatorial type $(G(\I_o),D(\I_o))$  of the special fiber $\I_o$ in terms of the combinatorial type $(G(\I_\eta),D(\I_\eta))$ of the generic fiber.  

Recall that the space of $1$-cochains (and the one of $1$-chains) on a graph $G$ can be defined in the following way. Denote by $\bE(G)$ the set of oriented edges of $G$ and denote by $s,t:\bE(G)\to V(G)$ the source and target functions that associate to an oriented edge $\be$ its source and target, respectively.  Given an oriented edge $\be$, we denote by $\ov{\be}$ the opposite oriented edge so that $t(\be)=s(\ov \be)$ and $s(\be)=t(\ov \be)$. 
The space of $1$-chains and $1$-cochains on $G$ with coefficients in an abelian group $A$ (e.g. $A=\Z$)  are defined as 
\begin{equation*}
\bC_1(G,A):=\frac{\bigoplus_{\be \in \bE(G)} A\cdot \be}{(\be=-\ov \be)} \quad \text{ and } \quad \bC^1(G,A):=\{f:\bE(G)\to A\: : f(\be)=-f(\ov \be)\}. 
\end{equation*}


In order to define the $1$-cochain associated to $\I$, consider  the double dual $\L(\I)$ of the pull-back of $\I$ via the normalization map $\nu_B:\wt X_B:=\wt X\times_k B\to X_B$, which is line bundle since it is reflexive sheaf on a regular $2$-dimensional scheme. Similarly, denote by $\L(\I_\eta)$  the torsion-free quotient of the pull-back of $\I_{\eta}$ via the normalization map $\nu_{\eta}:\wt X_{\eta}:=\wt X\times_k k(\eta) \to X_{\eta}:=X\times_k k(\eta)$ and by $\L(\I_o)$  the torsion-free quotient of the pull-back of $\I_{o}$ via the normalization map $\nu:\wt X \to X$.
Note that $\L(\I)$ is the unique line bundle on $\wt X_B$ whose generic fiber is $\L(\I)_{\eta}=\L(\I_{\eta})$. 
As shown in \cite[Prop. 12.7]{OS}, there is a canonical presentation of $\I$ of the form 
\begin{equation}\label{E:pres-I}
0\to \I \to (\nu_B)_*(\L(\I))\xrightarrow{a} \bigoplus_{n\in \NF(\I_{\eta})^c}\O_{\{n\}\times B}\to 0.
\end{equation}
For any $n\in \NF(\I_{\eta})^c=E(G(\I_\eta))$, choose an oriented edge $\be$ whose underlying edge $e$ corresponds to the node $n$. 
Denote by $n_\be^s$ and $n_\be^t$ the two inverse images of $n$ under the normalization map $\nu:\wt X\to X$ in such a way that $n_\be^s\in \wt X_{s(\be)}$ and $n_\be^t\in \wt X_{t(\be)}$.
The restriction of $a$ to  $\{n\}\times B$ induces a surjection
$$
a_{|\{n\}\times B}: (\nu_B)_*(\L(\I))_{|\{n\}\times B}=\L(\I)_{|\{n_\be^s\}\times B}\oplus \L(\I)_{|\{n_{\be}^t\}\times B}\twoheadrightarrow \O_{\{n\}\times B}
$$
which defines an element $[x_{\be}^s,x_{\be}^t]\in \P^1(R)$. Since the restriction $a_{|\{n\}\times \eta}$ is surjective on each of the two factors (see \cite[Prop. 12.7]{OS}), the element 
$[x_\be^s,x_{\be}^t] \in \P^1(R)$ is different from the $0$ and $\infty$. We set 
\begin{equation}\label{E:gammaI}
\gamma(\I)(\be):=\val\left(\frac{x_{\be}^t}{x_{\be}^s}\right)\in \Z. 
\end{equation}
Note that $\gamma(\I)(\be)=-\gamma(\I)(\ov \be)$ by construction, and hence we get a well-defined element $\gamma(\I)\in \bC^1(G(\I_{\eta}),\Z)$, called the \emph{$1$-cochain associated to $\I$.}  Intuitively, the integer $\gamma(\I)(\be)$ measures the ''speed'' at which $\I_o$ is smoothened out in the direction of $\be$.

We define the support of $\gamma(\I)$ as 
\begin{equation}\label{E:supp-gamma}
\supp \gamma(\I):=\{e\in E(G(\I_\eta))\: : \gamma(\I)(\be)\neq 0 \text{ for some orientation } \be  \text{ of } e\},
\end{equation}
and the \emph{orientation} associated to $\I$ (or to $\gamma(\I)$) as 
\begin{equation}\label{E:OI}
\O(\I):=\bigcup\{\be: \gamma(\I)(\be)>0\}.
\end{equation}
Note that $\supp \O(\I)=\supp \gamma(\I)$.


The orientation $\O(\I)$ associated to $\I$ allows us to 
describe the combinatorial type $(G(\I_o),D(\I_o))$ of the special fiber $\I_o$ in terms of the combinatorial type $(G(\I_\eta),D(\I_\eta))$.

\begin{lemma}\label{L:Io-Ieta}
With the above notation, we have that 
$$\begin{aligned}
&G(\I_o)=G(\I_{\eta})\setminus \supp \O(\I), \\
&D(\I_o)=D(\I_{\eta})-\D(\O(\I)).
\end{aligned}$$
\end{lemma}
\begin{proof}
Since we have that $\L(\I)_{\eta}=\L(\I_{\eta})$, then $\L(\I)$ is a line bundle on $\wt X\times_k B$ of relative multidegree equal to $D(\I_\eta)$. 

The presentation \eqref{E:pres-I} induces a presentation  of the special fiber $\I_{o}$ 
\begin{equation*}\label{E:pres-Io}
0\to \I_{o} \to (\nu_\eta)_*(\L(\I)_{o})\xrightarrow{a_o} \bigoplus_{n\in \NF(\I_\eta)^c}\O_{\{n\}}\to 0.
\end{equation*}
The restriction of $a_o$ to  $\{n\}$ induces a surjection
$$
(a_{o})_{|\{n\}}: (\nu)_*(\L(\I)_o)_{|\{o\}}=\L(\I_o)_{|\{n_\be^s\}}\oplus \L(\I_o)_{|\{n_{\be}^t\}}\twoheadrightarrow \O_{\{n\}},
$$
which defines an element $[x_{\be}^s(o),x_{\be}^t(o)]\in \P^1(k)$. 
As shown in the proof of \cite[Lemma 12.6]{OS}, we have that 
\begin{equation*}\label{E:for-Io}
\L(\I)_o=\L(\I_o)\left(\sum_{x_{\be}^s=0} n_{\be}^s+\sum_{x_{\be}^t=0} n_{\be}^t\right).
\end{equation*}
By passing to the multidegrees, we get 
\begin{equation}\label{E:deg-Io}
D(\I_\eta)=D(\I_o)+\sum_{x_{\be}^s=0} s({\be})+\sum_{x_{\be}^t=0} t(\be).
\end{equation}
Since the element $[x_{\be}^s(o),x_{\be}^t(o)]\in \P^1(k)$ is the reduction of the element $[x_{\be}^s,x_{\be}^t]\in \P^1(R)$ defined above, we have that 
\begin{equation}\label{E:null-cord}
\begin{sis}
x_{\be}^t=0 \Leftrightarrow \gamma(\I)(\be)>0, \\
x_{\be}^s=0 \Leftrightarrow \gamma(\I)(\be)<0.
\end{sis}
\end{equation}
We conclude by combining \eqref{E:deg-Io} and \eqref{E:null-cord}.
\end{proof}

Finally, we want to describe all the relative torsion-free rank-$1$ sheaves  $X_B$ that have the same generic fiber. 

Recall that the space of $0$-cochains on a  graph $G$ with coefficients in an abelian group $A$ (e.g. $A=\Z$, $(R^*,\cdot)$, or $(K^*,\cdot)$) is given by 
\begin{equation*}\label{E:0-cochains}
C^0(G,A):=\left\{g:V(G)\to A\right\}.
\end{equation*}
and it is endowed with the coboundary map 
\begin{equation*}\label{E:cobound}
\begin{aligned}
\delta=\delta_G: C^0(G,A) & \longrightarrow \bC^1(G,A),\\
g &\mapsto \delta(g)(\be):=g(t(\be))-g(s(\be)).
\end{aligned}
\end{equation*}

We now want to define an action of $C^0(\Gamma_X,\Z)$ on the set $|\TF_X(B)|$ of isomorphism classes of the groupoid $\TF_X(B)$. Take $\I\in \TF_X(B)$ and $g\in \C^0(\Gamma_X,\Z)$. The exact sequence of abelian groups
$$
0\to (R^*,\cdot) \to (K^*,\cdot) \xrightarrow{\val} \Z\to 0
$$
induces an exact sequence of $0$-cochains
$$
0\to C^0(\Gamma_X, R^*) \to C^0(\Gamma_X,K^*) \xrightarrow{C^0(\val)} C^0(\Gamma_X,\Z)\to 0.
$$
Pick a lift $\wt g\in C^0(\Gamma_X,K^*)$ of $g$. The element $\wt g$ induces an automorphism $\wt g_*$ of the line bundle $\L(\I_{\eta})$ which is the scalar multiplication by $\wt g(v)\in K^*$ on the irreducible component $(\wt X_v)_{\eta}:=\wt X_v\times_k K$ of $\wt X_{\eta}$ corresponding to the vertex $v$ of $\Gamma_X$. The automorphism $\wt g_*$ extends uniquely to an automorphism of the line bundle $\L(\I)$ on $\wt X_B$ and 
hence to an automorphism of the sheaf $(\nu_B)_*(\L(\I))$ on $X_B$, that we will also denote by $\wt g_*$. Consider now the presentation \eqref{E:pres-I} of $\I$ and define a new element $\wt g(\I)$ of $\TF_X(B)$ as it follows
\begin{equation}\label{E:pres-gI}
0\to \wt g(\I):=\ker(a\circ \wt g_*) \to (\nu_B)_*(\L(\I))\xrightarrow{a\circ \wt g_*} \bigoplus_{n\in \NF(\I_{\eta})^c}\O_{\{n\}\times B}\to 0.
\end{equation}
Arguing as in the proof of \cite[Prop. 12.3]{OS}, it follows that if $\wt g'$ is another lift of $g$ (so that $\wt g-\wt g'\in C^0(\Gamma_X,R^*)$), then $\wt g(\I)$ is isomorphic to $\wt g'(\I)$ in $X_B$. 
Hence, we get a well-defined action 
\begin{equation}\label{E:action}
\begin{aligned}
C^0(\Gamma_X,\Z)\times |\TF_X(B)|& \longrightarrow |\TF_X(B)|\\
(g,\I) & \mapsto g(\I):=\wt g(\I).
\end{aligned}
\end{equation}

\begin{prop}\label{P:non-sepa}
\noindent 
\begin{enumerate}
\item \label{P:non-sepa1} 
The action \eqref{E:action} induces a bijection
$$
\begin{aligned}
|\TF_X(B)|/C^0(\Gamma_X,\Z) & \xrightarrow{\cong} |\TF_X(\eta)|\\
[\I] &\mapsto \I_{\eta}.
\end{aligned}
$$
\item  \label{P:non-sepa2}
The $1$-cochains associated to conjugate elements under the action \eqref{E:action}  satisfy the relation
$$
\gamma(g(\I))=\gamma(\I)+\delta_{G(\I_{\eta})}(g)\in \bC^1(G(\eta),\Z),
$$
where $g$ is interpreted as an element of $C^0(G(\I_{\eta}),\Z)$ using that $V(G(\I_{\eta}))=V(\Gamma_X)$. 
 \end{enumerate}
\end{prop}
In particular, part \eqref{P:non-sepa1} says that $\TF_X$ satisfies the existence part of the valuative criterion for properness and that the action of $C^0(\Gamma_X,\Z)$ is responsible for the non-separatedness of $\TF_X$. 
\begin{proof}
Part \eqref{P:non-sepa1} is a reformulation of \cite[Cor. 12.9]{OS}.

Let us prove part \eqref{P:non-sepa2}. Consider the presentation of $g(\I)=\wt g(\I)$ in \eqref{E:pres-gI}. As above,  for any $n\in \NF(\I_{\eta})^c=E(G(\I_\eta))$, choose an oriented edge $\be$ whose underlying edge $e$ corresponds to the node $n$ and denote by $n_\be^s$ and $n_\be^t$ the two inverse images of $n$ under the normalization map $\nu:\wt X\to X$ in such a way that $n_\be^s\in X_{s(\be)}$ and $n_\be^t\in X_{t(\be)}$. Using that $ (\nu_B)_*(\L(\I))_{|\{n\}\times B}=\L(\I)_{|\{n_\be^s\}\times B}\oplus \L(\I)_{|\{n_{\be}^t\}\times B}$, 
the restriction of $a\circ \wt g_*$ to  $\{n\}\times B$ is equal to 
$$
(a\circ \wt g_*)_{|\{n\}\times B}: \L(\I)_{|\{n_\be^s\}\times B}\oplus \L(\I)_{|\{n_{\be}^t\}\times B}\xrightarrow[\cong]{\wt g_*} \L(\I)_{|\{n_\be^s\}\times B}\oplus \L(\I)_{|\{n_{\be}^t\}\times B} \stackrel{a}{\twoheadrightarrow} \O_{\{n\}\times B}.
$$
Since the automorphism $\wt g_*$  is the scalar multiplication by $\wt g(v)\in K^*$ on the irreducible component  $\wt X_v$, the element of $\P^1(R)$ corresponding to the surjection  
$(a\circ \wt g_*)_{|\{n\}\times B}$ is equal to $[\wt g(s(\be))x_{\be}^s,\wt g(t(\be))x_{\be}^t]$, where $[x_{\be}^s,x_{\be}^t]$ is the element of $\P^1(R)$ corresponding to the surjection 
$a_{|\{n\}\times B}$. Hence, by the definition \eqref{E:gammaI}, we have that 
$$
\gamma(g(\I))(\be)=\val\left(\frac{\wt g(t(\be))x_{\be}^t}{\wt g(s(\be))x_{\be}^s}\right)=\val\left(\frac{x_{\be}^t}{x_{\be}^s}\right)+g(t(\be))-g(s(\be))=\gamma(\I)(\be)+\delta_{G(\I_{\eta})}(g)(\be),
$$
which concludes our proof. 
\end{proof}

\subsection{Fine compactified Jacobians}

In this subsection, we introduce fine compactified Jacobians of a connected nodal curve $X$ over $k=\ov k$.

\begin{defi}\label{D:compJac}(\cite[Def. 3.1]{PT2})
A \textbf{fine compactified Jacobian}  of a connected nodal curve $X$ is a connected proper open substack $\ov J_X$ of $\TF_X$.
\end{defi}

\begin{remark}\label{R:irr}
If $X$ is an irreducible nodal curve, then $\TF_X^d=\Simp_X^d$ is proper and connected, and hence it is the unique fine compactified Jacobian of $X$ of degree $d$. 
 
On the other hand, if $X$ is reducible then $\Simp_X^d$ is not of finite type nor separated, and there are infinitely many fine compactified Jacobians of $X$ of any given degree $d$ (see e.g. Example \ref{Ex:clas-fcJ}).
\end{remark}

Two immediate consequences of the definition are given in the following 

\begin{lemma}\label{L:prop-fcJ}
Let $\ov J_X$ be a fine compactified Jacobian of $X$.
\begin{enumerate}[(i)]
\item \label{L:prop-fcJ1} $\ov J_X$ is contained in $\Simp_X$. Hence, $\ov J_X$ is a variety (i.e. a reduced scheme of finite type over the base field $k$).
\item  \label{L:prop-fcJ2} There exists a unique integer $d$ such that $\ov J_X\subseteq \TF_X^d$.
\end{enumerate}
\end{lemma}
The integer $d$ of \eqref{L:prop-fcJ2} is called the \emph{degree} of $\ov J_X$, and we sometimes write $\ov J_X^d$ to denote a fine compactified Jacobian of degree $d$. 
\begin{proof}
Part \eqref{L:prop-fcJ1}: let $I\in \ov J_X$. Since $\ov J_X$ is separated, the automorphism group $\ov{\Aut}(I)$ of $I$ is proper. Hence, by the description of $\Aut(I)$ of \eqref{E:Aut-I}, 
it follows that $\ov{\Aut}(I)$ is trivial, or equivalently that $I$ is simple. The fact that $\ov J_X$ is a variety follows from Fact \ref{F:Simploc}.

Part \eqref{L:prop-fcJ2}: since $\ov J_X$ is connected by definition, we have that $\ov J_X$ is contained in a unique connected component of $\TF_X$. We conclude that $\ov J_X$ is contained in a unique $\TF_X^d$ since  \eqref{E:Torsd} is the decomposition of $\TF_X$ into connected components.
\end{proof}

\begin{cor}\label{C:fcJ-orb1}
Let $\ov J_X$ be a fine compactified Jacobian of $X$ of degree $d$.
If $\TF(G,D)\subset \ov J_X$ then $\deg(D)+|E(G)^c|=d$.
\end{cor}
\begin{proof}
This follows from the fact that if $I\in \TF(G,D)$ then $\deg I=|\un d|+|S|$ by \eqref{E:deg-mdeg}.
\end{proof}

An important property of fine compactified Jacobians is given by the following 

\begin{prop}\label{P:fcJ-orb}(\cite[Lemma 7.2]{PT2})
Any fine compactified Jacobian $\ov J_X$ of $X$ is a union of finitely many $\PIC_X^{\un 0}$-orbits of $\TF_X$. In particular, the generalized Jacobian $\PIC_X^{\un 0}$ acts on $\ov J_X$ via tensor product.  
\end{prop}

\begin{cor}\label{C:fcJ-orb2}
Let $\ov J_X$ be a fine compactified Jacobian of $X$.
If $\TF(G,D)\subset \ov J_X$ then
$$\TF(G', D')\subset \ov J_X \text{ for any } (G',D')\geq (G,D).$$ 
\end{cor}
\begin{proof}
Proposition \ref{P:fcJ-orb} implies that $\ov J_X$ is a union of $\PIC_X^{\un 0}$-orbits of $\TF_X$. Moreover, since $\ov J_X$ is open in $\TF_X$, if $\ov J_X$ contains an orbit  $\TF(S, \un{d})$, then it must contains all the other orbits $\TF(S', \un{d'})$ whose closure contains $\TF(G,D)$. We then conclude by Fact \ref{F:orbits}\eqref{F:orbits4}. 
\end{proof}

The above Proposition \ref{P:fcJ-orb} allows to define the poset of orbits of a fine compactified Jacobian. 

\begin{defi}\label{D:orb-fcJ}
Let $\ov J_X^d$ be a fine compactified Jacobian of $X$ of degree $d$. The \emph{poset of orbits} of $\ov J_X^d$ is the following upper subset 
$$\cP(\ov J_X^d):=\{(G,D)\in \bO_{\con}^d(\Gamma_X)\: : \TF(G,D)\subset \ov J_X\}\subset \bO_{\con}^d(\Gamma_X).$$ 
\end{defi}
The poset of orbits for fine classical compactified Jacobians was studied in \cite[\S 3.1]{MRV}.

Note that Proposition \ref{P:fcJ-orb} implies that 
$$\ov J_X^d=U(\cP(\ov J_X^d)).$$

\begin{remark}\label{R:pos-orb}
It can be shown, arguing as in \cite[Prop. 3.4]{MRV} and using the local structure of $\TF_X$ described in Fact \ref{F:Simploc}, that the poset of orbits $\cP(\ov J_X^d)$ is isomorphic to the singular poset of $\ov J_X^d$ (see \cite[p. 5361]{MRV} for the precise definition), and hence that it depends only on the isomorphism class of $\ov J_X^d$. 

Furthermore, arguing as in \cite[Prop. 3.5]{MRV}, it can be shown that if $k=\C$ then the poset of orbits $\cP(\ov J_X^d)$ depends only on the homeomorphism class of $\ov J_X^d$. 
\end{remark}

We now prove a lower bound on the number of orbits of a fine compactified Jacobian. 

\begin{thm}\label{T:bound-orb}
Let $\ov J_X$ be a fine compactified Jacobian. For every connected spanning subgraph $G$ of $\Gamma_X$, we have that 
$$|\{D\in \Div(\Gamma_X)\: : \TF(G,D)\subset \ov J_X\}|\geq c(G).$$
\end{thm}
\begin{proof}
Call $d$ the degree of $\ov J_X$ and set $\ov J_X^d:=\ov J_X$. We first perform two reductions. 

\un{Reduction 1:} it is enough to show that 
\begin{equation}\label{E:red1}
\{D\in \Div(\Gamma_X)\: : \TF(G,D)\subset \ov J_X\}\neq \emptyset \: \text{ for any connected spanning subgraph } G \text{  of } \Gamma_X.  
\end{equation}

Indeed, if \eqref{E:red1} holds, then for any spanning tree $T$ of $\Gamma_X$, there exists $D_T\in \Div(\Gamma_X)$ such that 
$$
\TF(T,D_T)\subset \ov J_X^d.
$$
Since $D$ must have degree $d-|E(T)^c|$ by Corollary \ref{C:fcJ-orb1} and $|E(T)^c|=g(\Gamma_X)$ for any spanning tree $T$, the function 
$$
\begin{aligned}
I:\ST(\Gamma_X) & \longrightarrow \Div^{d-g(\Gamma_X)}(\Gamma_X)\\
T & \mapsto I(T):=D_T
\end{aligned}
$$
is a tree function of degree $d$ on $\Gamma_X$. 
By Definition  \ref{D:BD} and Corollary \ref{C:fcJ-orb2}, the BD-set $\BD_I$ associated to $I$ satisfies 
\begin{equation}\label{E:incl-pos}
\BD_I\subseteq \cP(\ov J_X^d)\subset \bO_{\con}^d(\Gamma_X).
\end{equation}
Then, by applying Lemma \ref{L:inc-BD} and Theorem  \ref{T:BD-Ntype}\eqref{T:BD-Ntype1}, we get that 
$$
|\{D\in \Div(\Gamma_X)\: : \TF(G,D)\subset \ov J_X\}|\geq |\BD_I(G)|\geq c(G)
$$
for any connected spanning subgraph $G$ of $\Gamma_X$, which proves the conclusion of the Theorem.

\un{Reduction 2:} it is enough to show that 
\begin{equation}\label{E:red2}
\{D\in \Div(\Gamma_X)\: : \TF(\Gamma_X\setminus \{e\},D)\subset \ov J_X\}\neq \emptyset \: \text{ for any non-separating edge $e$ of } \Gamma_X. 
\end{equation}

Indeed, we assume that \eqref{E:red2} holds and we will prove \eqref{E:red1} by induction on $g(\Gamma_X)$.

First of all, if $G=\Gamma_X$, then \eqref{E:red1} follows since $\ov J_X^d$ is open in $\TF_X^d$ and $\PIC_X^d$ is open and dense in $\TF_X^d$. This settles the base case of the induction, since if $g(\Gamma_X)=0$ then the only connected spanning subgraph of $\Gamma_X$ is itself. 

On the other hand, if $G\neq \Gamma_X$, then $G$ is a spanning subgraph of a certain $\Gamma_X\setminus \{e\}$, for some non-separating edge $e$ of $\Gamma_X$.
Consider the partial normalization $\nu_e:X_e\to X$ of $X$ at the node $e$. By Fact \eqref{F:orbits}, the push-forward along $\nu_e$ induces a closed embedding 
$$
(\nu_e)_*:\TF_{X_e}^d\hookrightarrow \TF_X^d,
$$
whose image is identified with the union of all the strata $\TF(H,D)$ of $\TF_X^d$ such that $H\leq \Gamma_X\setminus \{e\}=\Gamma_{X_e}$. 
The inverse image $(\nu_e)_*^{-1}(\ov J_X^d)$ is therefore an open and proper subset of $\TF_{X_e}^d$, and it is non-empty by the assumption \eqref{E:red2}.
Hence, a connected component of $(\nu_e)_*^{-1}(\ov J_X^d)$ is a fine compactified Jacobian of $X_e$, call it $\ov J_{X_e}^d$. 
Since $g(\Gamma_X\setminus \{e\})<g(\Gamma_X)$, we can apply our induction hypothesis in order to conclude that there exists an orbit of the form $\TF(G,D)$ contained in $\ov J_{X_e}^d$, and hence also 
in $\ov J_X^d$, and we are done. 

\vspace{0.1cm}

We now conclude showing that \eqref{E:red2} holds true by distinguishing two cases:

\vspace{0.1cm}

\un{Final step A:} \eqref{E:red2} holds true if $e$ is a loop of $\Gamma_X$.

Indeed, since $\ov J_X$ is open in $\TF_X^d$ by definition and it is a union of $\Pic^{\un 0}_X$-orbits by Proposition \ref{P:fcJ-orb}, we can find $B=\Spec R$, with $R$ a discrete valuation $k$-ring, with special point $o=\Spec k$ and generic point $\eta=\Spec K$, and an element $\I\in \TF_X(B)$ such that $\I_{\eta}$ is a line bundle belonging to $\ov J_X(K)$ and $\I_o$ is a torsion-free sheaf on $X$ that is not locally free exactly at the node $n_e$ corresponding to $e$.  Lemma \ref{L:Io-Ieta} implies that the $1$-cochain $\gamma(\I)\in \bC^1(\Gamma_X,\Z)$ associated to $\I$ is such that $\supp \gamma(\I)=\{e\}$.

Since $\ov J_X$ is proper, by the valuative criterion of properness there exists  $\I'\in \TF_X(B)$ such that $\I'_\eta\cong \I_\eta$ and $\I'_o\in \ov J_X$. Proposition \ref{P:non-sepa}\eqref{P:non-sepa1} implies that there exists $g\in C^0(\Gamma_X,\Z)$ such that $\I'\cong g(\I)$ while Proposition \ref{P:non-sepa}\eqref{P:non-sepa2} implies that 
$$\gamma(\I')=\gamma(g(\I))=\gamma(\I)+\delta(g)\in \bC^1(\Gamma_X,\Z),$$ 
where $\delta=\delta_{\Gamma_X}$. Since $\supp \gamma(\I)=\{e\}$ and, for any orientation $\be$ of $e$, we have that $\delta(g)(\be)=g(t(\be))-g(s(\be))=0$ because $e$ is a loop, we have that 
\begin{equation}\label{E:eq-supp}
e\in \supp \gamma(\I').
\end{equation}
Using that $G(\I_o')=\Gamma_X\setminus \supp \gamma(\I')$ by Lemma \ref{L:Io-Ieta}, we deduce that $\I_o'$ is not locally free at $n_e$.
Since $\ov J_X^d$ is open, we can deform $\I'_o$ to a sheaf  $J\in \ov J_X^d$ that is not locally free only at $n_e$. Then we must have that 
$$
\TF_X(\Gamma\setminus \{e\},D(J))\subset \ov J_X^d.
$$

\vspace{0.1cm}

\un{Final step B:} \eqref{E:red2} holds true if $e$ is not a loop of $\Gamma_X$.

Denote by $u$ and $v$ the two distinct vertices to which $e$ is adjacent and choose an orientation $\be$ of $e$ that goes from $u$ to $v$. Since $e$ is not a separating edge, there exists a path $c$ from $v$ to $u$ disjoint from $e$, i.e. $c=\be_1+\ldots+\be_n\in \bC_1(\Gamma_X\setminus \{e\},\Z)$ with $s(\be_1)=v=:v_0$, $t(\be_i)=s(\be_{i+1})=:v_i$ for any $1\leq i \leq n-1$, $t(\be_n)=u=:v_{n}$ such that the vertices $\{v_i\}_{0\leq i \leq n}$  are pairwise disjoint. We now introduce the following partial order relation on the divisors of $\Div^d(\Gamma_X)$: 
$$
D\leq_c D' \Leftrightarrow (D_{v_0}, \ldots, D_{v_n}) \leq_{\rm lex}  (D'_{v_0}, \ldots, D'_{v_n}),
$$
where $\leq_{\rm lex}$ is the lexicographic order, and we say that 
$$
D<_c D' \Leftrightarrow D\leq_c D' \text{ and } D\neq D'.
$$

Pick now a divisor $D^1\in \Div^d(\Gamma_X)$ such that $\TF(\Gamma_X,D^1)\subset \ov J_X^d$, which exists since $\ov J_X^d$ is open in $\Simp_X^d$ by definition and it is a union of $\Pic^{\un 0}_X$-orbits by Proposition \ref{P:fcJ-orb}. Assume by contradiction that there exists no $E\in \Div^{d-1}(\Gamma_X)$ such that $\TF(\Gamma_X\setminus \{e\},E)\subset \ov J^d_X$. We can find a discrete valuation $k$-ring $R$ with special point $o=\Spec k$ and generic point $\eta=\Spec K$ and an element $\I\in \TF_X(B:=\Spec R)$ such that $\I_{\eta}$ is a line bundle on $X_K$ of multidegree $D(\I_\eta)=D^1$ (and hence such that $\I_\eta\in \ov J^d_X(K)$) 
and such that $\gamma(\I)$ is equal to the characteristic function $\chi_\be$ of $\be$, so that $\I_o$  is a torsion-free sheaf on $X$ that is not locally free exactly at the node $n_e$ corresponding to $e$ (and hence such that $\I_o\not \in \ov J^d_X$).
Since $\ov J_X$ is  proper, by the valuative criterion of properness there exists  $\I'\in \TF_X(B)$ such that $\I'_\eta\cong \I_\eta$ and $\I'_o\in \ov J^d_X$. Proposition \ref{P:non-sepa}\eqref{P:non-sepa1} implies that there exists $g\in C^0(\Gamma_X,\Z)$ such that $\I'\cong g(\I)$ while Proposition \ref{P:non-sepa}\eqref{P:non-sepa2} implies that 
$$\gamma(\I')=\gamma(g(\I))=\gamma(\I)+\delta(g)=\chi_{\be}+\delta(g)\in \bC^1(\Gamma_X,\Z)$$ 
where $\delta=\delta_{\Gamma_X}$. 
Since $\I_o'\in \ov J_X^d$, $G(\I_o')=G(\I_\eta)\setminus \supp \gamma(\I')=\Gamma_X\setminus \supp(\chi_{\be}+\delta(g))$ by Lemma \ref{L:Io-Ieta} and  $e\not\in \supp \gamma(\I')$ by our assumption, we must have that 
\begin{equation}\label{E:guv}
\delta(g)(\be)=g(t(\be))-g(s(\be))=g(v)-g(u)=-1.
\end{equation}
Note that $D(\I'_o)=D(\I_{\eta})-\D(\O(\I))=D^1-\D(\O(\I))$.

Now, \eqref{E:guv} implies that the sequence of integers 
$$\{g(v_0=v),g(v_1),\ldots, g(v_{n-1}),g(v_n=u)\}$$
cannot be non-decreasing. Then we pick the smallest index $i\in \{1,\ldots, n\}$ such that 
\begin{equation}\label{E:gvi}
g(v_{i-1})<g(v_i).
\end{equation}
Consider the following partial orientation 
\begin{equation}\label{E:Ohat}
  \wh \O:=\O(\I)\setminus \{\be_i\} \cup \{\ov \be_i\},
\end{equation}
which clearly satisfies $\supp \wh \O=\supp \O(\I)$. The degree-$d$ divisor
$$D^2:=D^1-\D(\O(\I))+\D(\wh \O),$$ 
 satisfies, by construction, the properties
$$
\begin{sis}
& D^1<_c D^2, \\
& (\Gamma_X,D^2)\geq (\Gamma_X\setminus \supp(\wh \O),D^1-\D(\O(\I))=(\Gamma_X\setminus \supp \O(\I),D^1-\D(\O(\I)).
\end{sis}
$$
Since $\I'_o\in \ov J_X^d\cap \TF_X(\Gamma_X\setminus \supp \O(\I),D^1-\D(\O(\I)))$ and $\ov J_X^d$ is open and union of $\Pic^{\un 0}_X$-orbits, we deduce that 
$$
\TF_X(\Gamma,D^2)\subseteq \ov J_X^d.
$$
Repeating the above construction starting with $D^2$ and keep iterating, we find an infinite chain of divisors of degree $d$ on $\Gamma$ such that 
$$
\begin{sis}
& D^1<_c D^2<_c D^3<_c\ldots <_c D^r<_c\ldots \\
& \TF_X(\Gamma_X,D^r)\subset \ov J_X^d \text{ for any } r\geq 1.
\end{sis}
$$
This contradicts the fact that $\ov J_X^d$ is of finite type, and hence we are done.  
\end{proof}

The smooth locus and the irreducible components of a fine compactified Jacobian are described in the following

\begin{prop}\label{P:sm-irr}
Let $\ov J_X$ be a fine compactified Jacobian of $X$.
\begin{enumerate}
\item \label{P:sm-irr1} The smooth locus of $\ov J_X$ is equal to 
$$
(\ov J_X)_{\sm}=\coprod_{(\Gamma_X,D)\in \cP(\ov J_X)} \TF(\Gamma_X,D).
$$
\item \label{P:sm-irr2} The irreducible components of $\ov J_X$ are given by 
$$
\{\ov{\TF(\Gamma_X,D)}\: : (\Gamma_X,D)\in \cP(\ov J_X)\}.
$$
\end{enumerate}
\end{prop}
\begin{proof}
Part \eqref{P:sm-irr1} follows from Proposition \ref{P:fcJ-orb} and Fact \ref{F:Simploc}. 
Part \eqref{P:sm-irr2} follows from Proposition \ref{P:fcJ-orb} and Fact \ref{F:orbits}.
\end{proof}

\begin{cor}\label{C:irr-fcJ}
Let $\ov J_X$ be a fine compactified Jacobian of a connected nodal curve $X$. 
\begin{enumerate}[(i)]
\item  The smooth locus of $\ov J_X$ is isomorphic to a disjoint union of at least $c(\Gamma_X)$ copies of the generalized Jacobian $\PIC^{\un 0}(X)$. 
\item The number of irreducible components of  $\ov J_X$ of $X$ is greater than or equal to $c(\Gamma_X)$. 
\end{enumerate}
\end{cor}
\begin{proof}
This follows from Proposition \ref{P:sm-irr} and Theorem \ref{T:bound-orb} for $G=\Gamma_X$, and noticing that  $\TF(\Gamma_X,D)$ is a torsor for $\PIC^{\un 0}(X)$. 
\end{proof}

We now introduce some special classes of fine compactified Jacobians. The first types of fine compactified Jacobians come from Definition \ref{D:opTF}.

\begin{defi}\label{D:spec-fcJ}
Let $\ov J_X^d$ be a fine compactified Jacobian of $X$ of degree $d$. Then we say that $\ov J_X^d$ is of  \emph{sN-type} (resp. of \emph{numerical sN-type}, resp. of \emph{N-type}, resp. of \emph{numerical N-type})  if $\cP(\ov J_X^d)$ has the same property (as in Definition \ref{D:upper-sN}).
\end{defi}

Note that, we have the following implications
$$
\xymatrix{
\left\{\text{Fine compactified Jacobians of sN-type}\right\} \ar@{=>}[r] \ar@{=>}[d]  & \left\{\text{Fine compactified Jacobians of N-type}\right\} \ar@{=>}[d] \\
\left\{\text{Fine comp. Jac. of numerical sN-type}\right\} \ar@{=>}[r] & \left\{\text{Fine comp. Jac. of numerical N-type}\right\}.
}
$$
Moreover, it follows from Proposition \ref{P:sm-irr} that $\ov J_X^d$ if of numerically N-type if and only if it has $c(\Gamma_X)$ irreducible components. 

A few words about the terminology: the term N-type comes from its relation to the property of being of N\'eron-type (see Definition \ref{D:Neron} and Proposition \ref{P:Neron-N}), the term sN-type means \emph{strongly} N-type, the term \emph{numerical} N-type (resp. sN-type) refers to the fact that it is a numerical version of the property of being  N-type (resp. sN-type). Note also that a fine compactified Jacobian 
is of numerical sN-type (resp. numerical N-type) iff equality holds in Theorem \ref{T:bound-orb} (resp. if equality holds in Theorem \ref{T:bound-orb} for $G=\Gamma_X$).



\vspace{0.1cm}

The other special classes of fine compactified Jacobians of $X$ have to do with extension properties to smoothing of $X$, either in a one-parameter or in a semi-universal family, as we recall in the following

\begin{defi}\label{D:smoothX}
Let $X$ be a nodal curve. 
\begin{enumerate}
\item \label{D:smoothX1} A \emph{$1$-parameter smoothing}  of $X$ is a flat and proper morphism $\pi:\X\to \Delta=\Spec R$, where $R$ is a discrete valuation ring (=DVR), together with a fixed isomorphism $\X_0\cong X$, where $\X_0$ is the central fiber of $\pi$, and whose general fiber $\X_{\eta}$ is smooth.

Such a $1$-parameter family is \emph{regular} if the total space $\X$ is a regular. 
 \item \label{D:smoothX2} A \emph{semiuniversal effective deformation} of $X$ is a morphism $\Pi: \bX\to \Spec R_X$, where $R_X$ is Noetherian complete local $k$-algebra, together with a fixed isomorphism $\bX_0\cong X$, where $\bX_0$ is the central fiber of $\Pi$, and such that the induced morphism 
 $$
 h_{R_X}: \Hom(R_X,-)\to \Def_X
 $$ 
 of functors of Artin $k$-algebras is formally smooth and an isomorphism on tangent spaces.

\end{enumerate}
\end{defi}
Note that a semiuniversal effective deformation of $X$ exists and is unique (for any projective reduced curve $X$), and also that $R_X\cong k[[T_1,\ldots, T_N]]$ (for any reduced, projective curve with locally complete singularities $X$); see e.g. \cite[\S 4.3]{MRV} and the references therein.

\begin{defi}\label{D:type-fcJ}
Let $\ov J_X^d$ be a fine compactified Jacobian of $X$ of degree $d$. We say that $\ov J_X^d$ is:
\begin{enumerate}
\item \label{D:type-fcJ1} \emph{weakly smoothable} if there exists a one-parameter  smoothing $\pi:\X\to \Delta$ of $X$ and an open substack 
 $\ov J_{\X/\Delta}^d\subseteq \TF^d_{\X/\Delta}$ which is proper over $\Delta$ and such that $(\ov J_{\X/\Delta}^d)_0=\ov J_X^d$. 
\item \label{D:type-fcJ2} \emph{smoothable} if for any  one-parameter smoothing $\pi:\X\to \Delta$ of $X$ there exists an  open substack 
 $\ov J_{\X/\Delta}^d\subseteq \TF^d_{\X/\Delta}$ which is proper over $\Delta$ and such that $(\ov J_{\X/\Delta}^d)_0=\ov J_X^d$.
\item \label{D:type-fcJ3} \emph{universally smoothable} if there exists an open substack $\ov J_{\bX/\Spec(R_X)}^d$ of $\TF^d_{\bX/\Spec(R_X)}$, which is proper and flat over $\Spec(R_X)$, and such that  $(\ov J_{\bX}^d)_0\cong \ov J_X^d$. 
\end{enumerate}
\end{defi}

The above definition of  smoothable fine compactified Jacobians differs from the definition of smoothable fine compactified Jacobians in  \cite[Def. 2.4]{PT1} and \cite[Def. 3.1]{PT2}, where it is required that there exists a \emph{regular} one-parameter smoothing satisfying the conclusions of \eqref{D:type-fcJ1}. Hence, the property of being a smoothable fine compactified Jacobians in the sense of loc. cit. sits naturally in between 
our properties of being smoothable and weakly smoothable.
Anyhow, Pagani-Tommasi have shown that the conditions of being smoothable and weakly smoothable coincide (see Theorem \ref{T:PT-thm} below), so that the above change of terminology will not lead to any potential confusion.

\begin{remark}\label{R:smooth}
The stacks $\ov J_{\X/\Delta}^d$ appearing in \eqref{D:type-fcJ1} or \eqref{D:type-fcJ2} (resp. the stack $\ov J_{\bX/\Spec(R_X)}^d$ appearing in \eqref{D:type-fcJ3}) of the above Definition \ref{D:type-fcJ} satisfy the following properties:
\begin{enumerate}[(i)]
\item  $\ov J_{\X/\Delta}^d$  (resp. $\ov J_{\bX/\Spec(R_X)}^d$) is open in $\Simp_{\X/\Delta}^d$ (resp. $\Simp^d_{\bX/\Delta}$), because its central fiber $\ov J_X^d$ is open in $\Simp_X^d$ and the property of being simple is open. In particular, $\ov J_{\X/\Delta}^d$  (resp. $\ov J_{\bX/\Spec(R_X)}^d$) is a scheme (using \cite[Thm. B]{est1}).
\item $\ov J_{\X/\Delta}^d\to \Delta$ (resp. $\ov J_{\bX/\Spec(R_X)}^d\to \Spec(R_X)$) is flat because it is open in $\Simp_{\X/\Delta}^d$ (resp. $\Simp^d_{\bX/\Delta}$), which is flat over $\Delta$ (resp. $\Spec R_X$) since it is has equidimensional fibers over a regular base.
\item The fibers of $\ov J_{\X/\Delta}^d\to \Delta$ (resp. $\ov J_{\bX/\Spec(R_X)}^d\to \Spec(R_X)$) are  geometrically connected and reduced, 
since  the central fiber $\ov J_X^d$ is connected and reduced and the morphism is flat and proper. 
\item  Since the generic fiber of $\X \to \Delta$ (resp. of $\bX\to \Spec(R_X)$) is smooth, the generic fiber of $\ov J_{\X/\Delta}^d\to \Delta$ (resp. $\ov J_{\bX/\Spec(R_X)}^d\to \Spec(R_X)$)  is 
equal to, respectively:
$$
(\ov J_{\X/\Delta}^d)_{\eta}=\PIC^d_{\X_\eta} \text{ and } (\ov J_{\bX/\Spec(R_X)}^d)_{\eta}=\PIC^d_{\bX_\eta}.
$$
In particular, we have that 
$$
\ov J_{\X/\Delta}^d=\ov J_X^d\coprod \PIC_{\X_{\eta}}^d.
$$
\item The scheme $\ov J_{\bX/\Spec(R_X)}^d$ is regular (see \cite[Thm. 4.5]{MRV}).
\end{enumerate}
\end{remark}

Finally, following Caporaso \cite{capNtype} (see also \cite{capneron} and  \cite{capsurvey}), we introduce a class of fine compactified Jacobians that interacts well with the N\'eron model of the Jacobian of a smoothing.

\begin{defi}(Caporaso)\label{D:Neron}
A degree-$d$ fine compactified Jacobian $\ov J_X^d$ of $X$ is said to be of \emph{N\'eron-type} if for every one-parameter regular smoothing $\pi:\X\to \Delta$ of $X$
\begin{equation}\label{E:Neron1}
\ov J_{\X/\Delta}^d:=\ov J_X^d\coprod \PIC_{\X_{\eta}}^d\to \Delta \quad  \text{ is proper } 
\end{equation}
and the canonical $\Delta$-map (induced by the identity on the generic fiber)
\begin{equation}\label{E:Neron2}
\Phi:(\ov J_{X/\Delta}^d)_{\sm} \to N(\PIC_{\X_{\eta}}^d) \quad \text{ is an isomorphism,}
\end{equation}
where $ (\ov J_{X/\Delta}^d)_{\sm}$ is the relative smooth locus of $\ov J_{\X/\Delta}^d$ and $N(\PIC_{\X_{\eta}}^d)$ is the  N\'eron model of of $\PIC_{\X_{\eta}}^d$.   
\end{defi}

In other words, $\ov J_X^d$ is of N\'eron-type if $\ov J_{\X/\Delta}^d$ provides a compactfication of the N\'eron model of  $\PIC_{\X_{\eta}}^d$ for every one-parameter regular smoothing $\pi:\X\to \Delta$ of $X$. 

We now show that the property of being of N\'eron-type is strongly related to the property of being of N-type (which also explain why we adopted the terminology N-type).  

\begin{prop}\label{P:Neron-N}
Let $\ov J_X^d$ be a degree-$d$ fine compactified Jacobian of $X$. Then $\ov J_X^d$ is of N\'eron type if and only if \eqref{E:Neron1} holds and $\ov J_X^d$ if of N-type.
\end{prop}
\begin{proof}
Take a one-parameter regular smoothing $\pi:\X\to \Delta$ of $X$ and consider the scheme $\ov J_{\X/\Delta}^d$ which is proper over $\Delta$. 
Fact \ref{F:Simploc} implies that the relative smooth locus of $\ov J_{\X/\Delta}^d$ is equal to 
$$
(\ov J_{\X/\Delta}^d)_{\sm}=\coprod_{(\Gamma_X,D)\in \cP(\ov J_X^d)}\Pic^D(X) \coprod \PIC_{\X_{\eta}}^d.
$$
Now, the description of $N(\PIC_{\X_{\eta}}^d)$ provided by Raynaud \cite{Ray} shows that the canonical map 
$$
\Phi: (\ov J_{\X/\Delta}^d)_{\sm}=\coprod_{(\Gamma_X,D)\in \cP(\ov J_X^d)}\Pic^D(X) \coprod \PIC_{\X_{\eta}}^d\to N(\PIC_{\X_{\eta}}^d),
$$
is an isomorphism if and only if the composition 
$$
\cP(\ov J_X^d)(\Gamma_X)=\{D\in \Div^d(\Gamma_X)\: : (\Gamma_X,D)\in \cP(\ov J_X^d)\}\subset \Div^d(\Gamma_X)\twoheadrightarrow \Pic^d(\Gamma_X) \text{ is a bijection},
$$
or, equivalently, if and only if $\ov J_X^d$ is of N-type.  
\end{proof}


The relation between some of the above notions was clarified in \cite{PT2}.

\begin{thm}\label{T:PT-thm}(Pagani-Tommasi \cite{PT2})
Let $U$ be an open subset of $\Simp_X^d$. Then the following conditions for $U$ are equivalent:
\begin{enumerate}
\item \label{T:PT-thm1} $U$ is of sN-type;
\item \label{T:PT-thm2} $U$ is a smoothable fine compactified Jacobian of $X$;
\item \label{T:PT-thm3} $U$ is a weakly smoothable fine compactified Jacobian of $X$.
\end{enumerate} 
\end{thm}
\begin{proof}
The implication $\eqref{T:PT-thm1} \Rightarrow \eqref{T:PT-thm2}$ is proved in \cite[Lemma 6.5]{PT2}.

The implication $\eqref{T:PT-thm2} \Rightarrow \eqref{T:PT-thm3}$ is obvious.

The implication $\eqref{T:PT-thm3} \Rightarrow \eqref{T:PT-thm1}$ is proved in \cite[Prop. 7.12]{PT2} under the stronger assumption that the one-parameter smoothing appearing in Definition \ref{D:type-fcJ}\eqref{D:type-fcJ1} is regular. A proof in the general case was communicated to us by the authors of loc. cit. and it goes as follows\footnote{We thank the authors of \cite{PT2} for communicating us the following proof and giving us the permission to include it in this paper.}. 

Suppose that $U=\ov J_X^d$ is a weakly smoothable fine compactified Jacobian of $X$. We have to show that the upper subset $\cP:=\cP(\ov J_X^d)$ is of sN-type as in Definition \ref{D:upper-sN}. 

First of all, Theorem \ref{T:bound-orb} implies that 
\begin{equation}\label{E:card-P>}
|\cP(G)|\geq c(G) \text{ for any connected spanning subgraph } G \text{ of } \Gamma_X.
\end{equation}

\un{Claim}: We have that  $|\cP(\Gamma_X)|=c(\Gamma_X)$.

Indeed, by assumption (and Remark \ref{R:smooth}), there exists a one-parameter  smoothing $\pi:\X\to \Delta=\Spec R$ of $X$ such that  $\ov J_{\X/\Delta}^d:=\ov J_X^d\coprod \Pic_{\X_\eta}^d \subseteq \TF^d_{\X/\Delta}$ is proper over $\Delta$. For any edge $e$ of $\Gamma_X$, a local equation for the surface $\X$ at the node $n_e$ of $X$ corresponding to the edge $e$ is equal to $xy=\pi^{m(e)+1}$ for some unique $m(e)\in \bN$, where $\pi$ is a fixed uniformizer of $R$. This defines a function $m:E(\Gamma_X)\to \bN$. After blowing up $m(e)$ times the surface $\X$ at the point $n_e$ for each $e\in E(\Gamma_X)$, we get a regular surface $\pi^m:\X^m\to \Delta$ endowed with a $\Delta$-morphism $f:\X^m\to \X$ which is an isomorphism on the generic fiber and it is the contraction of a chain of $m(e)$ rational curves over each point $n_e$. In particular, the dual graph $\Gamma_X^m:=\Gamma_{\X^m_o}$ is obtained from  $\Gamma_X$ by inserting $m(e)$ new vertices (called exceptional vertices) inside each edge $e\in E(\Gamma_X)$.

Now, any divisor $D\in \cP(G)$ for some connected spanning subgraph $G$ of $\Gamma_X$ can be lifted to a collection of degree-$d$ divisor on $\Gamma_X^m$ by summing to it one (and only one)  exceptional vertex for each edge $e\in E(\Gamma_X)\setminus E(G)$, as in \cite[Eq. (5.5)]{PT2}. The collection of all such divisors on $\Gamma_X^m$ is denoted by $\cP^m$ and it come with a map $\phi:\cP^m\to \cP$ that sends a divisor in $\cP^m$ to the unique element of $\cP$ of which it is a lift. By construction, the set $\cP^m$ has cardinality 
\begin{equation}\label{E:card-Pm}
|\cP^m|=\sum_{G\in \SS(\Gamma_X)} \left(\prod_{e\in E(\Gamma_X)\setminus E(G)} m(e)\right)|\cP(G)|.
\end{equation}
 On the other hand, the complexity of $\Gamma_X^m$ is given by (see \cite[Lemma 5.3]{PT2})
\begin{equation}\label{E:compl-Gm}
c(\Gamma_X^m)=\sum_{G\in \SS(\Gamma_X)} \left(\prod_{e\in E(\Gamma_X)\setminus E(G)} m(e)\right)c(G).
\end{equation}
By comparing \eqref{E:card-Pm} and \eqref{E:compl-Gm}, and using \eqref{E:card-P>}, we get that $ |\cP^m|\geq c(\Gamma_X^m)$ with strict inequality  if $|\cP(\Gamma_X)|>c(\Gamma_X)$.

Hence, in order to conclude the proof of the Claim, we assume that $ |\cP^m|> c(\Gamma_X^m)$ and we show that it leads to a contradiction. Because of the assumption $ |\cP^m|> c(\Gamma_X^m)$, there exist two distinct divisors $D_1,D_2\in \cP^m\subseteq \Div^d(\Gamma_X^m)$ mapping to  the same element of $\Pic^d(\Gamma_X^m)$. This is equivalent to say that $D_2=D_1+\Delta(E)$ for some divisor $E\in \Div^0(\Gamma_X^m)$, where $\Delta:=\Delta_{\Gamma_X^m}$ is the Laplacian of $\Gamma_X^m$ as in \eqref{Lapla}. Pick now a line bundle $\L^1$ on the family $\X^m/\Delta$ such that  $\un\deg(\L^1_o)=D_1$ and another line bundle $\E$ on $\X^m/\Delta$ such that $\E$ is trivial on the generic fiber and $\un\deg(\E_o)=\Delta(E)$ (this latter line bundle exists by \cite[Prop. (8.1.2)]{Ray}). If we set $\L^2=\L^1\otimes \E$,  then $\L^1$ and $\L^2$ are two line bundles on the family $\X^m/\Delta$ that are isomorphic on the generic fiber and such that $\un \deg(\L^i)=D_i$ for $i=1,2$. By construction, the pushforwards $\I^1:=f_*(\L_1)$ and $\I^2:=f_*(\L_2)$ are two relative rank-$1$ torsion-free sheaves on $\X/\Delta$ that are isomorphic on the generic fiber and such that their restriction to the special fiber have combinatorial type equal to, respectively,  $\phi(D_1)$ and $\phi(D_2)$. 
In particular, $\I^1,\I^2\in \ov J_{\X/\Delta}^d(\Delta)$ since $\phi(D_1),\phi(D_2)\in \cP=\cP(\ov J_X^d)$. The two elements $\phi(D_1)$ and $\phi(D_2)$ of $\cP$ cannot be equal by \cite[Lemma 3.10(5)]{AAPT}, and hence also the sheaves $\I^1$ and $\I^2$ cannot be equal. This contradicts the separatedness of $\ov J_{\X/\Delta}^d\to \Delta$, and the Claim is proved.

\un{End of the proof}: by \eqref{E:card-P>}, the upper subset $\cP$ contains a BD-set, call it $\BD_I$. By the Claim and Theorem \ref{T:BD-Ntype}\eqref{T:BD-Ntype1}, we have that $\BD_I$ is of numerical N-type, i.e. 
$$|\BD_I(\Gamma_X)|=c(\Gamma_X).$$  
Then Theorem \ref{T:main-comb}\eqref{T:main-comb2} implies that $\BD_I$ is an upper subset of sN-type. By applying the implication $\eqref{T:PT-thm1} \Rightarrow \eqref{T:PT-thm2}$ (which has been already proved), we deduce that $U(\BD_I)$ is a (smoothable) compactified Jacobian of $X$ which is clearly an open substack of $U(\cP)=\ov J_X^d$. Since $U(\BD_I)$ and $U(\cP)=\ov J_X^d$ are proper, we deduce that the inclusion $U(\BD_I)\hookrightarrow \ov J_X^d$ is also closed and hence  that $U(\BD_I)= \ov J_X^d$ by connectedness of $\ov J_X^d$. Therefore, $\cP=\BD_I$ is of sN-type, and we are done. 
\end{proof}

\begin{cor}\label{C:PT-cor}
An open subset of sN-type is a fine compactified Jacobian of sN-type. 
\end{cor}

There is a useful equivalence relation on fine compactified Jacobians, given by translations by line bundles. 

\begin{defi}\label{D:eq-fcJ}(\cite[Def. 3.1]{MRV})
Let $X$ a connected nodal curve. 
\begin{enumerate}
\item If $\ov J_X$ is a fine compactified Jacobian of $X$ and $L$ is any line bundle on $X$, then the \emph{translation} of $\ov J_X$ by $L$ is the following fine compactified Jacobian of $X$
$$
\ov J_X\otimes L:=\{I\otimes L\: :   I\in \ov J_X\}.
$$
\item Two fine compactified Jacobians $\ov J_X$ and $\ov J'_X$ are said to be \emph{equivalent by translation} if there exists a line bundle $L$ on $X$ such that 
$\ov J_X'=\ov J_X\otimes L$. 
\end{enumerate}
\end{defi}
Note that the equivalence by translation is an equivalence relation among the fine compactified Jacobians of a given connected nodal curve, and that this equivalence relation preserves the property of being of (resp. numerical) sN-type or of (resp. numerical) N-type, as well as the property of being (resp. weakly or universally) smoothable.

\subsection{Fine V-compactified Jacobians}\label{sec:VfcJ}

In this subsection we introduce a new class of fine compactified Jacobians, namely the fine V-compactified Jacobians and we study their properties.

\begin{remark}\label{R:Vstab-cur}
Let $X$ be a nodal curve with dual graph $\Gamma_X$. Given a general V-stability condition $\frak n$ on $\Gamma_X$ of degree $d=|\frak n|\in \Z$ as in Definition \ref{D:Vstab}, we get a collection of integers 
$$
\frak n=\{\frak n_Y:=\frak n_{V(Y)} \: : \: Y\subset X\:  \text{ is a biconnected and non-trivial subcurve}\}
$$
satisfying the following properties:
\begin{enumerate}
\item for any $Y\subset X$ biconnected and non-trivial, we have 
\begin{equation}\label{E:Vstab-cur1}
\frak n_Y+\frak n_{Y^c}=d+1-|Y\cap Y^c|;
\end{equation}
\item  for any disjoint $Y_1,Y_2\subset X$ without common irreducible components such that $Y_1,Y_2$ and $Y_1\cup Y_2$ are biconnected and non-trivial, we have that 
\begin{equation}\label{E:Vstab-cur2}
-1\leq \frak n_{Y_1\cup Y_2}-\frak n_{Y_1}-\frak n_{Y_2}-|Y_1\cap Y_2|\leq 0.
\end{equation}
\end{enumerate}
\end{remark}

A general V-stability condition on $\Gamma_X$ of degree $d$ determines an open subset of $\Simp_X^d$, as in the following

\begin{lemdef}\label{D:V-open}
Let $\frak n$ be a degree-$d$ general V-stability condition on $\Gamma_X$. Then  \emph{V-open subset} associated to $\frak n$ is the open subset of $\Simp_X^d$ given by
\begin{equation}\label{E:Vopen2}
\ov J_X(\frak n):=U(\cP_{\frak n})=\coprod_{(G,D)\in \cP_{\frak n}} \TF_X(G,D).
\end{equation}
Explicitly, we have that  
\begin{equation}\label{E:Vopen1}
\ov J_X(\frak n):=\{I \in \TF_X^d\: : \:  \deg_Y(I)\geq \frak n_Y\}.
\end{equation}
\end{lemdef} 
\begin{proof}
The equivalence between \eqref{E:Vopen1} and \eqref{E:Vopen2} follows from Definition \ref{D:stabV} and equations \eqref{E:deg-sub} and \eqref{E:orbit1}.
The fact that $\ov J_X(\frak n)$ is open and contained in $\Simp_X^d$ follows from the fact that $\cP_{\frak n}$ is an upper subset contained in $\bO_{\con}^d(\Gamma_X)$. 
The last assertion is obvious. 
\end{proof}

\begin{remark}\label{R:upp-ineq}
Using Remark \ref{sym-inequ}, it follows that if $I\in \ov J_X(\frak n)$ then 
\begin{equation}\label{E:upp-ineq}
\deg_Y(I)\leq \frak n_Y-1-|Y\cap Y^c\cap \NF(I)^c|,
\end{equation}
for any non-trivial and biconnected subcurve $Y\subset X$. 
\end{remark}

\begin{example}\label{Ex:clas-fcJ}
Let $\phi$ be a  \emph{general numerical polarization} on $\Gamma_X$ of degree $d\in \Z$, as in Example \ref{Ex:phi}. 
Then the  V-open subset associated to the general V-stability condition $\frak n(\phi)$ determined by $\phi$ (as in \eqref{E:V-q}) is the 
\textbf{fine classical compactified Jacobian} associated to $\phi$:  
\begin{equation}\label{E:fcJphi}
\ov J_X(\frak n(\phi)):=\ov J_X(\phi)=\left\{I \in \TF_X^d\: : \:  \deg_Y(I)\geq \phi_Y-\frac{|Y\cap Y^c|}{2}:= \phi_{V(Y)}-\frac{|Y\cap Y^c|}{2}\right\}.
\end{equation}
Note that if $\phi$ and $\phi'$ are two general numerical polarizations such that $\frak n(\phi)=\frak n(\phi')$, or equivalently if $\phi$ and $\phi'$ belong to the same chamber 
of the hyperplane arrangement \eqref{E:arr-hyper}, then 
 $$\ov J_X(\phi)=\ov J_X(\phi').$$
 The converse of this statement will follow from Corollary \ref{C:unique-n}.
\end{example}

We now characterize the V-open subsets as those open subsets that are of (numerical) sN-type as in Definition \ref{D:opTF}, or as those fine compactified Jacobians 
that are of (numerical) sN-type or (numerical) N-type as in Definition \ref{D:spec-fcJ}.

\begin{thm}\label{T:fcJ-vine}
Let $U$ be an open subset of $\Simp_X^d$. Then the following conditions for $U$ are equivalent:
\begin{enumerate}
\item \label{T:fcJ-vine1} $U$ is a V-open subset;
\item \label{T:fcJ-vine2} $U$ is an open subset of sN-type;
\item \label{T:fcJ-vine3} $U$ is an open subset of numerical sN-type;
\item \label{T:fcJ-vine4} $U$ is a fine compactified Jacobian of sN-type;
\item \label{T:fcJ-vine5} $U$ is a fine compactified Jacobian of numerical sN-type;
\item \label{T:fcJ-vine6} $U$ is a fine compactified Jacobian of N-type;
\item \label{T:fcJ-vine7} $U$ is a fine compactified Jacobian of numerical N-type.
\end{enumerate} 
\end{thm}
It follows from the above Theorem that the V-open subset $\ov J_X(\frak n)$ is a fine compactified Jacobian of sN-type, that we will call the \textbf{fine V-compactified Jacobian} associated to $\frak n$.  By definition, the poset of orbits of $\ov J_X(\frak n)$ is 
$$
\cP(\ov J_X(\frak n))=\cP_{\frak n}. 
$$
The fact that fine classical compactified Jacobian are of sN-type (and hence also of numerical sN-type, of N-type and of numerical N-type) was proved  in  \cite[Thm. 5.1]{MV}. 

\begin{proof}
The equivalences $\eqref{T:fcJ-vine1}\Leftrightarrow \eqref{T:fcJ-vine2} \Leftrightarrow \eqref{T:fcJ-vine3}$ follows from Theorem \ref{T:main-comb}. 

The implication  $\eqref{T:fcJ-vine4}\Rightarrow \eqref{T:fcJ-vine2}$ is obvious and the converse $\eqref{T:fcJ-vine2}\Rightarrow \eqref{T:fcJ-vine4}$ follows from Corollary \ref{C:PT-cor}.

The implications $\eqref{T:fcJ-vine4}\Rightarrow \eqref{T:fcJ-vine5} \Rightarrow \eqref{T:fcJ-vine7}$ and $\eqref{T:fcJ-vine4}\Rightarrow \eqref{T:fcJ-vine6} \Rightarrow \eqref{T:fcJ-vine7}$
are obvious from the Definition \ref{D:spec-fcJ}.

It remains to prove the implication $\eqref{T:fcJ-vine7}\Rightarrow \eqref{T:fcJ-vine1}$. Let $\ov J_X^d$ be a degree-$d$ fine compactified Jacobian of numerical N-type. Theorem \ref{T:bound-orb}
(together with Corollary \ref{C:fcJ-orb1}) implies  that for any spanning tree $T$ of $\Gamma_X$, there exists $D_T\in \Div^{d-|E(T)|^c}(\Gamma_X)$ such that 
$$
\TF(T,D_T)\subset \ov J_X^d.
$$
Since $|E(T)^c|=g(\Gamma_X)$ for any spanning tree $T$, the function 
$$
\begin{aligned}
I:\ST(\Gamma_X) & \longrightarrow \Div^{d-g(\Gamma_X)}(\Gamma_X)\\
T & \mapsto I(T):=D_T
\end{aligned}
$$
is a tree function of degree $d$ on $\Gamma_X$. By Definition  \ref{D:BD} and Corollary \ref{C:fcJ-orb2}, the BD-set $\BD_I$ associated to $I$ satisfies 
\begin{equation}\label{E:incl-pos}
\BD_I\subseteq \cP(\ov J_X^d)\subset \bO_{\con}^d(\Gamma_X).
\end{equation}
This inclusion, together with the hypothesis that $\cP(\ov J_X^d)$ is of numerical N-type and Theorem \ref{T:BD-Ntype}\eqref{T:BD-Ntype1}, implies that $\BD_I$ is of numerical N-type. 
Therefore we can apply Theorem \ref{T:main-comb} in order to get a general V-stability condition $\frak n$ of degree $d$ on $\Gamma_X$   such that $\cP_{\frak n}=\BD_I$. 
The inclusion \eqref{E:incl-pos} of posets implies that we have an inclusion of open subsets of $\TF_X^d$
\begin{equation*}\label{E:inc-Jac}
\ov J_X(\frak n)\subseteq \ov J_X^d.
\end{equation*}
However, since $\ov J_X^d$ is proper and connected by assumption, and $\ov J_X(\frak n)$ is proper by the implication $\eqref{T:fcJ-vine1}\Rightarrow \eqref{T:fcJ-vine4}$ which was shown before, we infer that  
$$\ov J_X(\frak n)=\ov J_X^d,$$
which concludes the proof. 
\end{proof}

We now deduce some consequences from the above Theorem. 

First of all, fine V-compactified Jacobians of $X$ can be characterized as the ones having the smallest number of irreducible components, namely the complexity $c(\Gamma_X)$ of the dual graph $\Gamma_X$. 

\begin{cor}\label{C:sm-VfcJ}
Let $\ov J_X$ be a fine compactified Jacobian. The following conditions are equivalent:
 \begin{enumerate}[(i)]
\item $\ov J_X$ is a fine V-compactified Jacobian;
\item $\ov J_X$ has $c(\Gamma_X)$ irreducible components;
\item the smooth locus of  $\ov J_X$ of $X$ is  isomorphic to $c(\Gamma_X)$ copies of the generalized Jacobian $\PIC^{\un 0}_X$ of $X$.
\end{enumerate}
In particular, any two fine V-compactified Jacobians of $X$ are birational. 
\end{cor}
The fact that  fine classical compactified Jacobians have $c(\Gamma_X)$ irreducible components was shown in \cite[Cor. 14.4]{OS}.
\begin{proof}
This follows  from Proposition \ref{P:sm-irr} using that fine V-compactified Jacobians are exactly the fine compactified Jacobians  of numerical N-type by Theorem \ref{T:fcJ-vine}. 
\end{proof}

Another corollary of the above Theorem is that a fine V-compactified Jacobian comes from a unique general V-stability condition and that the equivalence by translation on general V-stability conditions (as in Definition \ref{D:eq-Vstab}) correspond exactly to the equivalence by translation on their associated fine V-compactified Jacobians (as in Definition \ref{D:eq-fcJ}).

\begin{cor}\label{C:unique-n}
Suppose that $\frak n, \frak n'$ are two general V-stability conditions on $\Gamma_X$  and let $L$ be a line bundle on $X$ of multidegree $D(L)=D$.
Then we have that 
$$
\ov J_X(\frak n)\otimes L=\ov J_X(\frak n') \Leftrightarrow \frak n+D=\frak n'.
$$
In particular, $\ov J_X(\frak n)=\ov J_X(\frak n')$ if and only if $\frak n=\frak n'$. 
\end{cor}
\begin{proof}
Let us first prove the last statement. If $\ov J_X(\frak n)=\ov J_X(\frak n')$ then we have that 
$$\cP_{\frak n}=\cP(\ov J_X(\frak n))=\cP(\ov J_X(\frak n'))=\cP_{\frak n'}.$$
We now conclude that $\frak n=\frak n'$ by Corollary \ref{C:unique-n}.

The general case follows from this special case together with the obvious fact that $\ov J_X(\frak n)\otimes L=\ov J_X(\frak n+D)$. 
\end{proof}
Note that the corresponding fact for general numerical polarizations is false: if $\phi$ and $\phi'$ are general polarizations on $\Gamma_X$, then $\ov J_X(\phi)\otimes L=\ov J_X(\phi')$ if and only if 
$\phi+D$ and $\phi'$ belong to the same chamber of the hyperplane arrangement \eqref{E:arr-hyper}   (which is equivalent to requiring that $\frak n(\phi+D)=\frak n(\phi')$). 

\begin{cor}\label{C:finite-eq}
Let $X$ be a connected nodal curve. Then there are a finite number of equivalence classes  of fine V-compactified Jacobians up to translation. 
\end{cor}
This result was shown for fine classical compactified Jacobian in \cite[Prop. 3.2]{MRV}.
\begin{proof}
This follows from Corollary \ref{C:unique-n} and Corollary \ref{C:fin-msa}. 
\end{proof}

Another consequence of Theorem \ref{T:fcJ-vine} is a formula for the class of any fine V-compactified Jacobian in the Grothendieck ring $K_0(\Var_k)$ of varieties over $k$.   

\begin{cor}\label{C:K-var}
For any general V-stability condition on the dual graph $\Gamma_X$ of a connected nodal curve $X$ over $k$, the class of $\ov J_X(\frak n)$ in $K_0(\Var_k)$ is equal to
$$
[\ov J_X(\frak n)]=c(\Gamma_X)[J^0_{X^{\nu}}]\cdot \mathbb{L}^{g(\Gamma_X)},
$$
where $J^0_{X^{\nu}}$ is the Jacobian of the normalization $X^{\nu}$ of $X$ and $\mathbb{L}=[\mathbb{A}^1]$. 
\end{cor}
This is proved for fine classical compactified Jacobian in \cite[Prop. B.2]{MSV}.
\begin{proof}
This is proved as in \cite[Appendix B]{MSV}, using that $\ov J_X(\frak n)$ is of sN-type.
\end{proof}




Let us now give the example of two classes of nodal curves for which we know a description of all fine compactified Jacobians.

\begin{example}\label{Ex:vine}
Let $X$ be a \emph{vine curve of type $t$}, i.e. $X$ consists of two smooth curves $C_1$ and $C_2$ meeting in $t$ nodes. A general V-stability on $\Gamma_X$ of degree $d$ consists of two integer numbers 
$\frak n=(\frak n_1:=\frak n_{V(C_1)},\frak n_2:=\frak n_{V(C_2)})$ such that $\frak n_1+\frak n_2=d+1-t$. The fine V-compactified Jacobian $\ov J_X(\frak n)$ associated to $\frak n$ is equal to 
\begin{equation}\label{E:fcJ-vine}
\ov J_X(\frak n):=\{I\in \TF_X^d\: : \: \frak n_1\leq \deg_{C_1}(I) \text{ and } \frak n_2\leq \deg_{C_2}(I)\}.
\end{equation}
It turns out that $\ov J_X(\frak n)$ is a fine classical compactified Jacobian equal to 
$$\ov J_X(\frak n)=\ov J_X\left(\frak n_1-\frac{t}{2}-\epsilon,\frak n_2+\frac{t}{2}-1+\epsilon\right) \: \text{ for any } 0<\epsilon <1. $$ 
Moreover, it turns out that any fine compactified Jacobian of $X$ is equal to $\ov J_X(\frak n)$, for a (unique) general V-stability condition $\frak n$, see \cite[Ex. 7.4]{PT2}.
\end{example}

\begin{remark}\label{R:Vtype}
We can now explain why we have baptized our stability conditions (and hence also the new fine compactified Jacobian) of \emph{vine type}, or for short of \emph{V-type}.

Observe that   given a non-trivial biconnected subcurve $Y\subset X$, we can smooth out all the nodes of $X$ except those that belong to the intersection $Y\cap Y^c$ and we get a vine curve of type $t:=|Y\cap Y^c|$.  Conversely, given a partial smoothing of $X$ to a vine curve of type $t$, by taking the closure of the two irreducible components of the vine curve and intersecting with $X$ we get a pair of non-trivial biconnected complementary subcurves $Y$ and $Y^c$ such that $t=|Y\cap Y^c|$. Since all the fine compactified Jacobians of degree $d$ of a vine curve $C_1\cup C_2$ of type $t$ are classified by a pair of natural numbers $(\frak n_{C_1},\frak n_{C_2})$ such that $\frak n_{C_1}+\frak n_{C_2}=d+1-t$ (see Example \ref{Ex:vine}), a general V-stability condition $\frak n$ of degree $d$ on $\Gamma_X$ is the same as giving a degree $d$ fine compactified Jacobian for any partial smoothing of $X$ to a vine curve, subject to a natural compatibility condition (without which $\ov J_X(\frak n)$ would not be a fine compactified Jacobian). 
\end{remark}

\begin{example}\label{Ex:cycle}
Let $X_n$ be a \emph{$n$-cycle curve} of genus $1$, i.e. a curve formed by $n$ smooth rational irreducible components $\{C_1,\ldots,C_n\}$ with $n$ nodes $\{P_1,\ldots, P_n\}$ such 
$P_i=C_i\cap C_{i+1}$ for each $i=1,\ldots, n$ (with the cyclic convention $C_{n+1}=C_1$). 

Pagani-Tommasi have classified in \cite[Sec. 3]{PT2} all the fine compactified Jacobians of $X_n$. In particular, they have showed that any fine compactified Jacobian of $X_n$ is a $n\rho$-cycle curve of genus $0$ and all the $\rho\geq 1$ can occur if $n\geq 3$ (while $\rho$ must be equal to $1$ for $n=2$, in which case $X_2$ is a vine curve of type $2$).

In particular, for the $n$-cycle curve $X_n$ with $n\geq 3$, we deduce that:
\begin{enumerate}[(i)]
\item there are fine compactified Jacobians that are not of numerical N-type (since $c(X_n)=n$), and hence they are not fine V-compactified Jacobians (by Theorem \ref{T:fcJ-vine});
\item there are infinitely many equivalence classes of fine compactified Jacobians up to translation. 
\end{enumerate}
\end{example}

We can now characterize the fine V-compactified Jacobian as those that are (universally or weakly) smoothable as in Definition \ref{D:smoothX}.

\begin{thm}\label{T:V-fcJ}
Let $\ov J_X$ be a fine compactified Jacobian of $X$. Then the following conditions are equivalent:
\begin{enumerate}
\item \label{T:V-fcJ1} $\ov J_X$ is fine V-compactified Jacobian, i.e. $\ov J_X=\ov J_X(\frak n)$ for some general V-stability condition $\frak n$ on $\Gamma_X$;
\item \label{T:V-fcJ2} $\ov J_X$ is universally smoothable.
\item \label{T:V-fcJ3} $\ov J_X$ is  smoothable.
\item \label{T:V-fcJ3bis}  $\ov J_X$ is of N\'eron-type.
\item \label{T:V-fcJ4} $\ov J_X$ is weakly smoothable.
\end{enumerate}
\end{thm}
This result was known for  fine classical compactified Jacobians: the fact that they are  universally smoothable (which then implies that they are smoothable and weakly smoothable, recovering \cite[\S 5]{Ish}) was proved in \cite[Thm. 5.4]{MRV} and the fact that they are of N\'eron-type was proved  in \cite[Thm. 4.1]{MV}.



\begin{proof}
We will prove a chain of implications.  Call $d$ the degree of $\ov J_X$.

$\bullet$ $\eqref{T:V-fcJ1}\Rightarrow \eqref{T:V-fcJ2}$: by assumption we have that $\ov J_X=\ov J_X(\frak n)$ for some general V-stability condition $\frak n$ on $\Gamma_X$ of degree $d$. 

Consider  the semiuniversal effective deformation $\Pi:\bX\to \Spec R_X$ of $X$ as in Definition \ref{D:smoothX}\eqref{D:smoothX2}. Since $\Spec R_X$ is a local scheme with unique closed point $0\in \Spec R_X$, for every point $s\in \Spec R_X$ there exists a unique \'etale specialization $\ov s\rightsquigarrow 0$, where $\ov s$ is a geometric point over $s$. 
This \'etale specialization gives rise to a morphism of dual graphs 
$$f_s:\Gamma_X=\Gamma_{\bX_0}\to \Gamma_{\bX_{\ov s}}:=\Gamma_{\ov s}$$ 
Consider the push-forward $\frak n_s:=(f_s)_*(\frak n)$ of $\frak n$ along $f_s$ (see Lemma-Definition \ref{LD:V-func}), which is a general V-stability condition of degree $d$ on $\Gamma_{\bX_{\ov s}}$. 
We denote by $\ov J_{\bX_s}(\frak n_s)$ the image of $\ov J_{\bX_{\ov s}}(\frak n_s)$ via the map $\Simp^d_{\bX_{\ov s}}\to \Simp^d_{\bX_s}$ and we set
$$
\ov J_{\bX/\Spec R_X}(\frak n):=\bigcup_{s\in \Spec R_X} \ov J_{\bX_s}(\frak n_s).
$$
Fact \ref{F:orbits}\eqref{F:orbits4}  and Propositions \ref{P:stabV-fun} and \ref{P:spec-orb} imply that $\ov J_{\bX/\Spec R_X}(\frak n)$ is an open subset of $\Simp^d_{\bX/\Spec R_X}$  
whose central fiber is clearly $\ov J_X(\frak n)$. We conclude by proving the following

\un{Claim:} $\Psi: \ov J_{\bX/\Spec R_X}(\frak n)\to \Spec R_X$ is proper. 

Indeed, consider the open dense subset $(\Spec R_X)^{\sm}$ parametrizing points $s\in \Spec R_X$ such that $\bX_s$ is smooth and let $\ov J_{\bX/\Spec R_X}(\frak n)^{\sm}:=\Psi^{-1}((\Spec R_X)^{\sm})\subset \ov J_{\bX/\Spec R_X}(\frak n)$, which is also open and dense. In order to show that $\Psi$ is proper, we will use \cite[\href{https://stacks.math.columbia.edu/tag/0CMF}{Tag0CMF}]{ST} (which we can use $\Spec R_X$ is Noetherian and the morphism $\Psi$ is of finite type, and hence also quasi-separated): given a commutative diagram of solid arrows
\begin{equation*}
\xymatrix{
\Spec K \ar@{^{(}->}[d]_{\eta}\ar[r] & \ov J_{\bX/\Spec R_X}(\frak n)^{\sm} \ar@{^{(}->}[r] & \ov J_{\bX/\Spec R_X}(\frak n) \ar[d]^\Psi \\
\Spec A \ar[rr]^h \ar@{-->}[rru]& & \Spec R_X
}
\end{equation*}
where $A$ is a DVR with field of fractions $K$, we need to show that there exists a unique dotted arrow making the diagram commutative. This amount to prove that the base change 
$$\wt \Psi:\ov J_{\bX/\Spec R_X}(\frak n) \times_{\Spec R_X} \Spec A\to \Spec A$$
 is proper. Since $h(\eta)\in (\Spec R_X)^{\sm}$ by assumption, the base change family $\bX\times_{\Spec R_X} \Spec A \to \Spec A$ is a $1$-parameter smoothing of $\bX_{\ov h(0)}$. Hence, we have that 
$$
 \ov J_{\bX/\Spec R_X}(\frak n) \times_{\Spec R_X} \Spec A=\ov J_{\bX_{\ov h(0)}}(\frak n_{\ov h(0)})\coprod \PIC^d_{\bX_{h(\eta)}}.
 $$
Since $\ov J_{\bX_{\ov h(0)}}(\frak n_{\ov h(0)})$ is a fine compactified Jacobian of sN-type by Theorem \ref{T:fcJ-vine}, we deduce that $\wt \Psi$ is proper by Theorem \ref{T:PT-thm}.

$\bullet$ $\eqref{T:V-fcJ2}\Rightarrow \eqref{T:V-fcJ3}$: by assumption,   there exists an open subscheme
$$
\ov J_{\bX/\Spec R_X}\subset \Simp_{\bX/\Spec R_X}^d
$$
which is proper over $\Spec R_X$ and such that $(\ov J_{\bX/\Spec R_X})_0=\ov J_X$. 

Let $\pi:\X\to \Delta$ be a one-parameter smoothing of $X$. Then, by the definition of semiuniversal effective deformation of $X$, there exists a map $h:\Delta \to \Spec R_X$ with respect to which $\X\cong \bX\times_{\Spec R_X} \Delta$. Then, the open subset  
$$\ov J_{\X/\Delta}:=\ov J_{\bX/\Spec R_X}\times_{\Spec R_X} \Delta\subset \TF^d_{\X/\Delta}$$
is proper over $\Delta$ and it satisfies $(\ov J_{\X/\Delta})_0=\ov J_X$. Hence, $\ov J_X$ is a smoothable fine compactified Jacobian.

$\bullet$ $\eqref{T:V-fcJ3}\Rightarrow \eqref{T:V-fcJ3bis}$: if $\ov J_X$ is smoothable, then $\ov J_X$ is of sN-type (and hence of N-type) by Theorem \ref{T:PT-thm}. Now we conclude by Proposition \ref{P:Neron-N}.

$\bullet$ $\eqref{T:V-fcJ3bis}\Rightarrow \eqref{T:V-fcJ4}$: obvious.

$\bullet$ $\eqref{T:V-fcJ4}\Rightarrow \eqref{T:V-fcJ1}$: if $\ov J_X$ is weakly smoothable, then $\ov J_X$ is of sN-type by Theorem \ref{T:PT-thm}. 
Theorem \ref{T:fcJ-vine} then implies that $\ov J_X$ is a fine V-compactified Jacobian. 
\end{proof}

We now derive some consequences from the above Theorem. The first consequence is the fact that fine V-compactified Jacobians are (singular) CY varieties.

\begin{prop}\label{P:trivdual}
Any fine V-compactified Jacobian $\ov J_X^d(\frak n)$ has trivial dualizing sheaf. 
\end{prop}
This was shown for fine classical compactified Jacobians in \cite[Cor. 5.7]{MRV}.
\begin{proof}
This is proved as in \cite[Cor. 5.7]{MRV}, by showing that the relative dualizing sheaf of the universal fine V-compactified Jacobian $\Psi: \ov J_{\bX/\Spec R_X}(\frak n)\to \Spec R_X$ is trivial. 
\end{proof}

Another consequence is the computation of the Leray perverse filtration on the cohomology of a fine V-compactified Jacobian $\ov J_X(\frak n)$ for a nodal curve $X$ over $\C$. This is defined as follows:
consider the universal fine V-compactified Jacobian $\Psi: \ov J_{\bX/\Spec R_X}(\frak n)\to \Spec R_X$ (which exists by Theorem \ref{T:V-fcJ} and it is regular by Remark \ref{R:smooth}) and define the $n$-graded piece of the Leray perverse filtration by 
\begin{equation}\label{E:per-fil}
\Gr_P^n H^*(\ov J_X(\frak n),\Q):=(^{p}R^n\Psi_*\Q)_0,
\end{equation}
where $^{p}R^n\Psi_*\Q$ denotes the n-th component in the perverse decomposition for $R\Psi_*\Q$. 

\begin{prop}\label{P:Lerper}
Let $X$ be a connected nodal curve over $\C$ and let $\frak n$ be a general V-stability condition  on  $\Gamma_X$. Then we have the following equality in the Grothendieck group of Hodge structures:
\begin{equation}\label{E:Lerper}
\sum_n q^n \Gr^n_P H^*(\ov J_X(\frak n), \Q)=\sum_m q^m H^m(J^0_{X^{\nu}},\Q) \cdot \sum_h n_h(\Gamma_X) (q\mathbb{L})^{g(\Gamma_X)-h} ((1-q)(1-q\mathbb{L}))^{h},
\end{equation}
where $J^0_{X^{\nu}}$ is the Jacobian of the normalization $X^{\nu}$ of $X$ and $n_h(\Gamma_X)$ is the number of genus $h$ connected spanning subgraphs of the dual graph $\Gamma_X$ of $X$.  

In particular, $\Gr^n_P H^*(\ov J_X(\frak n), \Q)$ is independent of the chosen general V-stability condition $\frak n$.
\end{prop}
This is shown for fine classical compactified Jacobians in \cite[Thm. 1.1]{MSV} (although in the statement of loc. cit. it is assumed for simplicity that all the irreducible components of $X$ are rational, or equivalently that $J^0_{X^{\nu}}$ is a point). 

\begin{proof}
This follows, as in the proof of \cite[Thm. 1.1]{MSV}, from the equality 
$$
^{p}R^n\Psi_*\Q=IC(\bigwedge^n R^1\Psi_{\sm,*}\Q),
$$
which is proved as in \cite[Thm. 5.12]{MSV}, together with the computation of the weight polynomial of $IC(\bigwedge^n R^1\Psi_{\sm,*}\Q)$ in \cite[Cor. 3.10]{MSV}.
\end{proof}

We now show that that fine V-compactified Jacobians of $X$ arise as special fibers of suitable Mumford models of the Jacobian of a smoothing of $X$, a fact that was shown for fine classical compactified Jacobians by Christ-Payne-Shen \cite{CPS}.

Let us fix the setting, following \cite{CPS} to which we refer for further details. 
Let $K=(K,v)$ be a complete and algebraically closed rank-$1$ valued field, with valuation ring $R$ and residue field $k=\ov k$. Let $\pi:\X\to \Spec R$ be a flat and proper morphism of relative dimension $1$
such that the generic fiber $\X_K$ is smooth and connected and the special fiber $\X_o=X$ is nodal and connected.
For a fixed general V-stability condition $\frak n$ on $X$ of degree $d$, consider the relative V-compactified Jacobian
$$
\ov J_{\X}(\frak n):=\ov J_X(\frak n)\coprod \Pic^d(\X_K)\to \Spec R,
$$
which is proper (and flat) by Theorem \ref{T:V-fcJ}.

Let $\Gamma$ be the metric graph associated to $\pi:\X\to \Spec R$, i.e. $\Gamma=(\Gamma_X,l)$ where $l:E(\Gamma_X)\to \R_{>0}$ is the length function defined as follows: for any $e\in E(\Gamma_X)$, the local equation of $\X$ at the node $n_e$ of $X$ corresponding to $e$ is given by $xy=f_e$ for some $f_e$  in the maximal ideal $\frak m_R$ of $R$, and then we set $l(e):=v(f_e)\in  v(\frak m_R)\subseteq \R_{>0}$. Denote by $J(\Gamma)$ the Jacobian of $\Gamma$, which is the real torus $\frac{H_1(\Gamma_X,\R)}{H_1(\Gamma_X,\Z)}$ of dimension $g(\Gamma_X)$ endowed with the integral affine structure induced by the lattice  $H^1(\Gamma_X,\Z)\subset H^1(\Gamma_X,\R)\cong H_1(\Gamma_X,\R)$, where the last isomorphism is induced by the scalar product $\langle , \rangle_l$  on $C_1(\Gamma_X,\R)$ defined by 
$$
\langle e,e'\rangle_l=
\begin{cases}
l(e) & \text{ if } e=e',\\
0 & \text{ otherwise.} 
\end{cases}
$$
We denote by $\Pic^d(\Gamma)$  the degree-$d$ Picard group of $\Gamma$, which is a torsor for the Jacobian $J(\Gamma)$.

\begin{prop}\label{P:MumV}
Setting as above. There exists an admissible polytopal decomposition $\Delta_{\frak n}$ of $\Pic^d(\Gamma)$ such that:
\begin{enumerate}
\item \label{P:MumV1} the fine compactified Jacobian  $\ov J_X(\frak n)$ is naturally isomorphic to 
$$
\frac{Y_{\Delta_{\frak n}}\times_{H^1(\Gamma_X,\Z)\otimes \Gm} \PIC_X^{\un 0}}{H_1(\Gamma_X,\Z)},
$$
where $Y_{\Delta_{\frak n}}$ is the stable toric variety associated to the decomposition $\Delta_{\frak n}$, and endowed with the commuting actions of the lattice $H_1(\Gamma_X,\Z)$
and of the torus  $H^1(\Gamma_X,\Z)\otimes \Gm$ which is the affine subgroup scheme  of the semiabelian variety $\PIC_X^{\un 0}$. 
\item \label{P:MumV2} the formal completion of the relative fine compactified Jacobian $\ov J_{\X}(\frak n)$ along the special fiber is naturally isomorphic to the Mumford model associated to $\Delta_{\frak n}$. 
\end{enumerate}
\end{prop}
Note that the decomposition $\Delta_{\frak n}$ depends on the length function $l$ on $\Gamma_X$, but the stable toric variety $Y_{\Delta_{\frak n}}$ is independent of $l$ (and so it depends only on the graph $\Gamma_X$), as it will be clear from the proof of the Theorem. 

The above Theorem is shown for fine classical compactified Jacobians in \cite[Prop. 5.10, Thm. 6.2]{CPS}, where the result is also proved for non-fine classical compactified Jacobians.  Moreover, part \eqref{P:MumV2} can be strengthened for classical compactified Jacobians by showing  that any relative classical compactified Jacobian is isomorphic to the corresponding algebraic Mumford model (see \cite[Thm. 6.2]{CPS}). This stronger version would be true for fine V-compactified Jacobians if one could prove the projectivity of $\ov J_{\X}(\frak n)\to \Delta$ (see Open Questions).

\begin{proof}
We will indicate how to extend the proofs of \cite[Prop. 5.10, Thm. 6.2]{CPS} to our setting. 

First of all, let us define the admissible polytopal decomposition $\Delta_{\frak n}$ following \cite[\S 5.2]{CPS}.  For any $(G,D)\in \bO^d(\Gamma)$ (see Definition \ref{D:posetO}), we denote by $\theta^o_{G,D}$ the subset of $\Pic^d(\Gamma)$ formed by linear equivalence  classes of divisors of the form $D+E$, where $E$ is a divisor of degree $|E(G)^c|$ consisting of the sum of one point in the interior of each edge of $E(G)^c$. 
We then set 
$$
\Delta_{\frak n}:=\{\theta_{G,D}=\ov{\theta_{G,D}^o}\: : \: (G,D)\in \cP_{\frak n}\},
$$
where $\cP_{\frak n}$ is the V-subset associated to $\frak n$ as in  Definition \ref{D:stabV}.

\un{Claim}: $\Delta_{\frak n}$ is an admissible polytopal decomposition of $\Pic^d(\Gamma)$. 

Indeed, the fact that each $\theta_{G,D}$ is an admissible polytope (as in \cite[Def. 5.5]{CPS}) is proved as in \cite[Prop. 5.7]{CPS}. The fact that each face of $\theta_{G,D}$ belongs to $\Delta_{\frak n}$ follows from the fact that the faces of $\theta_{G,D}$ are the polytopes $\theta_{(G',D')}$ where $(G',D')\geq (G,D)$, which belong to $\Delta_{\frak n}$ since $\cP_{\frak n}$ is an upper  set (see Proposition \ref{P:Bn-upper}). The fact that the intersection of two polytopes in $\Delta_{\frak n}$ is either empty or a common face of both, together with the fact that the polytopes in $\Delta_{\frak n}$ cover $\Pic^d(\Gamma)$, follow 
from the disjoint decomposition 
\begin{equation}\label{E:dec-Pic}
\coprod_{(G,D)\in \cP_{\frak n}}  \theta^o_{(G,D)}=\Pic^d(\Gamma),
\end{equation}
 which is a consequence, similarly to \cite[Prop. 4.4]{CPS}, of the properness of $\ov J_{\X}(\frak n)\to \Spec R$ and \cite[Prop. 4.2]{CPS}\footnote{Note that for classical compactified Jacobians, there is an alternative purely combinatorial proof of the decomposition \eqref{E:dec-Pic} in \cite[Prop. 5.7]{CPS}, which however we do not know how to adapt to fine V-compactified Jacobians.}. 
 
Part \eqref{P:MumV1} is deduced as in \cite[Prop. 5.10]{CPS} from the toric geometric description of $\Simp_X$ (and hence of the open subscheme $\ov J_X(\frak n)$), as described in \cite[Thm. 13.2]{OS}. 


Part \eqref{P:MumV2}: the proof of \cite[Thm. 6.2]{CPS} carries over, using that $\ov J_X(\frak n)$ is reduced (by Lemma \ref{L:prop-fcJ}\eqref{L:prop-fcJ1}).
\end{proof}

We next describe, following \cite{MRV2}, the two subgroup schemes
$$
\PIC^o(\ov J_X(\frak n))\subset \PIC^{\tau}(\ov J_X(\frak n))\subset \PIC(\ov J_X(\frak n))
$$
of the Picard scheme $\PIC(\ov J_X(\frak n))$ of $\ov J_X(\frak n)$, parametrizing line bundles on $\ov J_X(\frak n)$ that are, respectively, algebraic or numerically equivalent to the trivial line bundle.



\begin{prop}\label{P:autod}
Let $X$ be a connected nodal curve and let $\frak n$ be a general V-stability condition on  $\Gamma_X$. 
\begin{enumerate}
\item \label{P:autod1}
There exists a Poincar\'e line bundle $\cP$ on $\PIC^{\un 0}(X) \times \ov J_X(\frak n)$ inducing a group scheme isomorphism 
$$
\begin{aligned}
\beta_{\frak n}: \PIC^{\un 0}(X) & \xrightarrow{\cong} \PIC^o(\ov J_X(\frak n))\\
M & \mapsto \cP_M:=\cP_{\ov J_X(\frak n)\times \{M\}}.
\end{aligned}
$$
\item \label{P:autod2}
We have an equality of group schemes
$$\PIC^o(\ov J_X(\frak n))=\PIC^{\tau}(\ov J_X(\frak n)).$$
\end{enumerate} 
\end{prop}
This is shown for fine classical compactified Jacobians in \cite[Thm. C, D]{MRV2}, where the result is also proved for reduced curves with locally planar singularities. 

\begin{proof}
The proof of \cite[Thm. C, D]{MRV2} carries over since it is using only the following (key) ingredients:

\begin{enumerate}[(a)]
\item The universal fine compactified Jacobian  
$$\ov J_{\bX/\Spec(R_X)}^d(\frak n)\to \Spec(R_X)$$
constructed in the proof of Theorem \ref{T:V-fcJ}  is projective and flat with geometrically connected and reduced fibers over $\Spec(R_X)\cong \Spec k[[T_1,\ldots, T_N]]$ (by Theorem \ref{T:V-fcJ} and Remark \ref{R:smooth}), and the total space is regular (by Remark \ref{R:smooth}). 

\item 
Every geometric fiber over points of $\Spec R_X$ of codimension at most one is a classical fine compactified Jacobian (by Remark \ref{R:irr} and Example \ref{Ex:vine}, using that 
the geometric fibers over codimension at most one points are nodal curves with at most one node), for which the autoduality result is already known.

\item We have that
\begin{equation}\label{E:coom-O}
h^i(\ov J_X(\frak n), \O_{\ov J_X(\frak n)})=\binom{g(X)}{i} \text{ for any } 0\leq i \leq g(X). 
\end{equation}
 This is proved as in \cite[Thm. 4.3]{AN} using the toric geometric description of $\ov J_X(\frak n)$ in Proposition \ref{P:MumV}\eqref{P:MumV1} (see also the proof of \cite[Prop. 7.1]{MRV2}).
\end{enumerate}
\end{proof}

Finally, following \cite{MRV3}, we show  that every two fine V-compactified Jacobians are derived equivalent with respect to an integral transform induced by a natural Poincar\'e sheaf.

\begin{prop}\label{P:der}
Let $X$ be a connected nodal curve $X$ over an algebraically closed field of characteristic either zero or greater than $g(X)$ and let $\frak n$ and $\frak n'$ be two general V-stability conditions on $\Gamma_X$. 
Then there exists a maximal Cohen-Macauly sheaf (called the Poincar\'e sheaf) $\ov \cP$ on $\ov J_X(\frak n)\times \ov J_X(\frak n')$, flat over the two factors, such that the induced integral transform 
$$
\begin{aligned}
\Phi^{\ov \cP}: D^b_{\rm{(q)coh}}(\ov J_X(\frak n))& \longrightarrow D^b_{\rm{(q)coh}}(\ov J_X(\frak n'))\\
\cE^\bullet &\mapsto Rp_{2*}(p_1^*(\cE^{\bullet})\otimes^L \ov \cP)
\end{aligned}
$$ 
is an equivalence of triangulated categories, where $D^b_{\rm{(q)coh}}(-)$ denotes the derived category of quasi-coherent (resp. coherent) sheaves. 
\end{prop}
This is shown for fine classical compactified Jacobians in \cite[Thm. A]{MRV3}, where the result is also proved for reduced curves with locally planar singularities. 

\begin{proof}
This is derived as in \cite[Thm. A]{MRV3} from the isomorphism 
$$
Rp_{13*}\left(p_{12}^*(\ov \cP^{\rm un})^{\vee}\otimes p_{23}^*(\ov \cP^{\rm un}) \right)\xrightarrow{\cong} \O_{\Delta^{\rm{un}}}
$$
in $D^b_{\rm coh}(\ov J_{\bX/\Spec R_X}(\frak n)\times_{\Spec R_X}  \ov J_{\bX/\Spec R_X}(\frak n))$, where $\O_{\Delta^{\rm{un}}}$ is the structure sheaf of the diagonal, 
$p_{ij}$ is the projection of $\ov J_{\bX/\Spec R_X}(\frak n)\times_{\Spec R_X} \ov J_{\bX/\Spec R_X}(\frak n') \times_{\Spec R_X} \ov J_{\bX/\Spec R_X}(\frak n)$ onto the $i$-th and $j$-th factors, and $\ov \cP^{\rm un}$ denotes the universal Poincar\'e sheaf, which is proved as in \cite[Thm. 6.2]{MRV3}.
\end{proof}

\bibliographystyle{apsr}

\end{document}